\documentclass[11pt,reqno]{amsart}

\usepackage{amsmath,amsfonts,amsthm,amssymb,amsxtra}
\usepackage[colorlinks,citecolor=red,hypertexnames=false]{hyperref} 

\usepackage{color}
\usepackage[shortlabels]{enumitem}
\usepackage{bbm} 
\usepackage{stmaryrd}
\usepackage{nicematrix}
\usepackage{mathrsfs} 
\usepackage{float}
\usepackage{graphicx}
\usepackage{todonotes}

\usepackage[margin=1.1in]{geometry}
\usepackage{extarrows}

\usepackage{subcaption}
\usepackage{esint}

\usepackage{tikz}
\usepackage{pgfplots}
\pgfplotsset{compat=1.18} 
\usetikzlibrary{calc}

\newtheorem{theorem}{Theorem}[section]

\newtheorem{proposition}[theorem]{Proposition}
\newtheorem{lemma}[theorem]{Lemma}
\newtheorem{corollary}[theorem]{Corollary}

\theoremstyle{definition}

\theoremstyle{remark}

\newtheorem{remark}{Remark}[section]

\numberwithin{equation}{section}
\allowdisplaybreaks

\renewcommand{\epsilon}{\varepsilon}

\newcommand{\bbP}{{\mathbb{P}}}

\newcommand{\cN}{\mathcal{N}}

\renewcommand{\phi}{\varphi}
\newcommand{\R}{\mathbb{R}}

\newcommand{\T}{\mathbb{T}}

\newcommand{\Z}{\mathbb{Z}}

\newcommand{\E}{\mathbb{E}}

\newcommand{\one}{\mathbbm{1}}

\DeclareFontFamily{U}{mathx}{}
\DeclareFontShape{U}{mathx}{m}{n}{<-> mathx10}{}
\DeclareSymbolFont{mathx}{U}{mathx}{m}{n}
\DeclareMathAccent{\widehat}{0}{mathx}{"70}
\DeclareMathAccent{\widecheck}{0}{mathx}{"71}

  \renewcommand{\mod}{{\rm \, mod\, }}

\newcommand{\braket}[3]{\left \langle #1 \, | #2 |\, #3 \right \rangle }

\newcommand{\norm}[1]{\left\Vert#1\right\Vert}
\newcommand{\abs}[1]{\left\vert#1\right\vert}

\renewcommand{\P}{\mathbb{P}}

\newcommand{\eps}{\varepsilon}

\newcommand{\vertiii}[1]{{\left\vert\kern-0.25ex\left\vert\kern-0.25ex\left\vert #1
    \right\vert\kern-0.25ex\right\vert\kern-0.25ex\right\vert}}

\DeclareMathOperator{\dist}{dist}

\newcommand{\comments}[1]{}

\begin{document}

 \title[Div Grad Doubling]{One-Dimensional Divergence-Type Jacobi Operators Driven by The Doubling Map}

\author[L. Li, W. Wang, S. Zhang ]{ Long Li, Wei Wang, Shiwen Zhang}

\begin{abstract}

We study one-dimensional Jacobi operators of divergence-gradient type on \(\ell^2(\Z_{\ge0})\), with coefficients generated by the doubling map, \(a_n(x)=a(2^n x \mod 1)\). For continuous positive sampling functions, we show that the almost-sure essential spectrum is an interval containing the bottom of the spectrum. Under the assumption \(a\in C^1(\T)\), we prove square-root asymptotics for the integrated density of states and obtain a small-energy expansion for the Lyapunov exponent as \(E\to0^+\). The leading term of the Lyapunov exponent is linear and positive precisely under a nondegeneracy condition on the sampling function. In this case, we prove a large-deviation estimate for the transfer matrices with an explicit rate depending on \(E\). As consequences, we obtain local H\"older continuity of the Lyapunov exponent and the integrated density of states near the bottom of the spectrum, as well as the Anderson localization. Some of these results extend the small-coupling results of Chulaevsky--Spencer and Bourgain--Schlag for Schr\"odinger operators generated by the doubling map to divergence-gradient type operators. The arguments adapt methods from those works, but require uniform control in the small-energy parameter and refined correlation estimates reflecting the divergence-gradient structure.

\end{abstract}

  \maketitle

\tableofcontents

\section{Introduction and Main Results}
In this work, we consider Jacobi operators of divergence-gradient (div-grad) type acting on \(\ell^2(\Z_{\ge 0})\), with Dirichlet boundary condition \(u_{-1}=0\), of the form
\begin{align}\label{eqn:div-grad}
    (Hu)_n=-a_{n+1}u_{n+1}+(a_n+a_{n+1})u_n-a_nu_{n-1}, \quad n\ge 0.
\end{align}
The terminology ``div-grad'' reflects the divergence-form structure of the operator: \eqref{eqn:div-grad} can be rewritten as
\[
    (Hu)_n=-\big[a_{n+1}(u_{n+1}-u_n)-a_n(u_n-u_{n-1})\big].
\]
Thus, \(H\) is the natural discrete analogue of the one-dimensional divergence-form differential operator
\[  \mathcal L\psi(x)=-\frac{d}{dx}\left(a(x)\frac{d\psi}{dx}\right), \quad a:\R\to\R. \]
Operators of this type, together with their higher-dimensional analogues, arise naturally in models of inhomogeneous media, disordered lattices, and elastic or acoustic wave propagation in complex environments; see, for example, \cite{AizenmanMolchanov,FigotinKlein,figotin96loca}. In particular, the spectral properties of one-dimensional div-grad models with random coefficients have been studied in \cite{delyon83,pastur1992book}, while more recently we investigated their quantum transport properties in \cite{li2026upper}. Here, we study a deterministic dynamical version, with coefficients of the form
\begin{align} \label{eqn:an}
  a_n(x)=a(2^n x \mod 1), \quad  x\in \T , \quad n \in \Z_{\ge0} ,  
\end{align}
where \(a:\T\to \R\) is a real-valued sampling function and \(x\mapsto 2x\mod 1\) is the {\it doubling map.} We assume that the sampling function \(a\in L^\infty(\T)\) is strictly positive, and we denote its infimum and supremum by
\begin{align}\label{eqn:sample-bound}
  0<a_-:=  \inf_{x\in \T}a(x)<\sup_{x\in \T}a(x)=a_+<\infty. 
\end{align}
Throughout the paper, the baseline assumption is the uniform ellipticity condition \eqref{eqn:sample-bound}; additional regularity assumptions on \(a\), such as continuity or \(C^1\) regularity, will be imposed in the statements of the corresponding results.

The doubling map preserves the Lebesgue measure on \(\T\) and is ergodic. We denote the Lebesgue measure by \(\P(\cdot)\), and we denote the associated integral by \(\E[\cdot]\):
\begin{align}
\P(A)={\rm Leb}(A),\ A\subset \T, \quad {\rm and} \quad    \E[f] =\int_{\T} f(x) \, dx, \ \ f\in L^1(\T). 
\end{align}

There are several previous results on the spectral properties of discrete one-dimensional Schr\"odinger operators with potentials generated by the doubling map. In particular, Chulaevsky--Spencer \cite{ChulaevskySpencer1995CMP} showed that, away from the spectral edges and the center, the Lyapunov exponent of these Schr\"odinger operators admits an asymptotic expansion in the small coupling constant of the potential. Bourgain--Schlag \cite{BourgainSchlag2000CMP} then established Anderson localization in the same spectral regime. More recently, Damanik--Fillman \cite{DamFil2023CMP} showed that the essential spectrum of these operators is always connected. For important recent progress on Schr\"odinger operators generated by hyperbolic transformations, see Avila-Damanik-Zhang \cite{AvilaDamanikZhang2023Invent,AvilaDamanikZhang2024arXiv}.

In this work, we extend some of these Schr\"odinger-operator results to the div-grad setting. We now describe our main results for the family of ergodic operators \(\{H=H(x)\}_{x\in\T}\) defined by \eqref{eqn:div-grad} and \eqref{eqn:an}. Denote by \(\sigma_{\rm ess}(\cdot)\) and \(\sigma(\cdot)\) the essential and full spectra of an operator, respectively. By ergodicity, there exists a fixed set \(\Sigma_{\rm ess}\subseteq\R\), depending only on \(a\), such that
\begin{equation}
    \Sigma_{\rm ess}
    =
    \sigma_{\rm ess}(H(x)),
    \quad
    \bbP\text{-a.e.\ }x\in\T.
\end{equation}

Our first result identifies the almost-sure essential spectrum.

\begin{theorem}\label{thm:ess-spectrum}
    For any \(a\in C^0(\T,\R)\), the essential spectrum of the associated div-grad model is an interval for almost every realization. In particular, with \(a_+\) as in \eqref{eqn:sample-bound},
    \begin{equation}
        [0,4a(0)]
        \subseteq
        \Sigma_{\rm ess}
        =
        [0,A]
        \subseteq
        [0,4a_+]
    \end{equation}
    for some \(A>0\) that depends only on \(a\).
\end{theorem}

\begin{remark}
    If \(a_+=a(0)\), then the preceding theorem gives \(\Sigma_{\rm ess}=[0,4a_+]\). Since the spectrum is contained in \([0,4a_+]\), it follows that \(H(x)\) has no discrete spectrum for \(\P\)-a.e. \(x\), and
\begin{equation}
\sigma (H(x))=   \Sigma_{\rm ess} =[0,4a_+], \quad \bbP\text{-a.e.\ } x \in \T.
\end{equation}
Thus, in this special case, the almost-sure spectrum is determined explicitly. Exact spectral descriptions of this type are standard in many random models; see, for example, \cite{pastur1992book,aizenman2015random,DamFil2022ESO1}. For deterministic dynamically generated coefficients, explicit spectral descriptions are typically more delicate.
\end{remark}

Having identified the almost-sure essential spectrum of the div-grad model, we next consider its small-energy behavior near the bottom of the spectrum. For these results, the proofs follow the same broad strategy as in \cite{ChulaevskySpencer1995CMP,BourgainSchlag2000CMP} for the corresponding Schr\"odinger setting, but the role of the small parameter is different. In the Schr\"odinger setting, the energy and the coupling constant are independent parameters, and the analysis is carried out at fixed energy as the coupling tends to zero. In the present div-grad setting, by contrast, the sampling function \(a\) is fixed, and the only small parameter is the energy \(E\to0^+\). Thus, estimates obtained in the Schr\"odinger case at fixed energy and small coupling must be revisited here with explicit control as \(E\to0^+\). One could also consider a regime in which \(a_-\) approaches \(a_+\), which would provide a different analogy with the small-coupling Schr\"odinger setting, but that is not the regime studied here. Some rough estimates that are harmless in the Schr\"odinger setting are no longer sufficient in the present problem, because their constants may depend on \(E\) in a way that obstructs the desired small-energy asymptotics. Keeping track of this energy dependence, especially in the phase iteration and in the correlation estimates generated by the doubling map, is one of the main technical points of the present analysis.

We begin the small-energy analysis with the integrated density of states and its behavior near the bottom of the spectrum. Let \(H_N\) denote the restriction of \(H=H(x)\) to the interval \([0,N-1]\), with Dirichlet boundary conditions \(\psi_{-1}=\psi_N=0\). Let
\[
    E_1^N(x)\le E_2^N(x)\le\cdots\le E_N^N(x)
\]
be the eigenvalues of \(H_N\). The \emph{integrated density of states} (IDS) of \(H\) is then defined by
\begin{align}\label{eqn:ids}
    \mathcal N(E)
    \xlongequal{\text{\rm a.s.}}
    \lim_{N\to\infty}
    \frac{1}{N}
    \#\left\{\,j:E_j^N(x)\le E\right\},
\end{align}
where \(\#\) denotes cardinality. By ergodicity, the limit exists almost surely and is independent of the boundary conditions.

We now state the asymptotic behavior of the IDS near the bottom of the spectrum. 
\begin{theorem} 
\label{thm:NE-root}
 Suppose \(a(x)\in C^1(\T)\). 
\begin{align}\label{eqn:NE-root}
    {\mathcal N}(E) = \frac{1}{\pi \sqrt{\kappa}} \sqrt{E} +  {\mathcal O}(E|\log E|)
\end{align}
as \(E \to 0^+\), where \(\kappa=\big(\E[a^{-1}]\big)^{-1}\).
\end{theorem}
\begin{remark}
   Throughout the paper, \(X={\mathcal O}(Y)\) means that \(|X|\le C |Y|\) for some constant \(C>0\), which may depend on \(a\), but is independent of \(x,E\), and \(n\). We also write \(X\sim Y\) if \(X={\mathcal O}(Y)\) and \(Y={\mathcal O}(X)\). 
\end{remark}

We next turn to the Lyapunov exponent, another central object in one-dimensional ergodic spectral theory. The Lyapunov exponent is closely tied to the IDS through the Thouless formula and, in this sense, may be viewed as its natural dual quantity. In one-dimensional spectral theory, transfer-matrix methods give a powerful dynamical-systems approach through \(SL(2,\R)\) cocycles. The Lyapunov exponent characterizes the exponential growth rate of the transfer matrices and plays a central role in describing the growth or decay of solutions to \(H(x)\varphi=E\varphi\). In physical terminology, it is often interpreted as the inverse localization length, and positivity of the Lyapunov exponent is a key indicator of possible localization.

For the div-grad model considered here, the relevant transfer matrices are defined as follows. Let \(T_0^E=Id\), and for \(n\ge 1\),
\begin{align}\label{eqn:Tn-intro}
  T_n^E= A^E_{n-1} \cdots A^E_0  , \quad 
    A_j^E = \frac{1}{a_j} \begin{pmatrix} a_{j+1} + a_j - E & -a_j^2 \\ 1 & 0 \end{pmatrix}, \quad j\ge 0, \quad E\in \R.
\end{align}

The \emph{Lyapunov exponent} is defined by
\begin{align}\label{eqn:Lyp}
    L(E) = \inf_{n> 0} \frac{1}{n} \mathbb{E}\!\left( \log \|T_n^E(x)\| \right)
    \xlongequal{\text{\rm a.e.}\ x} \lim_{n \to \infty} \frac{1}{n}    \log \|T_n^E(x)\|.
\end{align}

The next result gives the asymptotic behavior of the Lyapunov exponent near the bottom of the spectrum.
\begin{theorem}\label{thm:LE-linear-intro}
Suppose \(a(x)\in C^1(\T)\). 
   Denote by \(q(x)=\kappa^{-1}-a^{-1}(x)\) and by 
\begin{align}\label{eqn:ca}
    c_a=\frac{\kappa}{8}\Big(\E\big[q^2(x) \big]+ 2\sum_{s=1}^\infty\E\big[q(x)q(2^sx)\big] \Big) .
\end{align}
Then \(c_a\ge 0\) and 
    \begin{align}\label{eqn:LE-linear-intro}
        L(E)= c_aE+{\mathcal O}(|\log E|^2E^{3/2}) 
    \end{align}
    as \(E\to 0^+\). 
\end{theorem}
\begin{remark}
  The sum appearing in the definition of \(c_a\) is related to the so-called spectral density of the process \(\{q(2^sx)\}\). We will discuss related properties in Section~\ref{sec:lyp}. The Lyapunov exponent is asymptotically linear only when \(c_a>0\). In Section~\ref{sec:lyp}, we will also exhibit sampling functions \(a\) for which \(c_a=0\); in these examples, the linear term in \eqref{eqn:LE-linear-intro} vanishes, so the Lyapunov exponent is sublinear in \(E\) as \(E\to 0^+\).    
\end{remark}

The Lyapunov-exponent asymptotics above concern the averaged finite-scale quantities
\[
L_n(E)=\frac{1}{n}\mathbb{E}\!\left(\log\|T_n^E(x)\|\right)
\]
and their convergence to \(L(E)\). A more delicate problem is to control the pointwise fluctuations of the finite-scale transfer-matrix growth around this asymptotic value. We prove the following large-deviation estimate.

\begin{theorem}\label{thm:ldt-intro}
Retain the assumptions of Theorem~\ref{thm:LE-linear-intro}. Let \(c_a\) be as in \eqref{eqn:ca}. If \(c_a>0\), then, for every \(\varepsilon>0\), there exist constants \(c=c(\varepsilon)>0\) and \(E_0=E_0(\varepsilon)>0\) such that for all \(0<E<E_0\), there exists \(N_0=N_0(E,\varepsilon)\) such that for all \(N\ge N_0\),
\begin{align}
      \P\Big(\, x\in \T\, :\, 
      \Big|\frac{1}{N}\log\|T_N^E(x)\|-c_aE\Big|
      >\varepsilon c_aE
      \Big)
      \le
      \exp\Big(-c\frac{E}{|\log E|^3}N\Big).
\label{eqn:ldt-main}
\end{align}
\end{theorem}
Although \eqref{eqn:ldt-main} is stated with center \(c_aE\), the same estimate can be used with center \(L(E)\) or \(L_N(E)\), after absorbing lower-order deterministic errors. We make this precise in Section~\ref{sec:ldt}; see Remark~\ref{rem:ldt-equivalent-centers}.

The proof is guided by the strategy of Bourgain--Schlag \cite{BourgainSchlag2000CMP} for the Schr\"odinger case. In line with the discussion above, the main additional difficulty in the div-grad setting is that the perturbative analysis must be carried out uniformly in the small-energy parameter. Similar issues already arise in the proof of the Lyapunov-exponent asymptotics, compared with the corresponding analysis of Chulaevsky--Spencer \cite{ChulaevskySpencer1995CMP}. The second technical point is that the Bourgain--Schlag martingale argument was written for the cosine sampling function. The underlying mechanism extends naturally to general \(C^1\) sampling functions, but one has to keep careful track of the conditional-expectation errors that were simpler in the cosine case. We will discuss these points in the proof.

Large-deviation estimates for transfer matrices of the above form have become a powerful tool in the study of one-dimensional Schr\"odinger and Jacobi operators, especially since the work of Bourgain--Goldstein \cite{bourgain2000nonperturbative}; see \cite{bourgain2005greens} for a thorough review. In the present setting, \eqref{eqn:ldt-main} provides quantitative control of finite-scale transfer matrices at most phases, which can then be combined with the Avalanche Principle and finite-volume Green's function estimates.

As a first consequence, we obtain local H\"older regularity of the Lyapunov exponent and the IDS in the small-energy regime.

\begin{theorem}\label{thm:holder}
Let \(c_a\) be as in Theorem~\ref{thm:ldt-intro}. Suppose \(c_a>0\). Then there exists \(E_0>0\) such that both \(L(E)\) and \(\mathcal N(E)\) are locally H\"older continuous on \((0,E_0)\). More precisely, for each \(E\in(0,E_0)\), there exist \(\tau=\tau(E)\in(0,1)\), \(\varepsilon_0=\varepsilon_0(E)>0\), and \(C=C(E)>0\) such that
\begin{align}
    |L(E)-L(E')|+|\mathcal N(E)-\mathcal N(E')|
    \le
    C|E-E'|^\tau
    \label{eqn:holder-local}
\end{align}
whenever \(E'\in(0,E_0)\) and \(|E-E'|<\varepsilon_0\).
\end{theorem}
The proof gives a local H\"older exponent depending explicitly on the energy scale, which degenerates as \(E\to0^+\) and may not be optimal. We discuss this point after the proof; see Remark~\ref{rem:holder-exponent-degenerates}.

The second consequence is Anderson localization at small energies. Here the positivity of \(L(E)\), together with the large-deviation estimate, yields finite-volume Green's function decay; the standard generalized-eigenfunction argument then gives exponential decay of eigenfunctions.
\begin{theorem}\label{thm:AL}
   Retain the assumptions of Theorem~\ref{thm:holder}. Then, for a.e. \(x\in \T\), \(H(x)\) has pure-point spectrum in \([0,E_0]\), and the corresponding eigenfunctions decay exponentially.
\end{theorem}

Extending these results beyond the bottom of the spectrum remains an interesting open problem. In particular, for the div-grad model generated by the doubling map, it is natural to ask whether the Lyapunov exponent is positive for all \(E\neq0\), and whether this positivity leads to Anderson localization throughout the spectrum for a.e. \(x\in\T\). Related questions for dynamically generated Schr\"odinger operators and cocycles, including settings in which the underlying dynamics generalizes the doubling map, have been studied in several recent works; see, for example, \cite{AvilaDamanikZhang2023Invent,AvilaDamanikZhang2024arXiv,BourgainBourgainChang2015JST}. These works provide important context for the broader problem of proving positivity and localization beyond perturbative regimes.

The rest of the paper is organized as follows.
\begin{itemize}
    \item In Section~\ref{sec:ess-spectrum}, we prove the statement on the almost-sure essential spectrum.
    \item In Section~\ref{sec:prufer}, we introduce the singular change of frame and the modified Pr\"ufer variables used throughout the small-energy analysis.
    \item In Sections~\ref{sec:ids} and~\ref{sec:lyp}, we derive the asymptotic formulas for the IDS and the Lyapunov exponent, respectively.
    \item In Section~\ref{sec:ldt}, we prove the large-deviation estimate for transfer matrices.
    \item In Section~\ref{sec:applications}, we explain how the large-deviation estimate implies local H\"older regularity of the IDS and the Lyapunov exponent, as well as Anderson localization at small energies.
    \item The appendices contain the auxiliary mixing estimates, the Avalanche Principle argument for H\"older regularity, and the Green's function estimates used in the proof of Anderson localization.
\end{itemize}

Throughout the paper, constants denoted by \(C\), \(c\), \(C_i\), and \(c_i\) may change from line to line. Unless explicitly stated otherwise, these constants are independent of the phase \(x\), the small energy \(E\), and the scale parameters such as \(n\), \(N\), and \(T\). We will also allow constants such as \(E_0\) and \(N_0\) to be decreased or increased, respectively, finitely many times within a proof, without changing the notations. This convention is used to absorb lower-order errors, polynomial prefactors, and logarithmic factors into the main estimates.


\section{Intervals in the Essential Spectrum}\label{sec:ess-spectrum}

We begin with the proof of Theorem~\ref{thm:ess-spectrum}. The fact that the almost-sure essential spectrum contains an interval starting at \(0\) provides the spectral framework for the small-energy analysis of the IDS, the Lyapunov exponent, and the localization results below. This can be shown with the help of gap-labeling theorem \cite{DamFil2023gaplabel,DamFilZha2023JST}. 

\begin{proof}[Proof of Theorem~\ref{thm:ess-spectrum}]
We use the standard description of the almost-sure essential spectrum for continuous ergodic Jacobi families
\begin{align}\label{eqn:Sigma-ess-union}
    \Sigma_{\rm ess}
    =
    \overline{\bigcup_{x\in\T}\sigma_{\rm ess}(H(x))}.
\end{align}
In particular, spectra associated with periodic points of the doubling map are contained in \(\Sigma_{\rm ess}\).

    In the notation of \cite{DamFil2023CMP}, the div-grad models correspond to the Jacobi matrix with choices of sampling functions on the off-diagonal $-a(x)$ and on the diagonal $a(x) + a(2x)$.
    
    In view of \cite[Corollary~5.3]{DamEmiFil2023JST}, it follows that $\Sigma_{\rm ess}$ is an interval.
    Since $H(x) \geq 0$ and $0 \in \sigma(H(x))$ for all $x$, it follows that the interval is in the indicated form.
    The statement $A \geq 4a(0)$ follows from a trial function argument, since $0$ is a periodic point of period one with respect to the doubling map, so one has 
    \[[0,A]
    = \Sigma_{\rm ess} 
    \supseteq \sigma_{\rm ess}(H(0))
    = [0,4a(0)].\]
    Indeed, one has $\Sigma_{\rm ess} \supseteq \sigma_{\rm ess}(H(x))$ for any periodic point $x$ of the doubling map.
\end{proof}



\section{Modified Pr\"ufer Variables} \label{sec:prufer}
We now introduce the main technical tools for the small-energy analysis of the div-grad cocycles in \eqref{eqn:Tn-intro}. These tools will be used to study the IDS, the Lyapunov exponent, and the large-deviation estimates near the bottom of the spectrum. The analysis is based on two key ingredients:
\begin{itemize}
    \item a change of frame that conjugates the divergence-gradient model to an isotopic harmonic chain;
    \item the modified Pr\"ufer variables.
\end{itemize}

These techniques were first developed in \cite{pastur1992book} for random operators and were later used in \cite{li2026upper} to study quantum transport properties. Several of the key constructions are deterministic, or hold in a more general ergodic setting. In what follows, we adopt the framework and notations of \cite{li2026upper} and reformulate the relevant results for the doubling-map model.  

\subsection{A singular change of frame}
We first recall the cocycle formulation of the eigenvalue equation. A sequence \(u = \{u_n\}_{n\ge -1}\) solves
\begin{align}\label{eqn:hu=eu}
    -a_{n+1} u_{n+1} + (a_{n+1} + a_n) u_n - a_n u_{n-1} = E u_n, \quad n \in \mathbb{Z}_{\ge0},
\end{align}
if and only if, for \(n\ge 1\),
\begin{align}\label{eqn:u-cocycle}
    \begin{pmatrix}
        a_{n} u_{n} \\ u_{n-1}
    \end{pmatrix}
    = A_{n-1}^E
    \begin{pmatrix}
        a_{n-1} u_{n-1} \\ u_{n-2}
    \end{pmatrix}
    = T_n^E
    \begin{pmatrix}
        a_0 u_0 \\ u_{-1}
    \end{pmatrix}, 
\end{align}
where \(A_n^E,T_n^E\) are as in \eqref{eqn:Tn-intro}.

For \(E > 0\), if \(u_n\) satisfies \eqref{eqn:hu=eu}, define \(v_n = a_n (u_n - u_{n-1})\) for \(n \in \mathbb{Z}\). Then \(v_n\) satisfies
\begin{align}\label{eqn:hv=ev}
    -v_{n+1} + 2 v_n - v_{n-1} = \frac{E}{a_n} v_n, \quad n \in \Z_+.
\end{align}

The change of variables between \((u_n, u_{n-1})\) and \((v_n, v_{n-1})\), depending on \(E\), can be expressed as
\begin{align} 
    \begin{pmatrix}
        v_n \\
        v_{n-1}
    \end{pmatrix}
    = W_n
    \begin{pmatrix}
        a_n u_n \\
        u_{n-1}
    \end{pmatrix}, \quad W_n =
    \begin{pmatrix}
        1 & -a_n \\
        1 & E - a_n
    \end{pmatrix}. \notag 
\end{align}

Similarly to \eqref{eqn:u-cocycle}, we can rewrite \eqref{eqn:hv=ev} in transfer-matrix form as
\begin{align}\label{eqn:v-cocycle}
    \begin{pmatrix}
        v_n \\ v_{n-1}
    \end{pmatrix}
    = B^E_{n-1}
    \begin{pmatrix}
        v_{n-1} \\ v_{n-2}
    \end{pmatrix}
    = F_n^E
    \begin{pmatrix}
        v_0 \\ v_{-1}
    \end{pmatrix}, \quad n\ge 1,
\end{align}
where \(F_0^E=Id\) and, for \(n\ge 1\),
\begin{align*}
    F_n^E = B^E_{n-1} \cdots B^E_0, \quad  B^E_j =
    \begin{pmatrix}
        2 - \frac{E}{a_j} & -1 \\
        1 & 0
    \end{pmatrix}.
\end{align*}
 
It follows from \eqref{eqn:u-cocycle} and \eqref{eqn:v-cocycle} that
\begin{align} 
    B_n^E = W_{n+1} A_n^E W_n^{-1}, \quad \text{and} \quad F_n^E = W_n T_n^E W_0^{-1}. \notag 
\end{align}
Although this change of frame is singular as \(E\to 0\), it changes the norm of the transfer matrix only by polynomial factors in \(E,E^{-1}\), and therefore preserves the Lyapunov exponent. The following result can be found in \cite{li2026upper}. 
\begin{proposition}[{\cite[Proposition 3.3]{li2026upper}}]
There exists a constant \(c > 0\), depending only on \(a_+\) in \eqref{eqn:sample-bound}, such that for any \(n \ge 0\) and \(0 < E \le 4 a_+\),
\begin{align}\label{eqn:F-Tn-norm}
    \frac{E}{c} \|F_n^E\| \le \|T_n^E\| \le \frac{c}{E} \|F_n^E\|.
\end{align}
\end{proposition}

Since the comparison in \eqref{eqn:F-Tn-norm} is uniform in \(n\), the Lyapunov exponent can equivalently be computed using the transfer matrices associated with \(v_n\), namely, for each fixed $E>0$
\begin{align} 
     L(E) = \lim_{n \to \infty} \frac{1}{n} \mathbb{E}\!\left( \log \|F_n^E(x)\| \right). \notag 
\end{align}

\subsection{Modified Pr\"ufer variables}

We now recall the Figotin--Pastur phase formalism, also known as the modified Pr\"ufer variables.

Recall that
\begin{align} 
    \kappa= \big(\mathbb{E}[a^{-1}]\big)^{-1}. \notag 
\end{align}
For \(E> 0\), let 
\begin{align}\label{eqn:eta}
    P = \begin{pmatrix}
        1 & -\cos \eta \\
        0 & \sin \eta
    \end{pmatrix}, 
      \quad 
    \eta (E) = \cos^{-1}\!\Big(1 - \frac{1}{2\kappa} E\Big) = \frac{\sqrt{E}}{\sqrt{\kappa}} + {\mathcal O}(E^{3/2}).
\end{align}

For \((v_n, v_{n-1})\) satisfying \eqref{eqn:v-cocycle}, define the modified Pr\"ufer variables \(\big(\rho_n(E), \chi_n(E)\big)\) with respect to the matrix \(P\) by
\begin{align}\label{eqn:modi-prufer}
    \rho_n(E)
    \begin{pmatrix}
        \cos \chi_n(E) \\
        \sin \chi_n(E)
    \end{pmatrix}
    = P
    \begin{pmatrix}
        v_n \\ v_{n-1}
    \end{pmatrix}, \quad n \ge 0,
\end{align}
with initial data
 \begin{align}
    \begin{pmatrix}
        v_0 \\ v_{-1}
    \end{pmatrix}
    = W_0
    \begin{pmatrix}
        a_0 u_0 \\ u_{-1}
    \end{pmatrix}, \quad |a_0 u_0|^2 + |u_{-1}|^2 = 1.  \notag 
 \end{align}

The initial variables are determined by 
\begin{align*}
     \rho_{0}(E)\begin{pmatrix}
        \cos \chi_0(E) \\
        \sin \chi_0(E)
    \end{pmatrix} = r_0\begin{pmatrix}
        \cos \theta_0 - \cos \eta \sin \theta_0 \\ \sin \eta \sin \theta_0
    \end{pmatrix}, \ \ {\rm with}\  \chi_0(E) \in [0,\pi].
\end{align*}

Iterating \eqref{eqn:v-cocycle} together with \eqref{eqn:modi-prufer}, we obtain, for \(n \ge 0\),
\begin{align} 
    \rho_n
    \begin{pmatrix}
        \cos \chi_n(E) \\
        \sin \chi_n(E)
    \end{pmatrix}
    = P
    \begin{pmatrix}
        v_n \\ v_{n-1}
    \end{pmatrix}
    = P F_n^E P^{-1} \rho_0
    \begin{pmatrix}
        \cos \chi_0(E) \\
        \sin \chi_0(E)
    \end{pmatrix}. \notag 
\end{align}

For \(E>0\),
\begin{align}\label{eqn:Fn-rhon}
    \|P F_n^E P^{-1}\|
    = \max_{\chi_0 \in [0,\pi)} \Big\| P F_n^E P^{-1}
    \begin{pmatrix}
        \cos \chi_0 \\
        \sin \chi_0
    \end{pmatrix} \Big\|
    = \frac{\rho_n}{\rho_0}.
\end{align}
Combining this with \eqref{eqn:F-Tn-norm}, we obtain
\begin{align}\label{eqn:Tn-rhon}
    \frac{E^{3/2}}{4c\sqrt{2\kappa}} \frac{\rho_n}{\rho_0}  \le  \|T_n^E\|
    \le   \frac{4c\sqrt{2\kappa}}{E^{3/2}} \frac{\rho_n}{\rho_0}.
\end{align}

The modified Pr\"ufer variables translate spectral quantities into asymptotic information about the phase \(\chi_n\) and amplitude \(\rho_n\). The usefulness of these variables is that the phase encodes the IDS, while the amplitude encodes the Lyapunov exponent. In particular, the following identities hold for any ergodic operator of the form \eqref{eqn:div-grad}. Details can be found in \cite{li2026upper}.

\begin{lemma} \label{lem:ids-lyp}
Let \(\cN(E)\) be the IDS and let \(L(E)\) be the Lyapunov exponent of the operator \(H\) defined by \eqref{eqn:div-grad} and \eqref{eqn:an}. For \(E > 0\),
\begin{align}\label{eqn:IDS-chi}
\cN(E) = \lim_{n \to \infty} \frac{1}{\pi n}\E \big[\chi_n(E)\big],
\end{align}
and
\begin{align}\label{eqn:lyp-rhon}
    L(E) = \lim_{n \to \infty} \frac{1}{n} \E \Big[\log \frac{\rho_n}{\rho_0}\Big].
\end{align}
\end{lemma}

\subsection{Iteration of the Pr\"ufer variables}
We now turn to the iteration of the Pr\"ufer variables, which will be used through Lemma~\ref{lem:ids-lyp} to estimate both the IDS and the Lyapunov exponent. The iteration is driven by a small effective potential term \(Q_n\), whose size is \({\mathcal O}(\sqrt E)\) in the small-energy regime; this term plays the role of the potential in the Schr\"odinger case.

We introduce the following sampling functions on \(\T\):
\begin{align}\label{eqn:Q-samp}
  Q(x)=\frac{E}{\sin \eta(E)}q(x), \qquad   q(x)=\kappa^{-1}-a^{-1}(x).
\end{align}
For \(n\in \Z_{\ge 0}\), we write 
\begin{align} \notag 
 Q_n(x)=Q(2^nx) , \qquad q_n(x)=q(2^nx).
\end{align}

The asymptotic expansion of \(\eta\) in \eqref{eqn:eta} implies that, for any \(n\in \Z_{\ge0}\),
\begin{align}\label{eqn:Q-asymp}
    Q_n=\frac{E}{\sin \eta(E)}(\kappa^{-1}-a_n^{-1})= \sqrt{\kappa E}(\kappa^{-1}-a_n^{-1})+{\mathcal O}(E^{\frac{3}{2}}).
\end{align} 

Since \(a(x)\in C^1\) and \(a(x)>0\), for any \(n\in \Z_{\ge0}\), the functions \(q_n(x)\) and \(Q_n(x)\) belong to \(C^1\). Moreover, there exists a constant \(C>0\), independent of \(n\) and \(E\), such that
\begin{align}\label{eqn:Qn-bound}
 \sup_{x\in \T} |q_n(x)|\le C , \qquad    \sup_{x\in \T} |Q_n(x)|\le C\sqrt E  
\end{align}
and 
\begin{align}\label{eqn:Qn-prime}
 \sup_{x\in \T} |q'_n(x)|\le C  2^n, \qquad    \sup_{x\in \T} |Q'_n(x)|\le C\sqrt E 2^n. 
\end{align}

The following iteration formula is obtained by a direct computation from \eqref{eqn:modi-prufer}.
\begin{proposition}[{\cite[Proposition 3.4]{li2026upper}}]\label{prop:rho-chi-Qn}
For any \(n\ge 0 \) and \(0 < E < 2\kappa\),
\begin{align}\label{eqn:rho-n-iter}
\begin{cases}
    \rho_{n+1} \cos \chi_{n+1} = \rho_n \Big[\cos(\chi_n + \eta) + Q_n \sin( \chi_n+\eta)\Big], \\
    \rho_{n+1} \sin \chi_{n+1} = \rho_n \sin(\chi_n + \eta).
\end{cases}
\end{align}
 
\end{proposition}

It is readily seen from the definition of \(\kappa\) that \(\mathbb{E}Q_n=0\). Unlike in the i.i.d. case, however, \(Q_n\) and \(\chi_n\) are not independent. 
We therefore need to estimate the correlations between \(Q_n\) and \(\chi_n\) generated by the doubling-map dynamics.  

Define
\begin{align} \label{eqn:zeta-mu-def}
    \zeta_j = e^{2i \chi_j}, \quad \mu = e^{2i \eta}, \quad   j\ge 0.
\end{align}
From the recurrence relation \eqref{eqn:rho-n-iter}, we obtain
\[\cot\chi_{j+1}=\cot(\chi_j+\eta)+Q_j,\]
which is equivalent to
\begin{align}\label{eqn:zetaIter}
    \zeta_{j+1} = \mu \zeta_j + \frac{i}{2} Q_j \frac{(\mu \zeta_j - 1)^2}{1 - \frac{i}{2} Q_j (\mu \zeta_j - 1)} .
\end{align}
This iteration formula is the key dynamical relation for the modified Pr\"ufer phase; most of the estimates below are based on its perturbative expansion in the small effective potential \(Q_j\).

Iterating \eqref{eqn:zetaIter} and using \eqref{eqn:Q-asymp}, we get
\begin{align}\label{eqn:zeta-it-diff}
    |\zeta_n-\mu^r\zeta_{n-r}|\leq Cr \sqrt E.
\end{align}
Consequently,
\begin{align} \notag 
    |\zeta^2_n-\mu^{2r}\zeta^2_{n-r}|\le 2|\zeta_n-\mu^r\zeta_{n-r}|  \leq Cr \sqrt E.
\end{align}
Let  
\begin{align}\label{eqn:F-zeta-q}
 F(\zeta,q)=\mu\zeta+\frac{i}{2}\frac{E}{\sin\eta(E)}q\frac{(\mu\zeta-1)^2}{1-\frac{i}{2}\frac{E}{\sin\eta(E)}q(\mu\zeta-1)}.   
\end{align}
There exists a constant \(C>5\sqrt{\kappa}\), together with sufficiently small constants \(E_0>0\) and \(q_0>0\), such that for \(|E|<E_0\) and \(|q|<q_0\),
\begin{align}
 \left|\frac{\partial F}{\partial \zeta}\right|\le 1+C\sqrt E ,\quad \left|\frac{\partial F}{\partial q}\right|\le C\sqrt E, 
\end{align}
which implies
\begin{align}
 |F(\zeta_1,q_1)-F(\zeta_2,q_2)|\leq (1+C\sqrt E)|\zeta_1-\zeta_2|+C\sqrt E|q_1-q_2|. \notag 
\end{align}

If \(\zeta(x),q(x)\in C^1(\T)\), then \(F\big(\zeta(x),q(x)\big)\in C^1(\T)\), and
\begin{align}\label{eqn:F-prime}
    \big|\frac{dF}{dx}\big|\le \left|\frac{\partial F}{\partial \zeta}\right|\cdot |\zeta'| +\left|\frac{\partial F}{\partial q}\right|\cdot |q'|\le (1+C\sqrt E)|\zeta'|+C\sqrt E |q'|.
\end{align}

\begin{lemma}\label{lem:zetan-prime}
    If \(0<E<E_0\), then for all \(n\in \Z_{\ge0}\),
    \begin{align}\label{eqn:zeta-prime-bound}
       \sup_{x\in \T} |\zeta_n'|\le C\sqrt E \, 2^{n}
    \end{align}
\end{lemma}
\begin{remark}
   Chulaevsky--Spencer \cite[Lemma 4]{ChulaevskySpencer1995CMP} and Bourgain--Schlag \cite[Lemma 11.1]{BourgainSchlag2000CMP} proved analogous estimates for Schr\"odinger operators in terms of the small coupling constant of the potential. The proof for the div-grad model follows essentially the same strategy. The key difference is that, in the present setting, the estimates must track their explicit dependence on the small energy \(E\). 
\end{remark}

\begin{proof}
 Combining \eqref{eqn:zetaIter} and \eqref{eqn:F-zeta-q}, we can write \(\zeta_n\) as
    \begin{align*}
        \zeta_n=F(\zeta_{n-1},q_{n-1}).
    \end{align*}
   Applying the derivative estimates \eqref{eqn:Qn-prime} and \eqref{eqn:F-prime} gives
    \begin{align*}
        |\zeta_n'|\le (1+C\sqrt E)|\zeta_{n-1}'|+C\sqrt E |q_{n-1}'|\le (1+C\sqrt E)|\zeta_{n-1}'|+C\sqrt E 2^{n-1}.
    \end{align*}
   Iterating this estimate yields
    \begin{align*}
    |\zeta_n'|\le &  (1+C\sqrt E)^{n-1}|\zeta_1'| +\sum_{j=0}^n (1+C\sqrt E)^j C\sqrt E\, 2^{n-j-1}  \\
    \le & (1+C\sqrt E)^{n-1}|\zeta_1'| + C\sqrt E\, 2^{n-1}\sum_{j=0}^\infty \Big(\frac{1+C\sqrt E}{2}\Big)^j. 
    \end{align*}
Since the initial angle \(\chi_0\in[0,\pi]\) is independent of \(x\) and depends only on \(E\), we have \(\zeta_0'=0\). Moreover,
\begin{align*}
 |\zeta'_1|
 =&
 \Bigg|
 \mu\zeta_0'
 +\frac{i}{2}\frac{d}{dx}
 \Bigg[
 Q_0
 \frac{(\mu\zeta_0-1)^2}
 {1-\frac{i}{2}Q_0(\mu\zeta_0-1)}
 \Bigg]
 \Bigg| \\
 \le&
 \Bigg|
 \frac{Q_0'}{2}
 \frac{(\mu\zeta_0-1)^2}
 {1-\frac{i}{2}Q_0(\mu\zeta_0-1)}
 \Bigg|
 +
 \Bigg|
 \frac{Q_0}{2}
 \frac{d}{dx}
 \Bigg[
 \frac{(\mu\zeta_0-1)^2}
 {1-\frac{i}{2}Q_0(\mu\zeta_0-1)}
 \Bigg]
 \Bigg| \\
 \le&
 C\sqrt E,
 \qquad 0<E<E_0.
\end{align*}
Here we used that the two factors following \(Q_0'\) and \(Q_0\), respectively, are analytic functions of \(Q_0\) for sufficiently small \(E\), and hence are uniformly bounded in \(x\). The order \(\sqrt E\) then follows from the bounds on \(Q_0\) and \(Q_0'\) in \eqref{eqn:Q-asymp}.

 Therefore, if \(E_0\) is chosen so that \(1+C\sqrt E<2\), then
    \begin{align*}
    |\zeta_n'|\le 2^{n-1}C\sqrt E+  C \sqrt E\, 2^{n-1}  \le C\sqrt E\, 2^n. 
    \end{align*}
\end{proof}

The derivative bound in Lemma~\ref{lem:zetan-prime} yields the following exponential decay estimate for correlations generated by the doubling map.
\begin{lemma}
For any integers \(n\ge T\ge0\) and for both \(k=1\) and \(k=2\), 
    \begin{align}\label{eqn:qn-zetan-T}
     \Big|   \E\big[q_n \zeta^k_{n-T}\big]\Big|\le C\sqrt E\, 2^{-T}, \ \ \Big|\E\big[Q_n \zeta^k_{n-T}\big]\Big|\le C  E\,  2^{-T};
    \end{align}
    and 
      \begin{align}\label{eqn:qn-squ-zetan-T}
       \Big| \E\big[q^2_n \zeta^k_{n-T}\big]\Big|\le C\sqrt E\, 2^{-T}, \ \ \Big|\E\big[Q^2_n \zeta^k_{n-T}\big]\Big|\le C  E^{3/2}\,  2^{-T}.
    \end{align}
\end{lemma}

\begin{proof}
Applying \eqref{eqn:zeta-prime-bound} to \(\zeta_{n-T}\), with \(n\ge T\), gives
\[
    |\zeta_{n-T}'|\le C\sqrt E \, 2^{n-T}.
\]
Since \(\E(q)=0\) by the definition of \(\kappa\), and hence \(\E(q_n)=0\), the exponential mixing estimate for the doubling map, Lemma~\ref{lem:doubling-mixing}, in the form \eqref{eqn:doubling-mixing-shifted-zero}, gives
\begin{align*} 
\abs{ \int_{\T} q(2^n x)\,\zeta_{n-T}\,dx }
\le 2^{-  n}\norm{q}_{L^1(\T)}\sup_{x\in \T}|\zeta_{n-T}'(x)|\le C \sqrt E\, 2^{-T}, 
\end{align*}
which proves \eqref{eqn:qn-zetan-T} for \(q_n\) and \(k=1\). The estimate for \(k=2\) follows from \(|(\zeta_{n-T}^2)'|\le 2|\zeta_{n-T}||\zeta'_{n-T}|\).

The remaining estimates are proved in the same way, using \(|Q_n|\le C\sqrt E\).  
\end{proof}


The next estimate controls the orbit average of \(\zeta_n\), after averaging over the underlying sampling space. It plays an important role in estimating lower-order correction terms for both the IDS and the Lyapunov exponent.
\begin{lemma}\label{lem:expZeta}
For \(0<E<E_0\), if  \(N\ge E^{-1}\), then, for both \(k=1\) and \(k=2\),
  \begin{align}
   \frac{1}{N}\sum_{n=0}^{N-1} \E[\zeta^k_n] = O\Big( \sqrt E |\log E|\Big).
 \end{align}
\end{lemma}
\begin{remark}
      Chulaevsky--Spencer \cite[Lemma 1]{ChulaevskySpencer1995CMP} proved an analogous estimate for the Schr\"odinger case in the weak-coupling regime. In their setting, the energy is fixed and bounded away from the spectral edges, so constants depending on the coupling constant can be absorbed into the \({\mathcal O}(\cdot)\) term, and the estimate follows directly from the phase iteration \eqref{eqn:zetaIter}.  In the div-grad setting, the energy \(E\) is itself the small parameter. We therefore need to track the dependence of the constants on \(E\) explicitly; otherwise, such constants could obstruct the desired small-energy order. Establishing this uniform small-energy estimate is one of the technical ingredients of the present work. The proof uses the phase iteration to introduce a time separation, then combines the derivative bound for \(\zeta_n\) with the exponential mixing of the doubling map to obtain the required correlation decay. Choosing the separation scale \(T\sim \log(E^{-1})\) then yields the stated \({\mathcal O}(\sqrt E|\log E|)\) bound.
\end{remark}

\begin{proof}
Using the formula obtained in \cite[Lemma C.3]{li2026upper}, we have
 \begin{align}
   \frac{1}{N}\sum_{n=0}^{N-1}\zeta_n
   =& -\frac{e^{-i\eta}}{4N\sin \eta}\sum_{n=0}^{N-1}(\mu \zeta_n - 1)^2 Q_n + \frac{{\mathcal O}(E)}{\sin \eta} + \frac{\zeta_0 - \zeta_n}{N\sin \eta} \notag \\
   =& -\frac{e^{-i\eta}}{4N\sin \eta}\sum_{n=0}^{N-1}(\mu \zeta_n - 1)^2 Q_n + {\mathcal O}(\sqrt{E}), \label{eqn:sum-zeta}
\end{align}
provided \(N>E^{-1}\), so that \(N\sin \eta>E^{-1/2}\).

For \(0<E<E_0\), choose \(T\sim \log(E^{-1})\). Then, by \eqref{eqn:zeta-it-diff}, for \(n\ge T\),
\begin{align}\label{eqn:zetan-n-T-squ}
    \Big|(\mu \zeta_n - 1)^2-(\mu^{T+1} \zeta_{n-T} - 1)^2\Big|\le 4 \Big| \mu \zeta_n -\mu^{T+1} \zeta_{n-T}  \Big|\le CT\sqrt E.
\end{align}

We split the first sum in \eqref{eqn:sum-zeta} at the scale \(T\):
\begin{align}\label{eqn:sum-zeta-split}
    \frac{e^{-i\eta}}{4N\sin \eta}\sum_{n=0}^{T-1}(\mu \zeta_n - 1)^2 Q_n+\frac{e^{-i\eta}}{4\sin \eta} \frac{1}{N}\sum_{n=T}^{N-1}(\mu \zeta_n - 1)^2 Q_n.
\end{align}
The first term is estimated directly using \eqref{eqn:Q-asymp}:
\begin{align}
    \Big| \frac{e^{-i\eta}}{4N\sin \eta}\sum_{n=0}^{T-1}(\mu \zeta_n - 1)^2 Q_n\Big|\le \frac{CT\sqrt E}{N \sqrt E}=\frac{CT}{N}\le CT\sqrt E, \notag
\end{align}
provided \(N>E^{-1/2}\).

For the second term in \eqref{eqn:sum-zeta-split}, we use \eqref{eqn:zetan-n-T-squ} term by term to obtain
\begin{align}
 &\frac{e^{-i\eta}}{4\sin \eta} \frac{1}{N}\sum_{n=T}^{N-1}(\mu \zeta_n - 1)^2 Q_n \notag \\
=&\frac{e^{-i\eta}}{4\sin \eta} \frac{1}{N}\sum_{n=T}^{N-1}(\mu^{T+1} \zeta_{n-T} - 1)^2 Q_n +\frac{e^{-i\eta}}{4\sin \eta} \frac{1}{N}\sum_{n=T}^{N-1}{\mathcal O}(T\sqrt E) Q_n \notag \\
=&\frac{e^{-i\eta}}{4\sin \eta} \frac{1}{N}\sum_{n=T}^{N-1}(\mu^{T+1} \zeta_{n-T} - 1)^2 Q_n+ {\mathcal O}(T\sqrt E), \notag 
\end{align}
where we used \(|Q_n|/\sin \eta\le C\) in the last line. 

Taking expectations, we therefore get
\begin{align}
    \frac{1}{N}\sum_{n=0}^{N-1} \E[\zeta_n]=\frac{e^{-i\eta}}{4\sin \eta} \frac{1}{N}\sum_{n=T}^{N-1}\E\Big[(\mu^{T+1} \zeta_{n-T} - 1)^2 Q_n\Big]+ {\mathcal O}(T\sqrt E). \notag
\end{align}
By the correlation estimates \eqref{eqn:qn-zetan-T} and \eqref{eqn:qn-squ-zetan-T}, for \(T\le n\le N-1\),
\begin{align}
\Big| \E\big[(\mu^{T+1} \zeta_{n-T} - 1)^2 Q_n\big]\Big| \le \Big| \E\big[ \zeta^2_{n-T}  Q_n\big]\Big| +2\Big| \E\big[  \zeta_{n-T}  Q_n\big]\Big| \le CE2^{-T}.  \notag
\end{align}
Hence,
 \begin{align}
   \Big|\frac{1}{N}\sum_{n=0}^{N-1} \E[\zeta_n]\Big|\le \frac{1}{ \sqrt E}CE2^{-T}+{\mathcal O}(T\sqrt E)=\mathcal{O}\Big(  \sqrt E|\log E|\Big), \notag
 \end{align}
provided \(T\sim \log(E^{-1})\).

The estimate for \(\frac{1}{N}\sum_{n=0}^{N-1} \E[\zeta^2_n]\) follows by the same argument. Indeed, using
\begin{align} 
    \zeta_{n+1}^2 = \mu^2 \zeta_n^2 + i\mu \zeta_n(\mu \zeta_n - 1)^2 Q_n + {\mathcal O}(Q_n^2), \quad 0 \le n \le N-1, \notag
\end{align}
together with
\[
1 - \mu^2 = -2i\mu\sin(2\eta), \quad \text{and} \quad 
    |\zeta^2_n-\mu^{2T}\zeta^2_{n-T}|\le 2 |\zeta_n-\mu^{T}\zeta_{n-T}|\leq C T \sqrt E, \ \ n\ge T, 
\]
one obtains the same bound  
\[
    \frac{1}{N}\sum_{n=0}^{N-1} \E[\zeta^2_n]
    =\mathcal{O}\Big(\sqrt E |\log E|\Big).
\]
  \end{proof}


\section{Asymptotic Formulas for the Integrated Density of States}\label{sec:ids}
We now prove Theorem~\ref{thm:NE-root}. By Lemma~\ref{lem:ids-lyp}, the IDS is determined by the averaged Pr\"ufer phase \(\chi_N\). The leading contribution comes from the rotation term \(N\eta\), with \(\eta(E)\sim \sqrt{E/\kappa}\), while the error is governed by the averaged oscillatory correction involving \(Q_n\).

Let \(\eta\) and \(\kappa\) be as in \eqref{eqn:eta}. The following expansion follows from the phase iteration; see \cite[Eqs.~(C.20)--(C.23)]{li2026upper}:
\begin{align}
  \chi_{n+1} - (\chi_n + \eta) = \frac{1}{2}Q_n \cos 2(\chi_n + \eta) -\frac{1}{2}Q_n + {\mathcal O}(E). \label{eqn:chi-diff-expan}
\end{align}
We decompose the angle variable as
\begin{align}
  \chi_N - N\eta   =  \chi_0  + \sum_{n=0}^{N-1} \big[\chi_{n+1} - (\chi_n + \eta)\big].  \notag
\end{align}
Taking expectations on both sides and using \(\E Q_n=0\), we obtain
\begin{align}
  \E[\chi_N - N\eta]   =  \chi_0  + \frac{1}{2}\sum_{n=0}^{N-1}\E\big[Q_n\cos 2(\chi_n + \eta)\big]+N{\mathcal O}(E).  \notag
\end{align}
Combining this with \eqref{eqn:eta}, \eqref{eqn:IDS-chi}, and the fact that \(\chi_0 \in [0,\pi]\), we get
\begin{align}
    \mathcal N(E) = \lim_{N \to \infty}\frac{\E(\chi_N)}{\pi N} 
    =&\frac{\eta(E)}{\pi}+\lim_{N \to \infty}\frac{\E(\chi_N - N \eta)}{\pi N} \notag \\
    =&  \frac{1}{\pi \sqrt{\kappa}}\sqrt{E} +  \lim_{N \to \infty}\frac{1}{2\pi N}\sum_{n=0}^{N-1}\E\big[Q_n\cos 2(\chi_n + \eta)\big]+ {\mathcal O}(E).
\end{align}

Theorem~\ref{thm:NE-root}, with the \({\mathcal O}(E|\log E|)\) error term, follows from the following lemma.

\begin{lemma}\label{lem:Q-cos}
For \(0<E<E_0\) and \(N\ge E^{-1/2}\),
\begin{align}\label{eqn:Qncos-exp}
 \frac{1}{N}\sum_{n=0}^{N-1}\E\big[Q_n\cos 2(\chi_n + \eta)\big]=  {\mathcal O}(E|\log E|).
\end{align}
In addition, the following two estimates hold:
\begin{align}\label{eqn:Qn-squ-cos-exp}
 \frac{1}{N}\sum_{n=0}^{N-1}\E\big[Q^2_n\cos 2(\chi_n + \eta)\big]=  {\mathcal O}(E^{3/2}|\log E|),
\end{align}
and
\begin{align}\label{eqn:Qn-squ-cos4-exp}
 \frac{1}{N}\sum_{n=0}^{N-1}\E\big[Q^2_n\cos 4(\chi_n + \eta)\big]=  {\mathcal O}(E^{3/2}|\log E|).
\end{align}
\end{lemma}

\begin{remark}\label{rem:Qn-refine}
 We only need \eqref{eqn:Qncos-exp} for the IDS asymptotics. The latter two estimates are proved by the same method and will be used in the next section to study the asymptotic behavior of the Lyapunov exponent.

 In the next section, we obtain a more precise expansion of the averaged term involving \(Q_n\zeta_n\), which is needed for the Lyapunov-exponent asymptotics; see Lemma~\ref{lem:ave-zetanQn}. Such a refinement could also improve the corresponding oscillatory contribution in the IDS calculation to order \({\mathcal O}(E)\), although we will not pursue the sharper IDS remainder in Theorem~\ref{thm:NE-root}.
\end{remark}

\begin{proof}
Fix \(0<E<E_0\), and choose \(T\sim \log(E^{-1})\). By \eqref{eqn:zeta-it-diff}, for \(n\ge T\),
 \begin{align}\label{eqn:zetan-n-T}
    |\zeta_n-\mu^T\zeta_{n-T}|\leq C T \sqrt E . 
\end{align}
For such \(n\), we write
\begin{align}
    Q_n\mu\zeta_n= Q_n\mu^{T+1}\zeta_{n-T}+Q_n(\mu\zeta_n-\mu^{T+1}\zeta_{n-T}). \notag
\end{align}
Taking expectations of the real parts gives
\begin{align}
 \E\big[Q_n\cos 2(\chi_n + \eta)\big]
 =&   \E\big[ \operatorname{Re}(Q_n\mu\zeta_n)\big] \notag\\
 =&\operatorname{Re}\Big[ \mu^{T+1}\E\big[  Q_n\zeta_{n-T}\big]\Big]
   +\operatorname{Re}\E\big[Q_n(\mu\zeta_n-\mu^{T+1}\zeta_{n-T})\big]. \notag 
\end{align}
Combining this identity with \eqref{eqn:qn-zetan-T} and \eqref{eqn:zetan-n-T}, we obtain, for \(n\ge T\),
\begin{align}
  \Big| \E\big[Q_n\cos 2(\chi_n + \eta)\big]\Big|
  \le &\Big|\E\big[  Q_n \zeta_{n-T}\big]\Big|
  +\E\Big[\big|Q_n(\mu\zeta_n-\mu^{T+1}\zeta_{n-T})\big|\Big] \notag\\
  \le & CE 2^{-T}+C \sqrt E\, T\sqrt E
  \le CTE.  \notag
\end{align}
We now split the average at the scale \(T\). Using the trivial bound \(|Q_n|\le C\sqrt E\) for the initial segment and the preceding estimate for \(n\ge T\), we get
\begin{align}
    \frac{1}{N}\Big|\sum_{n=0}^{N-1}\E\big[Q_n\cos 2(\chi_n + \eta)\big]\Big|
    \le &    \frac{1}{N}\Big|\sum_{n=0}^{T-1}\E\big[Q_n\cos 2(\chi_n + \eta)\big]\Big| \notag\\
    &+   \frac{1}{N}\Big|\sum_{n=T}^{N-1}\E\big[Q_n\cos 2(\chi_n + \eta)\big]\Big| \notag \\
    \le & \frac{CT\sqrt E}{N}+CTE \notag\\
    \le & CTE={\mathcal O}(E|\log E|), \notag
\end{align}
provided \(N\ge E^{-1/2}\). This proves \eqref{eqn:Qncos-exp}.

The proof of \eqref{eqn:Qn-squ-cos-exp} is similar, with \(Q_n\) replaced by \(Q_n^2\). Indeed, for \(n\ge T\),
\begin{align}
  \Big| \E\big[Q^2_n\cos 2(\chi_n + \eta)\big]\Big|
  \le &\Big|\E\big[  Q^2_n \zeta_{n-T}\big]\Big|
  +\E\Big[\big|Q^2_n(\mu\zeta_n-\mu^{T+1}\zeta_{n-T})\big|\Big]  \notag \\
  \le & CE^{3/2} 2^{-T}+C E\, T\sqrt E
  \le CTE^{3/2}.  \notag
\end{align}
Splitting the average at \(T\) again, and using \(|Q_n|^2\le CE\) on the initial segment, gives
\begin{align}
    \frac{1}{N}\Big|\sum_{n=0}^{N-1}\E\big[Q^2_n\cos 2(\chi_n + \eta)\big]\Big|
    \le &    \frac{1}{N}\Big|\sum_{n=0}^{T-1}\E\big[Q^2_n\cos 2(\chi_n + \eta)\big]\Big| \notag\\
    &+   \frac{1}{N}\Big|\sum_{n=T}^{N-1}\E\big[Q^2_n\cos 2(\chi_n + \eta)\big]\Big|  \notag \\
    \le & \frac{CTE}{N}+CTE^{3/2} \notag\\
    \le & CTE^{3/2}={\mathcal O}(E^{3/2}|\log E|), \notag
\end{align}
provided \(N\ge E^{-1/2}\). This proves \eqref{eqn:Qn-squ-cos-exp}.

It remains to prove \eqref{eqn:Qn-squ-cos4-exp}. We use the identity
\begin{align}
    Q_n^2\cos4(\chi_n+\eta)=\operatorname{Re}\Big[ Q_n^2\mu^2\zeta_n^2\Big]. \notag
\end{align}
In place of \eqref{eqn:zetan-n-T}, we use
\begin{align}\label{eqn:zetan-squ-n-T}
    |\zeta^2_n-\mu^{2T}\zeta^2_{n-T}|\le 2 |\zeta_n-\mu^{T}\zeta_{n-T}|\leq C T \sqrt E, \quad n\ge T.  \notag
\end{align}
The same argument used for \eqref{eqn:Qn-squ-cos-exp}, together with the correlation estimate for \(Q_n^2\zeta_{n-T}^2\), then gives
\[
    \frac{1}{N}\sum_{n=0}^{N-1}\E\big[Q^2_n\cos 4(\chi_n + \eta)\big]
    ={\mathcal O}(E^{3/2}|\log E|),
\]
which proves \eqref{eqn:Qn-squ-cos4-exp}. 
\end{proof}


\section{Asymptotics of the Lyapunov exponent}\label{sec:lyp}

Let \(\kappa^{-1}=\E(a^{-1})\) and \(q(x)=\kappa^{-1}-a^{-1}(x)\) be as before.  The autocorrelation function associated with the process \(\{q(2^s x)\}_{s\ge0}\) is defined\footnote{Both \(R_s\) and \(\gamma\) depend on the choice of the sampling function \(q\). We suppress this dependence in the notation.} by
\begin{align}\label{eqn:Rs}
    r_s=\E\big[q(x)q(2^sx)\big], \qquad s\ge 0.
\end{align}
Formally, the associated \emph{spectral density} is given by
\begin{align}
   \gamma(\xi)=r_0+2\sum_{s=1}^{\infty}r_s\cos(s\xi), \qquad \xi\in\R. \label{eqn:gamma}
\end{align}
The next lemma justifies this definition and records the basic properties of \(\gamma\) needed below.

\begin{lemma}\label{lem:density-gamma}
Assume that \(q\in C^1(\T)\) is real valued and satisfies \(\E[q]=0\). Then \(r_s\in \ell^1(\Z_{\ge0})\), and the series in \eqref{eqn:gamma} converges absolutely. Moreover, \(\gamma\) is real-valued, even, and real analytic in \(\xi\). It satisfies \(\gamma(\xi)\ge0\) for all \(\xi\in\R\), and hence \(\gamma(0)\ge0\). If \(q\not\equiv0\), then there exists \(\xi_0>0\) such that
\begin{align}
    \gamma(\xi)>0, \qquad 0<|\xi|<\xi_0.
\end{align}
\end{lemma}
The result is elementary. We include the proof for the reader's convenience. 
\begin{proof}
 By the exponential mixing estimate for the doubling map, Lemma~\ref{lem:doubling-mixing}, and the mean-zero property \(\E(q)=0\), the correlations \(r_s\) decay exponentially. In particular, \(r_s\in \ell^1(\Z_{\ge0})\), so the series in \eqref{eqn:gamma} converges absolutely and uniformly for real \(\xi\). Moreover, the exponential decay of \(r_s\) implies normal convergence in a complex strip around the real axis. Therefore \(\gamma\) is real analytic. Since \(q\) is real valued and the series is written in terms of cosines, \(\gamma\) is real valued and even.

It remains to prove the nonnegativity of \(\gamma\). For \(N\ge1\) and \(\xi\in\R\), set
\[
    S_N(\xi,x)=\sum_{j=0}^{N-1}q(2^j x)e^{-ij\xi}.
\]
A standard Fej\'er-kernel computation gives
\begin{align}
    \frac{1}{N}\E\big[|S_N(\xi,\cdot)|^2\big]
    =
    r_0+2\sum_{s=1}^{N-1}\left(1-\frac{s}{N}\right)r_s\cos(s\xi)
    \longrightarrow \gamma(\xi), \quad N\to \infty.
\end{align}
Thus \(\gamma(\xi)\ge0\) for all \(\xi\in\R\).

Finally, since \(q\not\equiv0\), the function \(\gamma\) is not identically zero. Since \(\gamma\) is real analytic, even, and nonnegative, its Taylor expansion at \(0\) contains only even powers:
\[
    \gamma(\xi)=\gamma(0)+\frac{\gamma''(0)}{2}\xi^2+\cdots .
\]
Thus either \(\gamma(0)>0\), or \(\gamma(0)=0\) and \(\gamma(\xi)={\mathcal O}(\xi^2)\), with the first nonzero even-order coefficient positive. This proves the asserted strict positivity near \(0\).
\end{proof}

We now restate Theorem~\ref{thm:LE-linear-intro} as follows. 
\begin{theorem}\label{thm:LE-linear}
Suppose \(a\in C^1(\T)\) satisfies \eqref{eqn:sample-bound}. Let \(\gamma\) be as in \eqref{eqn:gamma}.
   There is \(E_0\) such that for \(0<E<E_0\), we have
    \begin{align}\label{eqn:LE-linear}
        L(E)=\frac{\kappa \gamma(0)}{8}E+{\mathcal O}(|\log E|^2E^{3/2}).
    \end{align}
\end{theorem}

The linear coefficient is nonnegative
\( \frac{\kappa \gamma(0)}{8}\ge 0\) due to Lemma~\ref{lem:density-gamma}. A direct consequence is the positivity of the Lyapunov exponent near 0.
\begin{corollary}\label{cor:lyp-posi}
  Retain the assumptions of Theorem~\ref{thm:LE-linear}.  Suppose 
    \begin{align}\label{eqn:rho0-posi}
    c_a:= \frac{\kappa \gamma(0)}{8}>0   . 
    \end{align}
    Then there exists \(E_0>0\) such that for \(0<E<E_0\)
    \begin{align}\label{eqn:LE-ca-bound}
     0<   \frac{1}{2}c_a\, E<L(E)<2c_a\, E. 
    \end{align}
\end{corollary}

The following example shows that the condition \eqref{eqn:rho0-posi} is necessary for the Lyapunov to display linear asymptotic behavior near 0.
\begin{theorem}\label{thm:LE-counterexp}
    Suppose \begin{align}
        a(x)=\frac{1}{\,2-\cos(2\pi x)+\cos(4\pi x)\,}.
    \end{align}
    Let \(\gamma\) be as in \eqref{eqn:gamma}. Then 
    \begin{align}
        \gamma(0)=0,\quad {\rm and} \quad \gamma(\xi)=\frac{1}{2}\xi^2+{\mathcal O}(\xi^4)
    \end{align}
 As a consequence, there is \(E_0>0\) such that for  \(0<E<E_0\), 
    \begin{align}
   0\le  L(E)\le C |\log E|^2E^{3/2} . 
    \end{align}
\end{theorem}
\begin{remark}
   In the case of \(\gamma(0)=0\), we do not know whether \(L(E)>0\) holds for small \(E\) nor a suitable lower bound. 
\end{remark}

\begin{remark}
Given an i.i.d. sequence \(a_n(\omega)\), let \(R^{\rm iid}_s\) and \(\gamma^{\rm iid}\) be generated in the same way as \eqref{eqn:Rs} and \eqref{eqn:gamma}. The same asymptotic formula \eqref{eqn:LE-linear} for the Lyapunov exponent holds, see \cite{pastur1992book} and \cite{li2026upper}. Moreover, in the random case, \(R^{\rm iid}_s=0\) for \(s\neq 0\), and 
\begin{align}
     \gamma^{\rm iid}(0)=R^{\rm iid}_0= \E\Big[\big(\kappa^{-1}-a_0^{-1}(\omega)\big)^2\Big]>0.
\end{align}
Hence, the positivity \eqref{eqn:LE-ca-bound} always holds for the random case and does not require any additional assumptions.     
\end{remark}

\begin{proof}
  For the sampling function in the statement, we have \(\kappa^{-1}=\E(a^{-1})=2\), and hence
\[
    q(x)=\kappa^{-1}-a^{-1}(x)=\cos(2\pi x)-\cos(4\pi x).
\]
We compute the correlations \(r_s=\E[q(x)q(2^s x)]\). Using  the orthogonality of the cosines, we first compute
\[
    r_0=\E[q^2]
    =
    \E\big[(\cos(2\pi x)-\cos(4\pi x))^2\big]
    =
     \E\big[(\cos^2(2\pi x)\big]+\E\big[\cos^2(4\pi x)\big]=1.
\]
A similar direct computation gives
\[
    r_1=\E \big[ q(x)q(2x) \big]=-\E\big[\cos^2(4\pi x)\big]=-\frac{1}{2}, 
\]
and \( r_s=0 \) for all \(s\ge2\). 
Thus the spectral density in \eqref{eqn:gamma} is
\[
    \gamma(\xi)=r_0+2r_1\cos\xi=1-\cos\xi.
\]
In particular,
\[
    \gamma(0)=0,
    \qquad
    \gamma(\xi)=\frac{1}{2}\xi^2+{\mathcal O}(\xi^4)
\]
as \(\xi\to0\). Hence the linear term in \eqref{eqn:LE-linear} vanishes, and
\[
    L(E)={\mathcal O}(|\log E|^2E^{3/2})
\]
as \(E\to0^+\).
\end{proof}

The rest of the section is devoted to the proof of Theorem~\ref{thm:LE-linear}. According to \eqref{eqn:lyp-rhon}, it suffices to study the iteration expansion of the radial variable \(\rho_n\) as defined in \eqref{eqn:modi-prufer}. The following expansion of \(\rho_n\) were obtained in \cite{li2026upper}, by directly iterating \eqref{eqn:rho-n-iter} and using \eqref{eqn:Q-asymp}. 
\begin{align} 
    \frac{1}{N}   \log \frac{\rho_N}{\rho_0}
    &= \frac{1}{8N} \sum_{n=0}^{N-1}   Q_n^2 \label{eqn:rho-0} \\
    &\quad + \frac{1}{2N} \sum_{n=0}^{N-1}   \big[Q_n \sin 2(\chi_n + \eta)\big] \label{eqn:rho-1} \\
    &\quad - \frac{1}{4N} \sum_{n=0}^{N-1}  \big[Q_n^2 \cos 2(\chi_n + \eta)\big] \label{eqn:rho-2}\\
    &\quad + \frac{1}{8N}\sum_{n=0}^{N-1} \big[Q_n^2 \cos 4(\chi_n + \eta)\big] \label{eqn:rho-3} \\
    &\quad + {\mathcal O}(E^{3/2}), \label{eqn:rho-4}
\end{align}

Our goal is show that \(\frac{1}{N} \E \log \frac{\rho_N}{\rho_0}\) is asymptotically of \(cE+o(E)\) order for an explicit positive constant \(c>0\). 
By the last two estimates in Lemma~\ref{lem:Q-cos}, we see that 
\begin{align}
    \big|\E\eqref{eqn:rho-2}+\E\eqref{eqn:rho-3}  \big|={\mathcal O}(E^{3/2}|\log E|), \label{eqn:5.17}
\end{align}
which, together with \eqref{eqn:rho-4}, is already of lower order. The expectation value of the first term \eqref{eqn:rho-0} can be explicitly expanded using \eqref{eqn:Q-asymp} as 
\begin{align}\label{eqn:rho-0-expan}
    \frac{1}{8N} \sum_{n=0}^{N-1} \E [Q_n^2 ]=\frac{E^2}{8\sin^2\eta(E)}\E [q^2(x)]=\frac{\kappa \E[q^2]}{8}E+{\mathcal O}(E^{3/2}).
\end{align}

In the i.i.d. case, \(Q_n\) and \(\zeta_n\) are independent which ensures that the expectation of  \eqref{eqn:rho-1} vanishes. Hence,  \eqref{eqn:rho-0-expan} contributes the leading linear term for the Lyapunov exponent. This is no longer true for a deterministic coefficient such as in the case of the doubling map case that we are considering now. The key ingredient is the following estimate for the expectation value of the term in \eqref{eqn:rho-1}. 

\begin{lemma}\label{lem:ave-zetanQn} 
There is \(E_0>0\) such that for \(0<E<E_0\),  and \(N\gtrsim E^{-1}\), we have
\begin{align}\label{eqn:double-ave-zeta-Q}
   \frac{1}{N}\sum_{n=0}^{N-1}\mathbb{E}\big[\zeta_nQ_n\big]=\frac{i\kappa E\mu^{-1}}{4}\Big(\gamma(0)-\mathbb{E} [q^2 ]\Big)+{\mathcal O}(|\log E|^2E^{3/2}), 
\end{align}
\end{lemma}

We first use this lemma to complete the asymptotic expansion for the Lyapunov exponent. The proof of the lemma is left to the end of the section.

\begin{proof}[Proof of Theorem~\ref{thm:LE-linear}]
   Recall that \(\zeta_n\mu=e^{2i(\chi_n+\eta)}\), see \eqref{eqn:zeta-mu-def}. Hence, \( Q_n\sin(2\chi_n+\eta)=\mathrm{Im}(Q_n\zeta_n\mu)\). 
    Therefore, by Lemma \ref{lem:ave-zetanQn}, the contribution of \eqref{eqn:rho-1} is 
\begin{align}
    \E\eqref{eqn:rho-1}=\frac{1}{2N}\mathrm{Im}\left(\mu\sum_{n=0}^{N-1}\mathbb{E}\big[Q_n\zeta_n\big]\right)=\frac{\kappa E}{8}\big(\gamma(0)-\mathbb{E} [q^2 ]\big)+{\mathcal O}(|\log E|^2E^{\frac{3}{2}}). \label{eqn:lyp-main1}
\end{align}
     
  Adding   \eqref{eqn:rho-0}  to  \eqref{eqn:lyp-main1} cancels the term  \(\frac{1}{8}{\kappa \mathbb{E}[q^2]}E\) and proves \eqref{eqn:LE-linear}. 
\end{proof}

\begin{proof}[Proof of Lemma~\ref{lem:ave-zetanQn}]
For  \(0<E<1\), picking \(T\sim \log \frac{1}{E} \le  N\). \(T\) will be made sufficiently large by requiring \(E<E_0\) to be sufficiently small. The choice will be specified later.  Split the sum  on the left hand side of \eqref{eqn:double-ave-zeta-Q} into \(n<T\) and \(n\ge T\) as 
    \begin{align}\frac{1}{N}\sum_{n=0}^{N-1}\mathbb{E}\big[\zeta_nQ_n\big]=&\frac{1}{N}\sum_{n=0}^{T-1}\mathbb{E}\big[\zeta_nQ_n\big]+\frac{1}{N}\sum_{n=T}^{N-1}\mathbb{E}\big[\zeta_nQ_n\big]. \label{eqn:ave-zeta-Q}
    \end{align}
    The first term is bounded by \(CN^{-1}T\sqrt{E}\) for some constant \(C>0\). Taking \(N\sim 1/E\) will put it into \({\mathcal O}(\log(E^{-1})E^{\frac{3}{2}})\). Therefore, we will work with the second term.

It follows from \eqref{eqn:zetaIter} and \eqref{eqn:Q-asymp} that 
\begin{align} \notag 
    \zeta_{n}=\mu\zeta_{n-1}+\frac{i}{2}Q_{n-1}\big(\mu\zeta_{n-1}-1\big)^2+{\mathcal O}(E). 
\end{align}
For \(n\ge T\), iterating the above relation \(T\) times gives 
\begin{align}\label{eqn:zetan-T-iter}
    \zeta_n=\mu^T\zeta_{n-T}+\frac{i }{2}\sum_{s=1}^T\mu^{s-1}Q_{n-s}(\mu\zeta_{n-s}-1)^2+{\mathcal O}(TE).
\end{align}
Similarly to \eqref{eqn:zetan-n-T-squ}, 
by \eqref{eqn:zeta-it-diff}, we have for \(0\le s\le T\le n\), 
\begin{align}  \notag 
    \Big|(\mu \zeta_{n-s} - 1)^2-(\mu^{T+1-s} \zeta_{n-T} - 1)^2\Big|\le 2 \Big| \mu \zeta_{n-s} -\mu^{T+1-s} \zeta_{n-T}  \Big|  
    \le CT\sqrt E.
\end{align}
Then we   substitute \(\zeta_{n-s}\) in the sum  of \eqref{eqn:zetan-T-iter} by \(\zeta_{n-T}\) and recall \eqref{eqn:zeta-it-diff}
\begin{align}\label{eqn:zetaShiftT}
    \zeta_n=\mu^T\zeta_{n-T}+\frac{i }{2}\sum_{s=1}^T\mu^{s-1}Q_{n-s} \big(\mu^{T+1-s}\zeta_{n-T}-1\big)^2+{\mathcal O}(T^2E). 
\end{align}
Expanding the square in the sum  and  plugging into the second term of \eqref{eqn:ave-zeta-Q} imply 
  \begin{align}
    \frac{1}{N}\sum_{n=T}^{N-1}\mathbb{E}\big[\zeta_nQ_n\big]=& \mu^T\frac{1}{N}\sum_{n=T}^{N-1}\mathbb{E}\big[\zeta_{n-T}Q_n\big]\label{eqn:cor-est1}\\
    &+\frac{i }{2}\frac{1}{N}\sum_{n=T}^{N-1}\sum_{s=1}^T\mu^{s-1}\mathbb{E}\big[Q_{n-s}Q_n\big]\label{eqn:cor-est2}\\
    &+\frac{i }{2}\frac{1}{N}\sum_{n=T}^{N-1}\sum_{s=1}^T\mu^{s-1}\mathbb{E}\big[Q_{n-s}Q_n(\mu^{T+1-s}\zeta_{n-T})^2\big]\label{eqn:cor-est3}\\
    &-i \mu^{T}\frac{1}{N}\sum_{n=T}^{N-1}\sum_{s=1}^T\mathbb{E}\big[Q_{n-s}Q_n\zeta_{n-T}\big]\label{eqn:cor-est4}\\&+{\mathcal O}(T^2E^{\frac{3}{2}}).\end{align}

\noindent \(\bullet\) {\bf Estimate of \eqref{eqn:cor-est1}:} applying \eqref{eqn:qn-zetan-T} to each term in the sum   gives 
\begin{align} \notag 
    \Big|\mu^T\frac{1}{N}\sum_{n=T}^{N-1}\mathbb{E}\big[\zeta_{n-T}Q_n\big]\Big|\le    \frac{1}{N}\sum_{n=T}^{N-1} CE2^{-T}\le  CE2^{-T}. 
\end{align}
    
\noindent \(\bullet\) {\bf Estimate of \eqref{eqn:cor-est2}:} since  \(s\ge1\) and \(n-s\ge 0\), splitting \(\T\) into \(2^{n-s}\) dyadic intervals and applying the corresponding change of variables on each of them yield  \(\E[Q_nQ_{n-s}]=\E[Q_{s}Q_0]\)  lead  to 
\begin{align}  \notag 
    \sum_{s=1}^T\mu^{s-1 }\mathbb{E}\big[Q_{n-s}Q_n\big]=   \mu^{-1}\sum_{s=1}^\infty\mu^{s }\mathbb{E}\big[Q_{s}Q_0\big]-\sum_{s>T}  \mu^{s-1 }\mathbb{E}\big[Q_{s}Q_0\big].
\end{align}
The exponential decay of the correlation \eqref{eqn:doubling-mixing} implies that the tail is exponentially small in \(T\)
   \begin{align}
  \big|   \sum_{s>T}  \mu^{s-1 }\mathbb{E}\big[Q_{s}Q_0\big] \Big|\le \sum_{s>T}  2^{-s}\|Q_{ s}\|_{L^1}\sup_{x\in \T}|Q'(x)|\le CE2^{-T}, \label{eqn:424tail}
   \end{align}
where we used the \(\sqrt E\) order for both \(Q\) and \(Q'\) from \eqref{eqn:Qn-bound} and \eqref{eqn:Qn-prime}. 

Let \(r_s=\E[q_sq_0]\) be the correlation function and let \(\gamma\) be the spectral density of \(q(2^s x)\), as defined in \eqref{eqn:Rs} and \eqref{eqn:gamma}, respectively. Recall that \(\mu=e^{2i\eta}\) and
\[
    Q_s=\frac{E}{\sin\eta(E)}q_s,
\]
so that
\[
    \E[Q_sQ_0]=\frac{E^2}{\sin^2\eta}r_s.
\]
Then
\begin{align}
 \sum_{s=1}^\infty\mu^s\E[Q_sQ_0]
 &=
 \frac{E^2}{\sin^2\eta}
 \left[
     \sum_{s=1}^\infty\cos(2\eta s)r_s
     +i\sum_{s=1}^\infty\sin(2\eta s)r_s
 \right]
 \notag\\
 &=
 \frac{E^2}{\sin^2\eta}
 \sum_{s=1}^\infty\cos(2\eta s)r_s
 +\mathcal O(E^{3/2})\sum_{s=1}^\infty s|r_s|
 \notag\\
 &=
 \frac{E^2}{\sin^2\eta}
 \frac{1}{2}\bigl[\gamma(2\eta)-r_0\bigr]
 +\mathcal O(E^{3/2})
 \notag\\
 &=
 \frac{\kappa E}{2}
 \left[\gamma(0)-\E[q^2]\right]
 +\mathcal O(E^{3/2}).
 \label{eqn:424full}
\end{align}
Indeed, \(\sin(2\eta s)=\mathcal O(\eta s)\), and hence
\[
    \frac{E^2}{\sin^2\eta}
    \sum_{s=1}^\infty\sin(2\eta s)r_s
    =
    \mathcal O(E^{3/2})
    \sum_{s=1}^\infty s|r_s|.
\]
Here we used \(E/\sin\eta=\sqrt{\kappa E}+\mathcal O(E^{3/2})\) from \eqref{eqn:eta}, the bound \(\sum_{s\ge1}s|r_s|\le C\), and the analytic expansion
\[
    \gamma(2\eta)=\gamma(0)+\mathcal O(\eta^2)
    =\gamma(0)+\mathcal O(E)
\]
from Lemma~\ref{lem:density-gamma}.

Hence, combing \eqref{eqn:424tail} and \eqref{eqn:424full} gives 
  \begin{align} \notag 
    \sum_{s=1}^T\mu^{s-1}\mathbb{E}\big[Q_{n-s}Q_n\big]=\frac{\kappa E}{2}\mu^{-1} \left[\gamma(0)-\E[q^2]\right]+{\mathcal O}(E^{3/2}),
   \end{align}
provided \(2^{-T}\sim \sqrt E\). Plugging into \eqref{eqn:cor-est2} gives 
  \begin{align}
    \frac{i }{2}\frac{1}{N}\sum_{n=T}^{N-1}   \sum_{s=1}^T\mu^{s-1}\mathbb{E}\big[Q_{n-s}Q_n\big]= \frac{i }{2} \frac{\kappa E}{2}\mu^{-1} \left[\gamma(0)-\E[q^2]\right]+{\mathcal O}(E^{3/2}).     \label{eqn:ave-zetaQ-main-contri}
   \end{align}

\noindent \(\bullet\) {\bf Estimate of \eqref{eqn:cor-est3} and \eqref{eqn:cor-est4}:} the treatment of these two terms are similar,  we will only work with \eqref{eqn:cor-est4}.

For \(1\le s\le T\), 
we apply \eqref{eqn:doubling-mixing} to \(Q_{n-s}Q_n\) and \(\zeta_{n-T}\). Together with \eqref{eqn:zeta-prime-bound}, we have 
   \begin{align*}
  \Big|     \E\big[\zeta_{n-T}Q_{n-s}Q_n\big]- \E\big[\zeta_{n-T}\big]\E \big[Q_{n-s}Q_n\big]\Big|\le & 2^{-(n-s)}\|Q_{n-s}Q_n\|_{L^1}\sup_{x\in \T}|\zeta_{n-T}'(x)| \\
       \le & 2^{-(n-s)} \cdot (CE)\cdot (C\sqrt E 2^{n-T}) \\
        \le & C2^{s-T} E^{3/2} \\
        \le & C  E^{3/2}.
   \end{align*}
Plugged into the sum gives 
\begin{align*}
 \frac{1}{N}\sum_{n=T}^{N-1}\sum_{s=0}^{T }\mathbb{E}\big[\zeta_{n-T}Q_{n-s}Q_n\big]=&\frac{1}{N}\sum_{n=T}^{N-1}\sum_{s=0}^{T }\mathbb{E}\big[\zeta_{n-T}\big] \mathbb{E}\big[Q_{n-s}Q_n\big]+{\mathcal O}(T  E^{3/2})\\
 =&\Big(\frac{1}{N}\sum_{n=T}^{N-1}\mathbb{E}\big[\zeta_{n-T}\big]\Big) \Big(\sum_{s=0}^{T} \mathbb{E}\big[Q_{ s}Q_0\big]\Big)+{\mathcal O}(T  E^{3/2})
\end{align*}

By Lemma \ref{lem:expZeta}, for \(N>1/E\), the first sum is of order \({\mathcal O}\Big( \sqrt E |\log E|\Big)\). Hence,
\begin{align*}
    \frac{1}{N}\sum_{n=T}^{N-1}\sum_{s=0}^{T}\mathbb{E}\big[\zeta_{n-T}Q_{n-s}Q_n\big]=&{\mathcal O}\Big( \sqrt E |\log E|\Big){\mathcal O}\Big( TE\Big)+{\mathcal O}(T  E^{3/2})\\
    =&{\mathcal O}\Big(   E^{3/2} |\log E|^2\Big).
\end{align*}

Combing all estimates for \eqref{eqn:cor-est1}-\eqref{eqn:cor-est4} together proves \eqref{eqn:double-ave-zeta-Q}. Notice that the main contribution is from the estimate \eqref{eqn:ave-zetaQ-main-contri}. 
\end{proof}


\section{Large Deviation Estimates}\label{sec:ldt}

In this section, we prove the large-deviation estimate stated in Theorem~\ref{thm:ldt-intro}. The main result holds under the positivity assumption in Corollary~\ref{cor:lyp-posi}. More precisely, let \(\gamma\) be as in \eqref{eqn:gamma}, and assume that
\begin{align}\label{eqn:ca-ass}
    c_a=\frac{\kappa\gamma(0)}{8}>0,
    \qquad {\rm and}\qquad
    \frac{1}{2}c_a E<L(E)<2c_aE,
    \quad 0<E<E_0,
\end{align}
for some \(E_0>0\).

The deviation estimate for \(\|T_N^E\|\) in Theorem~\ref{thm:ldt-intro} follows from the corresponding estimate for the Pr\"ufer amplitude ratio \(\rho_N/\rho_0\).

\begin{theorem}\label{thm:ldt-rhon}
Assume \(a\in C^1(\T)\) and suppose that \eqref{eqn:ca-ass} holds for some \(E_0>0\). Then, for every \(\varepsilon>0\), after possibly decreasing \(E_0=E_0(\varepsilon)\), there exist constants \(c=c(\varepsilon)>0\) and, for each \(0<E<E_0\), a number \(N_0=N_0(E,\varepsilon)\) such that, for all \(N\ge N_0\),
\begin{align}
      \P\Big(\, x\in \T\, :\, 
      \Big|\frac{1}{N}\log\frac{\rho_N}{\rho_0}-c_aE\Big|
      >\varepsilon c_aE
      \Big)
      \le
      \exp\Big(-c\frac{E}{|\log E|^3}N\Big).
\label{eqn:ldt-rhon}
\end{align}
\end{theorem}

\begin{remark}\label{rem:ldt-equivalent-centers}
The large-deviation estimate in Theorem~\ref{thm:ldt-rhon} will be used below in several equivalent forms. By \eqref{eqn:Fn-rhon} and \eqref{eqn:Tn-rhon},
\[
   \log \frac{\rho_N}{\rho_0}
   =
   \log\|T_N^E\|+{\mathcal O}(|\log E|).
\]
After decreasing \(E_0\) and increasing \(N_0(E,\varepsilon)\) if necessary, this deterministic error can be absorbed into the deviation threshold. Hence, the deviation estimate for \(\|T_N^E\|\) in Theorem~\ref{thm:ldt-intro} follows from the corresponding estimate for the Pr\"ufer amplitude ratio \(\rho_N/\rho_0\) in Theorem~\ref{thm:ldt-rhon}.

Moreover, by Theorem~\ref{thm:LE-linear},
\begin{align}\label{eqn:L-Ln-asym}
    L(E)=c_aE+{\mathcal O}(|\log E|^2E^{3/2}),
    \qquad
    L_N(E)=c_aE+{\mathcal O}(|\log E|^2E^{3/2}).
\end{align}

Thus, after taking \(E_0\) smaller and \(N_0(E,\varepsilon)\) larger if necessary, the center \(c_aE\) may be replaced by either \(L(E)\) or \(L_N(E)\). Consequently, for the applications below, we may use the forms
\begin{align}
    \P\Big(
    x\in\T:
    \Big|
    \frac{1}{N}\log\|T_N^E(x)\|-L(E)
    \Big|
    >
    \varepsilon L(E)
    \Big)
    &\le
    \exp\Big(-c\frac{E}{|\log E|^3}N\Big),
    \notag
\end{align}
and
\begin{align}
    \P\Big(
    x\in\T:
    \Big|
    \frac{1}{N}\log\|T_N^E(x)\|-L_N(E)
    \Big|
    >
    \varepsilon L_N(E)
    \Big)
    &\le
    \exp\Big(-c\frac{E}{|\log E|^3}N\Big),
    \notag
\end{align}
with constants depending on \(\varepsilon\).
\end{remark}

In the proof of \eqref{eqn:ldt-rhon} below, we establish the estimate with the fixed tolerance \(\varepsilon=1/100\), which is sufficient for Theorem~\ref{thm:holder}. This keeps the exposition readable and avoids tracking several inessential numerical constants. The full statement with arbitrary \(\varepsilon>0\), which will be useful in the proof of Theorem~\ref{thm:AL}, follows from the same argument after replacing the numerical constants below by sufficiently small multiples of \(\varepsilon\), and after taking \(E_0\) smaller, \(N_0(E,\varepsilon)\) larger, and \(c(\varepsilon)>0\) smaller if necessary.

In the previous section, we derived the averaged asymptotic expansion \eqref{eqn:rho-0}--\eqref{eqn:rho-4} for \(\log(\rho_N/\rho_0)\). We now refine that analysis by controlling the finite-scale fluctuations of each term in the expansion and proving large-deviation bounds for their tails. The terms have different levels of difficulty. The leading \(Q_n^2\)-term in \eqref{eqn:rho-0} is the most direct one and can be estimated by applying the dyadic large-deviation lemma to a normalized \(C^1\) function. The oscillatory terms \eqref{eqn:rho-2} and \eqref{eqn:rho-3} are naturally of higher order on average, as reflected for example in \eqref{eqn:5.17}; after subtracting these higher-order averages, the corresponding deviation estimates are comparatively coarse. The most delicate contribution is the first-order term \eqref{eqn:rho-1}: it fluctuates at the linear scale in \(E\), and its mean must be extracted carefully by iterating the phase variable once more and separating the resulting correlation terms. This term ultimately determines the weakest large-deviation rate, of order \(E/|\log E|^3\), in the final estimate.

The main probabilistic input is the martingale method developed by Bourgain--Schlag \cite{BourgainSchlag2000CMP} for processes generated by the doubling map. In that approach, suitable dyadic conditional expectations are used to realize the relevant sums as martingales, and Azuma's inequality then yields exponential tail bounds. Although the original argument is written for the cosine sampling function, the underlying martingale mechanism extends naturally to general \(C^1\) sampling functions; the only additional work is to keep track of the conditional-expectation errors that are simpler in the cosine case.

As usual, constants denoted by \(C\), \(c\), \(C_i\), and \(c_i\) may change from line to line, while remaining independent of \(E\) and \(N\). We also allow \(E_0\) to be decreased and \(N_0\) to be increased, without changing the notation, whenever this is needed to absorb deterministic lower-order errors into the deviation thresholds or to absorb polynomial and logarithmic prefactors into the exponential bounds.

We first review the martingale preliminaries needed below. Readers familiar with dyadic martingales for the doubling map may skip to the next subsection.


\subsection{Preliminaries on martingales}

For each \(r\ge 0\), let
\begin{align}
\mathcal I_r:=\{I_{r,k}:0\le k<2^r\},
\qquad
I_{r,k}:=\Big[\frac{k}{2^r},\frac{k+1}{2^r}\Big),
\notag  
\end{align}
and let
\begin{align}
\mathcal F_r
:=\sigma(\mathcal I_r).
\notag  
\end{align}
Thus \(\mathcal F_r\) is the finite \(\sigma\)-algebra generated by the dyadic partition of
\([0,1)\) into intervals of length \(2^{-r}\). Moreover,
\begin{align}
\mathcal F_0\subset \mathcal F_1\subset \mathcal F_2\subset \cdots,
\notag  
\end{align}
so \((\mathcal F_r)_{r\ge 0}\) is a filtration.

If \(f\in L^1([0,1))\), then the conditional expectation of \(f\) with respect to
\(\mathcal F_r\) is given by averaging \(f\) on the dyadic cell containing \(x\), which we denote as 
\begin{align}
\mathbb E_r[f](x)
:=
\E[f\mid \mathcal F_r](x)
=
\sum_{I\in \mathcal I_r}
\left(
\frac{1}{|I|}\int_I f(t)\,dt
\right)\chi_I(x).
 \label{eqn:conditional-expectation-dyadic-sum}
\end{align}
In particular, \(\mathbb{E}_r \big[f\big](x)\) is constant on each interval \(I_{r,k}\), and 
therefore \(\mathcal F_r\)-measurable.

Recall that \((M_r)_{r\ge 0}\) is a martingale with respect to
\((\mathcal F_r)_{r\ge 0}\) if \(M_r\in L^1\) and \(M_r\) is \(\mathcal F_r\)-measurable
for every \(r\), and if for all \(s<r\),
\begin{align}
\E[M_r\mid \mathcal F_s]=M_s.
\notag  
\end{align}

Then \((\mathbb E_r[f])_{r\ge 0}\) is a martingale with respect to
\((\mathcal F_r)_{r\ge 0}\). Indeed, \(\mathbb E_r[f]\) is \(\mathcal F_r\)-measurable by
construction and belongs to \(L^1\). Moreover, for \(s<r\), the tower property gives
\begin{align}
\E_s\big[\mathbb E_r[f]\big] =
\E\bigl(\E[f\mid \mathcal F_r]\mid \mathcal F_s\bigr)=
\mathbb E_s[f].
\label{eqn:tower-martingale}
\end{align}
This is the standard Doob martingale associated with \(f\). In the dyadic setting,
\(\mathbb E_r[f]\) is the step function obtained by averaging \(f\) over intervals of
length \(2^{-r}\). The identity \eqref{eqn:tower-martingale} says that averaging this finer
approximation again over a coarser dyadic partition recovers the coarser dyadic
approximation. 

There is another useful observation for sequences generated by the doubling map. If \(f\in L^1(\T)\) and \(0\le m\le n\), then
\begin{align}
\E_m\bigl[f(2^n x) \bigr]=\E\bigl[f(2^n x)\mid \mathcal F_m\bigr]
=
\E\bigl[f(2^n x)\bigr]
=
\E\bigl[f\bigr].
\label{eqn:doubling-conditional-expectation}
\end{align}
Indeed, \(f(2^n x)\) is \(2^{-n}\)-periodic, and each dyadic interval of length \(2^{-m}\), with \(m\le n\), contains exactly \(2^{n-m}\) full periods. Hence the average of \(f(2^n x)\) over any atom of \(\mathcal F_m\) equals its global average on \([0,1)\), which gives \eqref{eqn:doubling-conditional-expectation}.

\begin{lemma}\label{lem:subsequence-martingale-transform}
Let \((\mathcal F_r)_{r\ge0}\) be the dyadic filtration defined above, and let \(0\le s_0<s_1<s_2<\cdots\) be a strictly increasing sequence of integers. Suppose that, for each \(j\ge1\), \(d_j\in L^1\) is \(\mathcal F_{s_j}\)-measurable and satisfies
\begin{align}
    \E\big[d_j\mid \mathcal F_{s_{j-1}}\big]=0. \notag
\end{align}
Then
\begin{align}
    M_n:=\sum_{j=1}^n d_j,\qquad n\ge1, \notag
\end{align}
is a martingale with respect to the sampled filtration \((\mathcal F_{s_j})_{j\ge0}\).

In particular, suppose that \(F_{s_j}\in L^1\) is \(\mathcal F_{s_j}\)-measurable, \(H_{s_{j-1}}\in L^1\) is \(\mathcal F_{s_{j-1}}\)-measurable, and
\begin{align}
    \E\big[F_{s_j}\mid \mathcal F_{s_{j-1}}\big]=H_{s_{j-1}}.
    \label{eqn:conditional-mean-H}
\end{align}
If \(G_{s_{j-1}}\) is bounded and \(\mathcal F_{s_{j-1}}\)-measurable, then
\begin{align}
    \sum_{j=1}^n
    \bigl(F_{s_j}-H_{s_{j-1}}\bigr)G_{s_{j-1}},
    \qquad n\ge1, \notag
\end{align}
is a martingale with respect to the sampled filtration \((\mathcal F_{s_j})_{j\ge0}\).
\end{lemma}

This is the standard martingale-difference construction, together with its martingale-transform version along the subsequence \((s_j)\). The first statement follows directly from the definition of a martingale, by verifying that
\[
    \E[M_n\mid \mathcal F_{s_{n-1}}]=M_{n-1}.
\]
For the second statement, one only uses that \(G_{s_{j-1}}\) is \(\mathcal F_{s_{j-1}}\)-measurable, and hence can be pulled out of the conditional expectation; the conclusion then follows from the first statement. We omit the details.

\begin{remark}\label{rem:varying-samples-martingale-transform}
The second part of Lemma~\ref{lem:subsequence-martingale-transform} is slightly more flexible than the usual martingale-transform construction. If \(F_{s_j}\) and \(H_{s_{j-1}}\) come from the same martingale, then \(F_{s_j}-H_{s_{j-1}}\) is simply a martingale difference. However, the second part allows the input functions to vary with \(j\), as long as the conditional-mean relation \eqref{eqn:conditional-mean-H} holds. The most frequent construction is obtained by taking
\begin{align}
    F_{s_j}=\E_{s_j}[f_j],
    \qquad
    H_{s_{j-1}}=\E_{s_{j-1}}[f_j].
   \notag 
\end{align}
Then the tower property \eqref{eqn:tower-martingale} gives
\begin{align}
    \E_{s_{j-1}}\big[\E_{s_j}[f_j]\big]
    =
    \E_{s_{j-1}}[f_j], \notag
\end{align}
which implies \eqref{eqn:conditional-mean-H}. This is the form of the martingale construction used in the conditional-expectation replacements below.
\end{remark}


The following is a standard large-deviation bound for sums of martingale differences with bounded increments; see, e.g., \cite[Chapter 7, Theorem 7.2.1 ]{alon92book}.
\begin{theorem}[Azuma's inequality]
      Let $0=X_0,\cdots,X_n$ be a martingale with $|X_{i+1}-X_i|\le 1$ for all $0\le i<n$. 
    Then for arbitrary $\delta>0$, 
    \begin{align}\label{eqn:azuma}
        \P\big( | X_n| >\delta \sqrt n\big)< 2e^{-\delta^2/2}.
    \end{align}
\end{theorem}
\begin{remark}
  Martingales may be viewed as conditional-mean analogues of sums of independent mean-zero random variables. They retain the essential mean-zero structure of i.i.d.\ sums through conditional expectations, while allowing dependence and non-identical distributions. In the i.i.d.\ setting, the corresponding exponential tail estimate is Hoeffding's inequality \cite{chernoff1952measure,hoeffding1963probability}. Azuma's inequality \eqref{eqn:azuma}, also known as the Azuma--Hoeffding inequality, extends this type of large-deviation bound to martingales with bounded increments.
\end{remark}

 The martingale structure described above leads to the following large-deviation estimate for dyadic averages. This result, proved in \cite{BourgainSchlag2000CMP}, is obtained by applying Azuma's inequality to the martingales naturally associated with the doubling map.
\begin{lemma}[{\cite[Lemma 8.1]{BourgainSchlag2000CMP}}]\label{lem:large-deviation-dyadic-average}
Let \(K>1\), and assume that \(F\) is a function on \(\mathbb T\) satisfying \(|F|\le 1\) and \(|F'|\le K\). Then there exists \(N_0=N_0(\delta)\) such that, for all \(N\ge N_0\),
\begin{align}
\mathbb P\Bigg[
x\in \mathbb T
\;\bigg|\;
\E[F]
-\frac{1}{N}\sum_{j=1}^{N} F\bigl(2^{j}x\bigr)
\bigg|
>
\delta
\Bigg]
<
\exp\!\left(
- \,
\frac{ \delta^{2}}{20\log(K\delta^{-1}) }N
\right).
\label{eqn:large-deviation-dyadic-average}
\end{align}
\end{lemma}

\begin{remark}
We include a self-contained proof below, in the present notation, to illustrate the martingale construction and the use of Azuma's inequality in the dyadic setting. The logarithmic factor in \eqref{eqn:large-deviation-dyadic-average} comes from grouping the orbit into martingales along dyadic subsequences with spacing \(T\sim \log(K/\delta)\). In the original proof of \cite[Lemma~8.1]{BourgainSchlag2000CMP}, a closely related estimate is obtained with a different logarithmic loss. For the applications considered here and in \cite{BourgainSchlag2000CMP}, the precise power of this logarithmic correction is not essential. The logarithmic correction can be removed for special sampling functions with suitable Fourier structure, for example \(F(x)=\cos(2\pi x)\). For a general \(C^1\) sampling function, however, it is not clear whether such a loss can be avoided.

The same martingale construction also applies to sparse dyadic averages, namely averages taken along a dyadic subsequence of the original orbit, provided the sampling gap is at least of order \(\log(K/\delta)\). We will use this sparse version in the proof of the Green's function estimates for Anderson localization; see Lemma~\ref{lem:large-deviation-sparse-dyadic-average}.
\end{remark}

\begin{proof}
    It suffices to assume \(\E[f]=0\) and estimate 
\begin{align}
\mathbb P\big(\, 
x\in \mathbb T\, :\, 
\;\big|\; \frac{1}{N}\sum_{n=1}^{N} F\bigl(2^{n}x\bigr)
\big|
>\delta\, 
\big)    .  \notag 
\end{align}
Denote by \(F_n=F(2^n x)\). 
For \(r\le n\),  \eqref{eqn:doubling-conditional-expectation} implies that 
\begin{align}
    \E_r[F_n]=\E[f]=0.  \notag
\end{align}
For \(r>n\), it follows from \(|\frac{d}{dx}F(2^nx)|\le K2^n\), the definition of \(\E_r\) in \eqref{eqn:conditional-expectation-dyadic-sum}, and the mean value theorem   that 
\begin{align}\label{eqn:Er-appro-F}
   \big|F(2^nx) -\E_r[F_n]\big|\le   K2^{n-r}.
\end{align}
For \(\delta>0\), we pick \(T\) so that \(K2^{- {T}}\sim \delta\). Then for \(n\ge T\), we have 
\begin{align}
   \Big|F(2^nx) - \E_{n+T}[F_n]  \Big|  \le K2^{- T }\le \delta/4. \notag
\end{align}
Assume \(N=mT+r_\ast, 1\le r_\ast \le T \).  Hence, 
\begin{align}
   \Big| \frac{1}{N}\sum_{n=1}^{mT} F_n(x)
  -\frac{1}{N}\sum_{n= 1}^{mT} \E_{n+ T }[F_n]  \Big|\le \frac{1}{4}\delta. \label{eqn:finite-block-conditional-approx}
\end{align}
It remains to control the conditional-expectation sum. We do this by grouping the indices according to their residue class modulo \(T\), which decomposes the sum into \(T\) martingales.

For \(1\le n\le mT\), write
\begin{align}
 n=(j-1)T+r,\qquad 1\le r\le T,\qquad 1\le j\le m. \notag
\end{align}

For each fixed residue class \(r\), define
\begin{align}
 s_j^r=jT+r,\qquad 0\le j\le m, \notag
\end{align}
and
\begin{align}
 d_j^r=\E_{s_j^r}\big[F_{s_{j-1}^r}\big],
 \qquad 1\le j\le m. \notag
\end{align}
Then set
\begin{align}
    X^r_0=0,\qquad
    X_k^r=\sum_{j=1}^{k}d_j^r,\qquad 1\le k\le m. \notag
\end{align}
By the first statement of Lemma~\ref{lem:subsequence-martingale-transform}, for each fixed \(r\), \((X_k^r)_{k=0}^{m}\) is a martingale with respect to the sampled filtration \((\mathcal F_{s_j^r})_{j\ge0}\). Indeed, \(d_j^r\) is \(\mathcal F_{s_j^r}\)-measurable, and
\begin{align}
    \E_{s_{j-1}^r}[d_j^r]
    =
    \E_{s_{j-1}^r}\Big[\E_{s_j^r}\big[F_{s_{j-1}^r}\big]\Big]
    =
    \E_{s_{j-1}^r}\big[F_{s_{j-1}^r}\big]
    =
    \E[F]
    =
    0. \notag
\end{align}
Here we used the tower property \eqref{eqn:tower-martingale} and \eqref{eqn:doubling-conditional-expectation}. Moreover, since \(|F|\le1\), the increments satisfy \( |X^r_{k+1}-X_k^r|\le 1.\)

Applying Azuma's inequality \eqref{eqn:azuma} to the martingale \(X_k^r\), with
\[
    \delta'=\frac{1}{2}\delta\sqrt m,
\]
gives
\begin{align}
\P\Big(\big|\sum_{j=1}^{m} d_j^r \big|>  \frac{\delta m}{2}  \Big)
  &<2\exp \Big(-\frac{\delta^2}{8}m\Big) 
\le 2\exp \Big(- \frac{\delta^2}{10T} N\Big),
\label{eqn:azuma-residue-class}
\end{align}
where in the last step we used \(N=mT+r_\ast\), with \(1\le r_\ast\le T\), and took \(N\ge 5T\).

Therefore,
\begin{align}
\P\Big(\Big|\sum_{n= 1}^{mT} \E_{n+ T }[F_n]   \Big|>  \frac{1}{2}\delta  N   \Big)
= \P\Big(\big|\sum_{r=1}^{T }\sum_{j=1}^{m}d_j^r\big|> \frac{1}{2}\delta N  \Big) 
\le& \sum_{r=1}^{T }\P\Big(\big|\sum_{j=1}^{m}d_j^r\big|> \frac{\delta N }{2T}  \Big) \notag \\ 
<&2T \exp \Big(-\frac{\delta^2}{10T} N\Big) \notag \\
<& \exp \Big(-\frac{\delta^2}{20\log(K/\delta)} N\Big),
\label{eqn:azuma-sum-residue-classes}
\end{align}
where we used \(T\sim \log(K/\delta)\) and absorbed the prefactor \(2T\) into the exponential by taking \(N\) sufficiently large.

It remains to account for the tail terms from \(mT+1\) to \(N\). Since this tail has length at most \(T\) and \(|F|\le 1\), choosing \(N>4T/\delta\) gives
\begin{align}
   \bigg| \frac{1}{N}\sum_{n=mT+1}^{N} F_n(x)
\bigg|\le \frac{T}{N}< \frac{1}{4}\delta .
\label{eqn:azuma-tail-block}
\end{align}

Combining \eqref{eqn:azuma-sum-residue-classes}, \eqref{eqn:azuma-tail-block}, and \eqref{eqn:finite-block-conditional-approx}, we obtain
\begin{align}
\P\Big(\Big|\frac{1}{N}\sum_{n=1}^{N} F_n(x)\Big|>  \delta \Big)
 \le
\P\Big(\frac{1}{N}\Big|\sum_{n=1}^{mT} \E_{n+ T }[F_n]  \Big|>  \frac{1}{2}\delta     \Big)  
 &<
\exp \Big(-\frac{\delta^2}{20\log(K/\delta)} N\Big).
\notag
\end{align}
This proves \eqref{eqn:large-deviation-dyadic-average}.
\end{proof}

\subsection{Deviation of \eqref{eqn:rho-0}}
We first estimate the fluctuation of the leading quadratic term \eqref{eqn:rho-0}. Recall that \(Q_n=Q(2^n x)\in C^1(\T)\) satisfies \eqref{eqn:Qn-bound} and \eqref{eqn:Qn-prime}. Choose \(C>0\) large enough so that, for \(F=Q^2/(CE)\), one has \(|F|\le 1\) and \(|F'|\le 2\).

Applying \eqref{eqn:large-deviation-dyadic-average} to this function \(F\), with \(\delta=8c_a/(10C)\), where \(c_a>0\) is as in \eqref{eqn:ca-ass}, gives, after reindexing,
\begin{align}
 \P\Big[x\in \T\,    : \Big|\frac{1}{ N} \sum_{n=0}^{N-1}   \frac{Q_n^2}{CE}-\frac{1}{CE}\E[Q^2]\Big|>\frac{8c_a}{1000C}  \Big]
 < e^{-c_1N}, \notag 
\end{align}
for all \(N\ge N_0\), with some constant \(c_1>0\). Here \(N_0\) and \(c_1\) are independent of \(E\). Hence
\begin{align}
    \P\Big[x\in \T\,    : \Big|\frac{1}{8N} \sum_{n=0}^{N-1}   Q_n^2-\frac{1}{8}\E[Q^2]\Big|>\frac{c_a}{1000}E \Big] 
 <  e^{-c_1N}.
 \label{eqn:ldt-rho0}
\end{align}
This is the desired deviation estimate for the term \eqref{eqn:rho-0}.

 \subsection{Deviation estimates of \eqref{eqn:rho-2} and \eqref{eqn:rho-3}}
 
In view of \eqref{eqn:zeta-mu-def}, the sums in \eqref{eqn:rho-2} and \eqref{eqn:rho-3} can be controlled by the corresponding complex sums involving \(Q_n^2\zeta_n\) and \(Q_n^2\zeta_n^2\). The goal of this subsection is to prove that, for \(E<E_0\), \(N\ge N_0(E)\), and both \(k=1,2\),
\begin{align}
    \P\Big[x\in \T\,  \, : \Big| \frac{1}{N}\sum_{n=0}^{N-1}  Q_n^2 \zeta_n^k\Big|>\frac{c_a}{300}E \Big]
    < 2\exp\Big(-\frac{c_2}{|\log E|}N\Big),
    \label{eqn:ldt-Q2-zeta-k-goal}
\end{align}
for some constant \(c_2>0\) independent of \(E\) and \(N\). Indeed, by \eqref{eqn:zeta-mu-def},
\[
    \frac{1}{N}\sum_{n=0}^{N-1} Q_n^2 \cos\bigl(2k(\chi_n+\eta)\bigr)
    =
    \operatorname{Re}\Bigl( \mu^k \frac{1}{N}\sum_{n=0}^{N-1} Q_n^2\zeta_n^k\Bigr),
    \quad k=1,2.
\]
Since \(|\mu|=1\) and \(|\operatorname{Re} z|\le |z|\), \eqref{eqn:ldt-Q2-zeta-k-goal} implies
\begin{align}
    \P\Big[x\in \T\,  \, : \Big|
    \frac{1}{N}\sum_{n=0}^{N-1} Q_n^2 \cos\bigl(2k(\chi_n+\eta)\bigr)
    \Big|>\frac{c_a}{300}E \Big]
    < 2\exp\Big(-\frac{c_2}{|\log E|}N\Big),
    \quad k=1,2.
    \label{eqn:ldt-Q2-cos-k}
\end{align}
Thus \eqref{eqn:ldt-Q2-cos-k} gives the required large-deviation bounds for \eqref{eqn:rho-2} and \eqref{eqn:rho-3}.

It suffices to prove \eqref{eqn:ldt-Q2-zeta-k-goal} for \(k=1\); the case \(k=2\) is identical. We first subtract the average \(\E[Q^2]\) from \(Q_n^2\) and write
\begin{align} 
  \frac{1}{N}\sum_{n=0}^{N-1}  Q_n^2 \zeta_n =& \frac{1}{N}\sum_{n=0}^{N-1} ( Q_n^2-\E[Q^2]) \zeta_n+ \ \E[Q^2] \frac{1}{N}\sum_{n=0}^{N-1} \zeta_n .\notag 
\end{align}
Applying \eqref{eqn:sum-zeta} to the second sum on the right-hand side, and using \(\E[Q^2]\sim E\), we obtain, for \(N\gtrsim E^{-1}\),
\begin{align}
    \frac{1}{N}\sum_{n=0}^{N-1}  Q_n^2 \zeta_n =& \frac{1}{N}\sum_{n=0}^{N-1}   (Q_n^2-\E[Q^2]) \zeta_n -\frac{\E[Q^2]e^{-i\eta}}{4\sin \eta}\frac{1}{N}\sum_{n=0}^{N-1}Q_n(\mu \zeta_n - 1)^2   + {\mathcal O}(E^{3/2})  \label{eqn:523}
\end{align}

Next, choose \(T\sim \log(E^{-1})\). Using \eqref{eqn:zeta-it-diff}, we replace \(\zeta_n\) by the delayed phase \(\mu^T\zeta_{n-T}\) in both sums. This gives
\begin{align}
    \frac{1}{N}\sum_{n=0}^{N-1}  Q_n^2 \zeta_n =& \frac{1}{N}\sum_{n=T}^{N-1} ( Q_n^2-\E[Q^2]) \zeta_{n-T}\label{eqn:521}\\
    &-\frac{\E[Q^2]e^{-i\eta}}{4\sin \eta}\frac{1}{N}\sum_{n=T}^{N-1}Q_n\big(\mu^{T+1} \zeta_{n-T} - 1\big)^2  \label{eqn:522} \\
    &+ {\mathcal O}(E^{3/2}|\log E|).  
\end{align}

We will estimate the two sums in \eqref{eqn:521} and \eqref{eqn:522} by rewriting them as martingale sums and applying Azuma's inequality \eqref{eqn:azuma}. To do this, we first replace the relevant factors by dyadic conditional expectations. The following lemma controls the resulting error.
\begin{lemma}\label{lem:ExpError}
 There exists \(C>0\) such that  for any \(r,n\ge 0\), 
    \begin{align}\label{eqn:zeta-constant}
      \left|\zeta_n-\mathbb{E}_r \big[\zeta_n\big]\right|\leq C\sqrt E 2^{n-r}
    \end{align}
\end{lemma}
\begin{proof}
   
By Lemma~\ref{lem:zetan-prime}, \(\zeta_n\) is Lipschitz with
\[
       \|\zeta_n'\|_\infty\le C\sqrt E\,2^n .
\]
It follows from the definition of \(\E_r[\cdot]\) in \eqref{eqn:conditional-expectation-dyadic-sum} and the mean value theorem that, for a dyadic interval \(I\) of length \(2^{-r}\) and any \(x\in I\),
\begin{align} 
  \left|\zeta_n(x)-\frac{1}{|I|}\int_I  \zeta_n(t) \,dt \right|
 \le
\frac{1}{|I|}\int_I |\zeta_n(x)-\zeta_n(t)|\,dt   
 \le
  C\sqrt E\,2^n |I|
= C\sqrt E\,2^{n-r}. \notag 
    \end{align}
Since \(\mathbb E_r[\zeta_n]\) is precisely the average of \(\zeta_n\) over the dyadic interval containing \(x\), this proves \eqref{eqn:zeta-constant}.

\end{proof}

\noindent \(\bullet\) {\bf Estimate of \eqref{eqn:521}.}  Let \(F(x)=Q^2(x)-\E[Q^2]\) and \(F_n=F(2^n x)\). Then \(\E[F_n]=0\), \(|F_n|\le CE\), and \(|F_n'|\le CE 2^n\) for some constant \(C>0\). By \eqref{eqn:Er-appro-F}, if \(2^{-T/2}\sim E\), then
\begin{align} \notag
  |F_n-\E_{n+\frac{T}{2}}[F_n]|\le CE  2^{-\frac{T}{2}}={\mathcal O}(E^2).
\end{align}
Similarly, \eqref{eqn:zeta-constant} gives
\begin{align}\notag  
  |\zeta_{n-T}-\E_{n-\frac{T}{2}}[\zeta_{n-T}]|\le C\sqrt E  2^{-\frac{T}{2}}={\mathcal O}(E^{3/2}).
\end{align}
Therefore, \eqref{eqn:521} can be written as
\begin{align}
    \frac{1}{N}\sum_{n=T}^{N-1} ( Q_n^2-\E[Q^2]) \zeta_{n-T}
    =
    \frac{1}{N}\sum_{n=T}^{N-1} \E_{n+\frac{T}{2}}[F_n]\,  \E_{n-\frac{T}{2}}[\zeta_{n-T}]
    + {\mathcal O}(E^{3/2}) .
    \label{eqn:rho2-first-replacement}
\end{align}

The strategy is similar to the proof of Lemma~\ref{lem:large-deviation-dyadic-average}. As in the treatment of the terminal block in \eqref{eqn:azuma-tail-block}, the contribution of the leftover indices is of lower order once \(N\) is chosen sufficiently large. Thus it suffices to consider the case \(N=N_0T\). We group the indices according to their residue class modulo \(T\), so that the sum in \eqref{eqn:rho2-first-replacement} is decomposed into \(T\) sums. For each fixed residue class, the corresponding sum will be realized as a martingale with respect to a sampled dyadic filtration, and Azuma's inequality will then be applied to each martingale.

More precisely, choose \(T\) so that \(T/2\) is an integer and \(2^{-T/2}\sim E\). For \(T\le n\le N_0T\), write
\begin{align}\notag
 n=mT+r,\qquad 0\le r\le T-1,\qquad 1\le m\le N_0.
\end{align}
For each fixed residue class \(r\), consider the sampled dyadic filtration along the subsequence
\begin{align}
 s_m^r=n+\frac{T}{2}=\Big(m+\frac{1}{2}\Big)T+r. \notag
\end{align}
Define
\begin{align}
    d_m^r
    =\E_{n+\frac{T}{2}}[F_n]\,\E_{n-\frac{T}{2}}[\zeta_{n-T}]
    =\E_{s_m^r}[F_n]\,\E_{s_{m-1}^r}[\zeta_{n-T}],
    \qquad 1\le m\le N_0, \label{eqn:541}
\end{align}
and set
\begin{align}
    X^r_0=0,\qquad
    X^r_k=\sum_{m=1}^k d_m^r,\qquad 1\le k\le N_0. \notag
\end{align}
Again, by the first statement of Lemma~\ref{lem:subsequence-martingale-transform}, it suffices to verify the conditional mean-zero property. Clearly, \(d_m^r\) is \(\mathcal F_{s_m^r}\)-measurable, and
\begin{align}
    \E_{s_{m-1}^r}[d_m^r]
    =
    \E_{s_{m-1}^r}[F_n]\,
    \E_{s_{m-1}^r}[\zeta_{n-T}]
    =
    0. \notag 
\end{align}
Here we used that \(\E_{s_{m-1}^r}[\zeta_{n-T}]\) is \(\mathcal F_{s_{m-1}^r}\)-measurable, the tower property \eqref{eqn:tower-martingale}, and \eqref{eqn:doubling-conditional-expectation}.

Moreover, the martingale increments are uniformly bounded:
\[
    |X^r_{k+1}-X_k^r|=|d_{k+1}^r|\le CE,
\]
since \(|F_n|\le CE\) and \(|\zeta_n|=1\). Applying Azuma's inequality \eqref{eqn:azuma} to the rescaled martingale \(X_k^r/(CE)\), with
\begin{align}
  \delta =\frac{c_a}{800C}  \sqrt {N_0},\notag
\end{align}
where \(c_a\) is as in \eqref{eqn:ca-ass}, gives
\begin{align}
\P\Big(\big|\sum_{m=1}^{N_0} \frac{d_m^r}{CE} \big|>  \delta  \sqrt {N_0}  \Big)
&=
\P\Big(\big|\sum_{m=1}^{N_0}d_m^r\big|>\frac{ c_aE N }{800T}\Big) \notag\\
&<2\exp \Big(-\frac{1}{2}\delta ^2\Big)
=2\exp \Big(-\frac{c}{ T} N\Big), \notag
\end{align}
for some constant \(c>0\) independent of \(E\) and \(N\).

Summing over the \(T\) residue classes, and using the fact that the contribution of the leftover indices is lower order once \(N\) is sufficiently large, we obtain
\begin{align}
\P\Big(\Big|\sum_{n=T}^{N-1}\E_{n+\frac{T}{2}}[F_n]\,
\E_{n-\frac{T}{2}}[\zeta_{n-T}] \Big|>  \frac{c_aE}{800}   N   \Big)
&\le   \sum_{r=0}^{T-1}\P\Big(\big|\sum_{m=1}^{N_0}d_m^r\big|> \frac{c_aE N }{800T}  \Big) \notag\\
&< 2T \exp \Big(-\frac{c}{ T} N\Big).
\label{eqn:547}
\end{align}

Plugging \eqref{eqn:547} into \eqref{eqn:rho2-first-replacement}, and using \(T\sim |\log E|\), gives that for \(E<E_0\) and \(N\ge N_0(E)\),
\begin{align}\label{eqn:ldt2}
    \P\left(x\in \T\,  \, : \Big|\frac{1}{N}\sum_{n=T}^{N-1} \left( Q_n^2-\E[Q^2]\right) \zeta_{n-T}\Big|>\frac{c_a}{900}E \right)
    < \exp \Big(-\frac{c}{ |\log E|} N\Big).
\end{align}
Together with the replacement formula \eqref{eqn:rho2-first-replacement}, this gives the desired deviation estimate for \eqref{eqn:521}.

\noindent \(\bullet\) {\bf Estimate of \eqref{eqn:522}.} 
The treatment of \eqref{eqn:522} is similar. We first replace \(Q_n\) and \(\big(\mu^{T+1}\zeta_{n-T}-1\big)^2\) by suitable dyadic conditional expectations. By \eqref{eqn:Er-appro-F}, we have
\begin{align}
  |Q_n-\E_{n+\frac{T}{2}}[Q_n]|
  \le C \sqrt E\, 2^{-\frac{T}{2}}
  ={\mathcal O}(E^{3/2}),  \label{eqn:549}
\end{align}
where \(2^{-T/2}\sim E\). Also, using \eqref{eqn:zeta-constant} for both \(\zeta_{n-T}\) and \(\zeta_{n-T}^2\), we obtain
\begin{align}
 \Big|\big(\mu^{T+1} \zeta_{n-T} - 1\big)^2
&- \E_{n-\frac{T}{2}}\big[\big(\mu^{T+1} \zeta_{n-T} - 1\big)^2\big]\Big| \notag\\
&\le C |\zeta_{n-T}-\E_{n-\frac{T}{2}}[\zeta_{n-T}]|
   +C |\zeta^2_{n-T}-\E_{n-\frac{T}{2}}[\zeta^2_{n-T}]| \notag\\
&\le C\sqrt E\,2^{-\frac{T}{2}}
={\mathcal O}(E^{3/2}).    \label{eqn:552}
\end{align}
Thus, in the sum appearing in \eqref{eqn:522}, we may insert these conditional expectations and get
\begin{align}
 \frac{1}{N}\sum_{n=T}^{N-1}Q_n\big(\mu^{T+1} \zeta_{n-T} - 1\big)^2
=  \frac{1}{N}\sum_{n=T}^{N-1}
\E_{n+\frac{T}{2}}[Q_n]\,
\E_{n-\frac{T}{2}}\big[\big(\mu^{T+1} \zeta_{n-T} - 1\big)^2\big]  +  {\mathcal O}(E^{3/2}). \label{eqn:553}
\end{align}

By \eqref{eqn:eta} and the definition of \(Q\) in \eqref{eqn:Q-samp}, there exists a constant \(C_1>0\) such that
\begin{align}\label{eqn:Q2-prefactor-bound}
    \left|\frac{\E[Q^2]e^{-i\eta}}{4\sin \eta}\right|\le C_1\sqrt E.
\end{align}

The remaining argument again rewrites the main sum in \eqref{eqn:553} as a sum of \(T\) martingales and applies Azuma's inequality, following the same structure as in \eqref{eqn:541}--\eqref{eqn:547}. The difference is that here the martingale increments are of order \(\sqrt E\), rather than \(E\):
\begin{align}
    |X^r_{k+1}-X^r_k|
    \le
    \Big|\E_{n+\frac{T}{2}}[Q_n]\,
    \E_{n-\frac{T}{2}}\big[\big(\mu^{T+1}\zeta_{n-T}-1\big)^2\big]\Big|
    \le C\sqrt E. \notag 
\end{align}
Thus we apply Azuma's inequality \eqref{eqn:azuma} to the rescaled martingales \(X_k^r/(C\sqrt E)\), with
\begin{align}\label{eqn:azuma-delta-rho2-second}
    \delta=\frac{c_a}{1000CC_1}\sqrt{N_0},
\end{align}
where \(c_a>0\) is as in \eqref{eqn:ca-ass}. As in \eqref{eqn:547}, applying Azuma to each residue class and then summing over the \(T\) classes gives
\begin{align}
 \P\Bigg(
 \Big|\frac{1}{C\sqrt E}
 \sum_{n=T}^{N-1}Q_n\big(\mu^{T+1} \zeta_{n-T} - 1\big)^2\Big|
 > \delta\sqrt{N_0}\,T
 \Bigg)
 \le 2T\exp\left(-\frac{1}{2}\delta^2\right).
\end{align}
With \(\delta\) chosen as in \eqref{eqn:azuma-delta-rho2-second} and \(N_0\sim N/T\), this yields
\begin{align}
 \P\Big(\Big|\sum_{n=T}^{N-1}Q_n\big(\mu^{T+1} \zeta_{n-T} - 1\big)^2\Big|
 > \frac{c_a}{1000C_1} \sqrt E\, N \Big)
 \le 2T\exp\left(-\frac{c}{T}  N \right).
 \label{eqn:556}
\end{align}

Using \eqref{eqn:556} together with the prefactor estimate \eqref{eqn:Q2-prefactor-bound} and shrinking $c$ if necessary, we obtain that, for \(E<E_0\) and \(N\ge N_0(E)\),
\begin{align} 
    \P\Big[x\in \T\,  \, : \Big|
    \frac{\E[Q^2]e^{-i\eta}}{4\sin \eta}
    \frac{1}{N}\sum_{n=T}^{N-1}Q_n\big(\mu^{T+1} \zeta_{n-T} - 1\big)^2
    \Big|>\frac{c_a}{900}E \Big]
    < \exp\Big(-\frac{c}{|\log E|}N\Big). \label{eqn:ldt3}
\end{align}

Together with \eqref{eqn:ldt2}, and through the decomposition \eqref{eqn:523}, this proves \eqref{eqn:ldt-Q2-zeta-k-goal} for \(k=1\). The proof for \(k=2\) is identical, replacing \(\zeta_n\) by \(\zeta_n^2\). This completes the deviation estimates for \eqref{eqn:rho-2} and \eqref{eqn:rho-3}.


\subsection{Deviation of \eqref{eqn:rho-1}}
It remains to estimate the term in \eqref{eqn:rho-1}, which is controlled through the sum \(\sum Q_n\zeta_n\). This is the most delicate contribution in the expansion \eqref{eqn:rho-0}--\eqref{eqn:rho-4}. Unlike the terms treated above, its leading contribution is not obtained from a direct large-deviation estimate; instead, it must be extracted by iterating the phase variable once more and separating the resulting correlation terms.

Recall that \(c_a=\kappa\gamma(0)/8>0\) by the standing assumption \eqref{eqn:ca-ass}. The goal of this subsection is to prove that, for \(E<E_0\) and \(N\ge N_0(E)\),
\begin{align}
     \P\Bigg[
     x\in\T:
     \Big|
     \frac{\mu}{N}\sum_{n=T}^{N-1}\zeta_nQ_n
     -
     i\frac{\kappa E}{4}
     \left(\gamma(0)-\E[q^2]\right)
     \Big|
     >\frac{c_a}{200}E
     \Bigg]
     <
     \exp\left(-\frac{cE}{|\log E|^3}N\right),
     \label{eqn:ldt-rho1-goal}
\end{align}
for some constant \(c>0\) independent of \(E\) and \(N\).

Once the deviation estimate in terms of \(\zeta_n\) is established, the sine term in \eqref{eqn:rho-1} is recovered by taking imaginary parts, as in \eqref{eqn:lyp-main1}. Indeed, recall that 
\[
    \frac{1}{N}\sum_{n=T}^{N-1}Q_n\sin 2(\chi_n+\eta)
    =
    \operatorname{Im}\left(
    \mu\,\frac{1}{N}\sum_{n=T}^{N-1}Q_n\zeta_n
    \right),
\]
and
\[
    \operatorname{Im}\left(
      i\frac{\kappa E}{4}
    \left(\gamma(0)-\E[q^2]\right)
    \right)
    =
    \frac{\kappa E}{4}\left(\gamma(0)-\E[q^2]\right).
\]
Therefore, \eqref{eqn:ldt-rho1-goal} implies the corresponding deviation estimate for the sine average,
\begin{align}
     \P\Bigg[
     x\in\T:
     \Big|
      \frac{1}{N}\sum_{n=T}^{N-1}Q_n\sin 2(\chi_n+\eta)
     -
     \frac{\kappa E}{4}\left(\gamma(0)-\E[q^2]\right)
     \Big|
     >\frac{c_a}{200}E
     \Bigg]
     <
     \exp\left(-\frac{cE}{|\log E|^3}N\right).
     \label{eqn:ldt-rho1-sin}
\end{align}

In order to prove \eqref{eqn:ldt-rho1-goal},
as before, choose \(T\) so that \(T/2\) is an integer and \(2^{-T/2}\sim E\), so in particular \(T\sim |\log E|\). 
Recall that the iteration formulas \eqref{eqn:zetan-T-iter} and \eqref{eqn:zetaShiftT} give
\begin{align}\notag 
    \zeta_n=\mu^T\zeta_{n-T}+\frac{i }{2}\sum_{s=1}^T\mu^{s-1}Q_{n-s} (\mu^{T+1-s}\zeta_{n-T}-1)^2+{\mathcal O}(T^2E). 
\end{align}
Multiplying by \(Q_n\) and averaging over \(T\le n\le N-1\), we obtain the decomposition
 \begin{align}
    \frac{1}{N}\sum_{n=T}^{N-1} \zeta_nQ_n
    =& \mu^T\frac{1}{N}\sum_{n=T}^{N-1} \zeta_{n-T}Q_n \label{eqn:562}\\
    &+\frac{i }{2}\frac{1}{N}\sum_{n=T}^{N-1}\sum_{s=1}^T\mu^{s-1} Q_{n-s}Q_n \label{eqn:563}\\
    &+\frac{i }{2}\frac{1}{N}\sum_{n=T}^{N-1}\sum_{s=1}^T\mu^{s-1} Q_{n-s}Q_n(\mu^{T+1-s}\zeta_{n-T})^2 \label{eqn:564}\\
    &-i \mu^{T}\frac{1}{N}\sum_{n=T}^{N-1}\sum_{s=1}^T Q_{n-s}Q_n\zeta_{n-T}  \label{eqn:565}\\
    &+{\mathcal O}(T^2E^{\frac{3}{2}}).
\end{align}

The key idea is the same as in the estimates of \eqref{eqn:521} and \eqref{eqn:522}: after replacing the relevant factors by dyadic conditional expectations, we convert the resulting sums into martingale sums and apply Azuma's inequality. Throughout this part, we use the same sampled dyadic filtration as above. More precisely, choose \(T\) so that \(T/2\) is an integer. For \(T\le n\le N_0T\), write
\begin{align}\label{eqn:nmTr}
 n=mT+r,\qquad 0\le r\le T-1,\qquad 1\le m\le N_0.
\end{align}
For each fixed residue class \(r\), set
\begin{align}\label{eqn:smr}
 s_m^r=n+\frac{T}{2}=\Big(m+\frac{1}{2}\Big)T+r.
\end{align}

\noindent \(\bullet\) {\bf Estimate of \eqref{eqn:562}.} 
We use \eqref{eqn:549} and \eqref{eqn:552} to replace \(Q_n\) and \(\zeta_{n-T}\) by their conditional expectations \(\E_{n+\frac{T}{2}}[Q_n]\) and \(\E_{n-\frac{T}{2}}[\zeta_{n-T}]\), respectively. This gives
\begin{align}
  \frac{1}{N}\sum_{n=T}^{N-1} \zeta_{n-T}Q_n
  =
  \frac{1}{N}\sum_{n=T}^{N-1}
  \E_{n+\frac{T}{2}}[Q_n]\,
  \E_{n-\frac{T}{2}}\big[ \zeta_{n-T}\big]
  +{\mathcal O}(E^{3/2}). \notag
\end{align}

This sum has the same martingale structure as the sums treated in \eqref{eqn:522} and \eqref{eqn:rho2-first-replacement}. The important difference is that the present term does not carry the \({\mathcal O}(\sqrt E)\) prefactor from \eqref{eqn:Q2-prefactor-bound}.  Therefore, to obtain a deviation estimate at the Lyapunov scale \(c_aE\), we choose a smaller Azuma parameter than in \eqref{eqn:azuma-delta-rho2-second}, with an additional factor of \(\sqrt E\). Namely, we take
\begin{align}  
    \delta=c\sqrt{\frac{E N}{T}}, \notag
\end{align}
where \(c>0\) is chosen sufficiently small depending only on the constants in the preceding bounds and on \(c_a\). Applying Azuma's inequality \eqref{eqn:azuma} with this choice of \(\delta\), then summing over the \(T\) residue classes and using \(|\mu^T|=1\), gives that for \(E<E_0\) and \(N\ge N_0(E)\),
\begin{align} 
    \P\Big[x\in \T\,  \, : \Big|\mu^T\frac{1}{N}\sum_{n=T}^{N-1}Q_n  \zeta_{n-T} \Big|>\frac{c_a}{1000 }  E \Big]
    \le \exp\left(-\frac{c E}{|\log E|} N\right).
    \label{eqn:ldt6}
\end{align}

\noindent \(\bullet\) {\bf Conditional expectation of the term \(Q_nQ_{n-s}\) in \eqref{eqn:563}--\eqref{eqn:565}.} 

The remaining terms \eqref{eqn:563}--\eqref{eqn:565} all contain the product \(Q_nQ_{n-s}\). Since Lebesgue measure is invariant under the doubling map, we have
\[
    \E[Q_nQ_{n-s}]=\E[Q_sQ_0],
    \qquad n\ge s.
\]
We denote
\begin{align}
    R_s:=\E[Q_sQ_0], \qquad s\ge0, \notag
\end{align}
and define the centered product
\begin{align}\label{eqn:Fns}
    F_{n,s}(x):=Q_n(x)Q_{n-s}(x)-R_s,\qquad n\ge s.
\end{align}
Then, for some \(C>0\), all \(E<E_0\), and all \(n\ge s\ge 0\),
\begin{align} \label{eqn:6.58}
 |R_s|\le CE,\qquad   \E[F_{n,s}]=0,\qquad  |F_{n,s}|\le CE. 
\end{align}
The next lemma allows us to replace \(F_{n,s}\) by a dyadic martingale difference, up to an error of order \(E^2\).
\begin{lemma}\label{lem:Fns-conditional-error}
For \(1\le s\le T\le n\), with \(2^{-T/2}\sim E\), one has
\begin{align}
    \big|F_{n,s}-\mathbb{E}_{n+\frac{T}{2}}\big[F_{n,s}\big]\big|
    \le {\mathcal O}(E^2), \label{eqn:577}
\end{align}
and
\begin{align}
   \Big| \mathbb{E}_{n-\frac{T}{2}}\big[F_{n,s}\big]\Big|
   \le  {\mathcal O}(E^2). \label{eqn:591}
\end{align}
\end{lemma}

\begin{proof}
First, for \(n\ge T\ge s\), we have
\[
    |(F_{n,s})'|
    =
    |(Q_nQ_{n-s})'|
    \le CE2^n.  
\]
Therefore, by \eqref{eqn:Er-appro-F},
\begin{align}
    \big|F_{n,s}-\mathbb{E}_{n+\frac{T}{2}}\big[F_{n,s}\big]\big|
    \le CE2^{-\frac{T}{2}}
    ={\mathcal O}(E^2), \notag 
\end{align}
which proves \eqref{eqn:577}.

If \(1\le s\le \frac{T}{2}\), then \(n>n-s\ge n-\frac{T}{2}\). Hence, by a change of variables on each dyadic interval of length \(2^{-(n-\frac{T}{2})}\),
\begin{align}
    \mathbb{E}_{n-\frac{T}{2}}\big[Q_nQ_{n-s}\big]
    =
    \E[Q_0Q_s]. \notag 
\end{align}
It follows that
\begin{align}
    \mathbb{E}_{n-\frac{T}{2}}\big[F_{n,s}\big]=0. \notag 
\end{align}

It remains to consider the range \(\frac{T}{2}\le s\le T\). The exponential mixing of the doubling map provides decay of correlations on intervals as well. More precisely, if \(f\in C^1(\T)\) has zero mean, then for every interval \(I\subset\T\), every \(n\ge0\), and every \(g\in C^1(\T)\),
\begin{align}\label{eqn:interval-mixing-by-parts-applied}
\frac{1}{|I|}   \Bigl| \int_I f(2^n x)g(x)\,dx \Bigr|
   \le \frac{2^{1-n}}{|I|}\|f\|_{L^1(\T)}\|g\|_\infty
      +2^{-n} \,\|f\|_{L^1(\T)}\|g'\|_\infty;
\end{align}
see Lemma~\ref{lem:interval-mixing-by-parts} in the appendix.

We apply \eqref{eqn:interval-mixing-by-parts-applied} on each dyadic interval \(I\) with
\[
    |I|=2^{-(n-\frac{T}{2})},\qquad f=Q,\qquad g(x)=Q(2^{n-s}x).
\]

Since \(\E(Q)=0\) and \(Q,Q'={\mathcal O}(\sqrt E)\) by \eqref{eqn:Q-samp}, we have
\[
    \|f\|_{L^1(\T)}\le C\sqrt E,\qquad
    \|g\|_\infty\le C\sqrt E,\qquad
    \|g'\|_\infty\le C\sqrt E\,2^{n-s},
\]
for some constant \(C>0\). Applying \eqref{eqn:interval-mixing-by-parts-applied}, we obtain
\begin{align}
     \frac{1}{|I|}\Bigl| \int_I Q(2^n x)Q(2^{n-s}x)\,dx \Bigr|
      \le C E\left(\frac{2^{-n}}{|I|}+2^{-n}2^{n-s}\right)  
     &\le C E\left(2^{-\frac{T}{2}}+2^{-s}\right) \notag\\
     &\le C E\,2^{-\frac{T}{2}}
     =
     {\mathcal O}(E^2), \notag 
\end{align}
where we used \(|I|=2^{-(n-\frac{T}{2})}\), \(s\ge T/2\), and \(2^{-T/2}\sim E\).

Meanwhile, since \(R_s\) is constant, Lemma~\ref{lem:doubling-mixing} gives
\begin{align}
   \Big| \mathbb{E}_{n-\frac{T}{2}} [R_s]\Big|
   =
   |R_s|
   =
   |\E[Q(2^s x)Q(x)]|
   \le CE2^{-s}
   \le CE2^{-\frac{T}{2}}
   =
   {\mathcal O}(E^2). \notag 
\end{align}
Therefore, for \(\frac{T}{2}\le s\le T\),
\begin{align}
     \Big| \mathbb{E}_{n-\frac{T}{2}} [F_{n,s}]\Big|
     &\le
     \Big| \mathbb{E}_{n-\frac{T}{2}} [Q_nQ_{n-s}]\Big|
     +
     \Big| \mathbb{E}_{n-\frac{T}{2}} [R_s]\Big|
     \le {\mathcal O}(E^2). \notag 
\end{align}

Combining the estimates for \(1\le s\le T/2\) and \(T/2\le s\le T\), we obtain \eqref{eqn:591} for all \(1\le s\le T\).
\end{proof}

\noindent \(\bullet\) {\bf Estimate of \eqref{eqn:563}.} 
We first split \(Q_nQ_{n-s}\) into its mean and centered parts. Recalling the definition of \(F_{n,s}\) in \eqref{eqn:Fns}, we write
\begin{align}
    \frac{1 }{2}\frac{1}{N}\sum_{n=T}^{N-1}\sum_{s=1}^T\mu^{s-1}  Q_{n-s}Q_n
     &=
    \frac{1 }{2}\sum_{s=1}^T\mu^{s-1} \frac{1}{N}\sum_{n=T}^{N-1}F_{n,s}   + \frac{1 }{2}\frac{1}{N}\sum_{n=T}^{N-1}\sum_{s=1}^T\mu^{s-1} R_s .
    \label{eqn:592}
\end{align}

Fix \(1\le s\le T\). The estimates \eqref{eqn:577} and \eqref{eqn:591} allow us to replace \(F_{n,s}\) by the martingale difference
\[
    \big(\E_{n+\frac{T}{2}}-\E_{n-\frac{T}{2}}\big)[F_{n,s}]
\]
up to an error of order \({\mathcal O}(E^2)\). Thus
\begin{align}
 \frac{1}{N}\sum_{n=T}^{N-1} F_{n,s}
 &=
 \frac{1}{N}\sum_{n=T}^{N-1}
 \big(\E_{n+\frac{T}{2}}-\E_{n-\frac{T}{2}}\big)[F_{n,s}]
 +{\mathcal O}(E^2) \notag\\
 &=
 \frac{1}{N}\sum_{r=0}^{T-1}\sum_{m=1}^{N_0}
 \big(\E_{s_m^r}-\E_{s_{m-1}^r}\big)[F_{n,s}]
 +{\mathcal O}(E^2), \label{eqn:6.63}
\end{align}
where \(s_m^r\) is given in \eqref{eqn:smr}, and the range of \(m\) is chosen so that \(n=mT+r\) runs over the relevant indices \(T\le n\le N-1\), with \(N_0\sim N/T\) increments.

For fixed \(s\) and fixed residue class \(r\), the interior sum over \(m\) is a martingale by the second statement of Lemma~\ref{lem:subsequence-martingale-transform}, in the form described in Remark~\ref{rem:varying-samples-martingale-transform}, with \(G=1\). Indeed, the summand is \(\mathcal F_{s_m^r}\)-measurable, and by the tower property \eqref{eqn:tower-martingale},
\begin{align}
    \E_{s_{m-1}^r}\Big[
    \big(\E_{s_m^r}-\E_{s_{m-1}^r}\big)[F_{n,s}]
    \Big]
    =
    \E_{s_{m-1}^r}[F_{n,s}]
    -
    \E_{s_{m-1}^r}[F_{n,s}]
    =
    0.
    \label{eqn:595}
\end{align}
By \eqref{eqn:6.58}, the martingale increments are bounded by \(CE\).

We choose \(\delta=c_a(2000CT)^{-1}\sqrt{N/T}\), where \(c_a\) is as in \eqref{eqn:ca-ass}. The reason for the extra factor \(T^{-1}\) will become clear when we sum over the two exterior indices \(r\) and \(s\). Applying Azuma's inequality \eqref{eqn:azuma} to the rescaled martingale, which has \(N_0\sim N/T\) increments, and normalizing as in the previous application of Azuma, we obtain
\begin{align}
    \P\Big(
    \Big|\sum_m\big(\E_{s_m^r}-\E_{s_{m-1}^r}\big)[F_{n,s}]\Big|
    \ge \frac{c_aE}{2000T^2}N
    \Big)
    \le 2\exp\left(-\frac{c}{T^3}N\right).
    \label{eqn:596}
\end{align}
Summing over the \(T\) residue classes and over \(1\le s\le T\), and using \(|\mu|=1\), the factor \(T^2\) from the double exterior sum cancels the \(T^{-2}\) in the deviation threshold. Hence
\begin{align}
    \P\Big(
    \Big|  \sum_{s=1}^T\mu^{s-1}\sum_{r=0}^{T-1}\sum_m
    \big(\E_{s_m^r}-\E_{s_{m-1}^r}\big)[F_{n,s}]\Big|
    \ge \frac{c_aE}{2000}N
    \Big)
    \le 2T^2\exp\left(-\frac{c}{T^3}N\right).\label{eqn:6.66}
\end{align}
Plugging this into \eqref{eqn:6.63}, and using \(T\sim |\log E|\), we obtain, after taking \(E_0>0\) smaller if necessary so that the accumulated \({\mathcal O}(E^2)\)-error is bounded by \(c_aE/4000\),
\begin{align}
     \P\Big(
     \Big| \frac{1 }{2N}\sum_{s=1}^T\mu^{s-1}  \sum_{n=T}^{N-1}F_{n,s}\Big|
     \ge \frac{c_a}{2000}E
     \Big)
     <  \exp\left(-\frac{c'}{|\log E|^3}N\right). \notag 
\end{align}

It remains to add back the deterministic correlation term in \eqref{eqn:592}. By \eqref{eqn:ave-zetaQ-main-contri},
\begin{align}
    \frac{1 }{2}\frac{1}{N}\sum_{n=T}^{N-1}\sum_{s=1}^T\mu^{s-1}R_s
    =
    \frac{\kappa E}{4}\mu^{-1} \left[\gamma(0)-\E(q^2)\right]
    +{\mathcal O}(E^{3/2}). \notag 
\end{align}
For \(E<E_0\), the error term \({\mathcal O}(E^{3/2})\) is bounded by \(c_aE/2000\). Combining this with the preceding deviation estimate for the centered term gives
\begin{align}
     \P\Bigg(
     \Big|
     \frac{1}{2N}
     \sum_{n=T}^{N-1}\sum_{s=1}^T\mu^{s-1}  Q_{n-s}Q_n
     -
     \frac{\kappa E\mu^{-1}}{4}
     \left[\gamma(0)-\E(q^2)\right]
     \Big|
     \ge \frac{c_a}{1000}E
     \Bigg) 
    <  \exp\left(-\frac{c'}{|\log E|^3}N\right).
    \label{eqn:ldt7}
\end{align}


\noindent \(\bullet\) {\bf Estimate of \eqref{eqn:564} and \eqref{eqn:565}.} 
The terms \eqref{eqn:564} and \eqref{eqn:565} are treated in the same way and satisfy deviation estimates of the same order. We discuss \eqref{eqn:565}, which contains the first power of \(\zeta_{n-T}\); the term \eqref{eqn:564}, containing the second power, is estimated by the same argument.

Using \(Q_{n-s}Q_n=F_{n,s}+R_s\), we decompose
\begin{align}
\frac{1}{N}\sum_{n=T}^{N-1}\sum_{s=1}^T  Q_{n-s}Q_n\zeta_{n-T}
&=
\sum_{s=1}^T\frac{1}{N}\sum_{n=T}^{N-1} F_{n,s}\zeta_{n-T} \label{eqn:599}\\
&\quad+\sum_{s=1}^T R_s\frac{1}{N}\sum_{n=T}^{N-1} \zeta_{n-T}. \label{eqn:5100}
\end{align}

In \eqref{eqn:599}, we use Lemma~\ref{lem:Fns-conditional-error} to replace \(F_{n,s}\) by the martingale difference
\(\big(\E_{n+\frac{T}{2}}-\E_{n-\frac{T}{2}}\big)[F_{n,s}]\), and we replace \(\zeta_{n-T}\) by \(\E_{n-\frac{T}{2}}[\zeta_{n-T}]\). By \eqref{eqn:577}, \eqref{eqn:591}, \eqref{eqn:zeta-constant}, and the bound \(|F_{n,s}|\le CE\) from \eqref{eqn:6.58}, this gives, for each fixed \(1\le s\le T\),
\begin{align}
   \frac{1}{N}\sum_{n=T}^{N-1}   F_{n,s} \zeta_{n-T} 
   =
   \frac{1}{N}\sum_{n=T}^{N-1}
   \big(\E_{n+\frac{T}{2}}-\E_{n-\frac{T}{2}}\big)[F_{n,s}]\,
   \E_{n-\frac{T}{2}} [\zeta_{n-T}]
   +{\mathcal O}(E^2).
   \label{eqn:5101}
\end{align}

With respect to the sampled dyadic filtration \((\mathcal F_{s_m^r})_{m\ge0}\) given by \eqref{eqn:nmTr} and \eqref{eqn:smr}, define, for fixed \(s\) and \(r\),
\begin{align}
    d_m^r
   =
    \big(\E_{n+\frac{T}{2}}-\E_{n-\frac{T}{2}}\big)[F_{n,s}]\,
    \E_{n-\frac{T}{2}}[\zeta_{n-T}] =
    \big(\E_{s_m^r}-\E_{s_{m-1}^r}\big)[F_{n,s}]\,
    \E_{s_{m-1}^r}[\zeta_{n-T}]. \notag 
\end{align}
The partial sums of \(d_m^r\) form a martingale by Lemma~\ref{lem:subsequence-martingale-transform}, in the form described in Remark~\ref{rem:varying-samples-martingale-transform}. Indeed, \(d_m^r\) is \(\mathcal F_{s_m^r}\)-measurable. Moreover, since \(\E_{s_{m-1}^r}[\zeta_{n-T}]\) is \(\mathcal F_{s_{m-1}^r}\)-measurable, it can be pulled out of the conditional expectation. Together with \eqref{eqn:595}, this gives
\begin{align}
    \E_{s_{m-1}^r}[d_m^r]
    =
    \Big(\E_{s_{m-1}^r}
    \big[\big(\E_{s_m^r}-\E_{s_{m-1}^r}\big)[F_{n,s}]\big]\Big)
    \E_{s_{m-1}^r}[\zeta_{n-T}] =
    0. \notag 
\end{align}

 By \eqref{eqn:6.58}, its increments are bounded by \(CE\). Applying Azuma's inequality \eqref{eqn:azuma} as in the estimate leading to \eqref{eqn:596}, with the same choice \(\delta=c_a(4000CT)^{-1}\sqrt{N/T}\), gives
\begin{align}\notag 
\P\Big(
\Big|\sum_m d_m^r\Big|>\frac{c_aE}{4000T^2}N
\Big)
<2\exp\Big(-\frac{c}{T^3}N\Big).
\end{align}

Applying the same summation over the \(T\) residue classes and over \(1\le s\le T\) as in \eqref{eqn:6.66}, the factor \(T^2\) from the double exterior sum cancels the \(T^{-2}\) in the deviation threshold. Together with \eqref{eqn:5101}, and after taking \(E_0>0\) smaller if necessary so that the accumulated \({\mathcal O}(E^2)\)-error is bounded by \(c_aE/4000\), the centered contribution \eqref{eqn:599} satisfies
\begin{align} 
\P\Big(\Big|\sum_{s=1}^T\frac{1}{N}\sum_{n=T}^{N-1} F_{n,s}\zeta_{n-T}\Big|>\frac{c_a}{2000}E\Big)
 <2T^2 \exp \Big(-\frac{c}{ T^3} N\Big)  \le  \exp \Big(-\frac{c'}{ |\log E|^3} N\Big).
\label{eqn:5117}
\end{align}

For the term in \eqref{eqn:5100}, we use the identity \eqref{eqn:sum-zeta} to write
\begin{align}
    \frac{1}{N}\sum_{n=T}^{N-1} \zeta_{n-T}
    =
    \frac{1}{N}\sum_{n=0}^{N-T-1} \zeta_n
    =
    -\frac{e^{-i\eta}}{4N\sin \eta}
    \sum_{n=0}^{N-1}(\mu \zeta_n - 1)^2 Q_n
    +{\mathcal O}(\sqrt E), \notag 
\end{align}
provided \(N>E^{-1}\). The sum on the right-hand side is of the same type as the one estimated in \eqref{eqn:553}--\eqref{eqn:556}, after replacing \(\zeta_n\) by \(\zeta_{n-T}\) using \eqref{eqn:zeta-it-diff}. Choosing the corresponding Azuma parameter and using \(\sin\eta\sim \sqrt E\), we obtain, for \(E<E_0\) and \(N\ge N_0(E)\),
\begin{align} 
    \P\Big[x\in \T\,  \, : \Big|
    -\frac{e^{-i\eta}}{4N\sin \eta}
    \sum_{n=0}^{N-1}(\mu \zeta_n - 1)^2 Q_n
    \Big|>\frac{c_a}{4000CT} \Big]
    < 2T \exp\Big(-\frac{c}{T}N\Big).\notag 
\end{align}
Together with \(|R_s|\le CE\) and the fact that the error term \(T{\mathcal O}(E^{3/2})\) is bounded by \(c_aE/4000\) for \(E<E_0\), this yields
\begin{align} 
    \P\Big[x\in \T\,  \, : \Big|
    \sum_{s=1}^T R_s\frac{1}{N}\sum_{n=T}^{N-1} \zeta_{n-T}
    \Big|>\frac{c_a}{2000}E \Big]
    <
    \exp\left( -\frac{c'}{ |\log E|} N\right).
    \label{eqn:ldt6.79}
\end{align}

Finally, using the decomposition of \eqref{eqn:565} into \eqref{eqn:599} and \eqref{eqn:5100}, together with the estimates \eqref{eqn:5117} for \eqref{eqn:599} and \eqref{eqn:ldt6.79} for \eqref{eqn:5100}, we obtain
\begin{align}
    \P\Big[x\in\T:\Big|
    i\mu^T\frac{1}{N}\sum_{n=T}^{N-1}\sum_{s=1}^T
    Q_{n-s}Q_n\zeta_{n-T}
    \Big|>\frac{c_a}{1000}E\Big]
    <
    \exp\left(-\frac{c'}{|\log E|^3}N\right).\label{eqn:ldt680}
\end{align}
The same argument, with \(\zeta_{n-T}\) replaced by \((\mu^{T+1-s}\zeta_{n-T})^2\), gives the corresponding deviation estimate for \eqref{eqn:564}.

Combining the estimates for \eqref{eqn:562}--\eqref{eqn:565}, multiplying the decomposition by \(\mu\) and using \(|\mu|=1\), and taking \(E_0>0\) smaller and \(N_0(E)\) larger if necessary, all preceding estimates hold and the error term satisfies
\({\mathcal O}(T^2E^{3/2})\le c_aE/1000\). After decreasing \(c>0\) if necessary, this proves \eqref{eqn:ldt-rho1-goal}.


\subsection{All deviations together}

Under the positivity assumption \(c_a>0\) in \eqref{eqn:ca-ass}, we now combine the deviation estimates for the terms in \eqref{eqn:rho-0}--\eqref{eqn:rho-4} and prove Theorem~\ref{thm:ldt-rhon}.

By the decomposition \eqref{eqn:rho-0}--\eqref{eqn:rho-4}, and after taking \(E_0>0\) smaller if necessary so that the remainder term \({\mathcal O}(E^{3/2})\) in \eqref{eqn:rho-4} is bounded by \(c_aE/1000\), the probability of the event
\[
\Bigg\{
x\in\T:
\Big|
\frac{1}{N}\log\frac{\rho_N}{\rho_0}
-\frac{1}{8}\E[Q^2]
-\frac{\kappa E}{8}\left(\gamma(0)-\E[q^2]\right)
\Big|
>5\frac{c_a}{1000}E
\Bigg\}
\]
is controlled by the sum of the deviation estimates in \eqref{eqn:ldt-rho0}, \eqref{eqn:ldt-Q2-cos-k}, and \eqref{eqn:ldt-rho1-sin}. Therefore, after increasing \(N_0(E)\) if necessary and decreasing the constant \(c>0\), we obtain
\begin{align}
\P\Bigg[
x\in\T:
\Big|
\frac{1}{N}\log\frac{\rho_N}{\rho_0}
-\frac{1}{8}\E[Q^2]
-\frac{\kappa E}{8}\left(\gamma(0)-\E[q^2]\right)
\Big|
>5\frac{c_a}{1000}E
\Bigg]
<
\exp\Big(-c\frac{E}{|\log E|^3}N\Big).
\label{eqn:ldt-comb1}
\end{align}

On the other hand, using the definition and asymptotic expansion of \(Q\) in \eqref{eqn:Q-samp}, we have
\begin{align}
    \frac{1}{8}\E[Q^2]
    + \frac{\kappa E}{8}\left(\gamma(0)-\E[q^2]\right)
    =
    \frac{\kappa\gamma(0)}{8}E
    +{\mathcal O}(E^2)
    =
    c_aE+{\mathcal O}(E^2). \notag 
\end{align}
Taking \(E_0>0\) smaller if necessary so that the error term \({\mathcal O}(E^2)\) is bounded by \(c_aE/1000\), and combining this with \eqref{eqn:ldt-comb1}, we obtain
\begin{align}
\P\Bigg[
x\in\T:
\Big|
\frac{1}{N}\log\frac{\rho_N}{\rho_0}
-c_aE
\Big|
>6\frac{c_a}{1000}E
\Bigg]
<
\exp\Big(-c\frac{E}{|\log E|^3}N\Big).
\notag 
\end{align}
This implies \eqref{eqn:ldt-rhon}, after shrinking \(c>0\) if necessary. This completes the proof of Theorem~\ref{thm:ldt-rhon}.

\begin{remark}
    Notice that the final rate function is \(E/|\log E|^3\), which tends to zero as \(E\to0\). This weakest rate comes from the deviation estimate for the first-order term \(Q_n\zeta_n\) in \eqref{eqn:ldt-rho1-goal}. By comparison, the leading \(Q_n^2\)-term in \eqref{eqn:rho-0} has an \(E\)-independent exponential rate, as shown in \eqref{eqn:ldt-rho0}, while the estimates for the oscillatory \(Q_n^2\zeta_n^k\)-terms, \(k=1,2\), have the stronger rate \(1/|\log E|\), as shown in \eqref{eqn:ldt-Q2-zeta-k-goal}.
\end{remark}




\section{Applications of the Large-Deviation Estimate}\label{sec:applications}

In this section, we discuss applications of the large-deviation estimate proved in the previous section. Throughout the section, we work under the positivity assumption \eqref{eqn:ca-ass}; in particular, \(c_a>0\) and \(L(E)\sim c_aE\) for \(0<E<E_0\). The systematic use of large-deviation estimates in this context goes back to Bourgain--Goldstein \cite{bourgain2000nonperturbative} and Goldstein--Schlag \cite{goldstein2001holder} for quasi-periodic Schr\"odinger cocycles, and has since been extended to many other models. We do not attempt to survey these developments here; see the monograph of Bourgain \cite{bourgain2005greens} for a comprehensive account.

Our goal is to explain how the large-deviation estimate \eqref{eqn:ldt-main} leads to the H\"older regularity of the IDS and the Lyapunov exponent in Theorem~\ref{thm:holder}, and to Anderson localization at small energies in Theorem~\ref{thm:AL}. We give only the main outline in this section, keeping the presentation focused on the role of the LDT. For completeness, self-contained proofs are provided in Appendix~\ref{sec:holder} and Appendix~\ref{sec:AL}. The arguments are close in spirit to those for Schr\"odinger operators generated by the doubling map in \cite{BourgainSchlag2000CMP}.

\subsection{H\"older regularity of the Lyapunov exponent and the IDS}\label{sec:applications-holder}

We use the large-deviation estimate in the fixed-tolerance form \(\varepsilon=1/100\), as explained in Remark~\ref{rem:ldt-equivalent-centers}. This fixed tolerance is sufficient for the Avalanche Principle, see  Proposition~\ref{prop:avalanche-principle}, which is used to prove the finite-scale estimate \eqref{eqn:L-Ln-L2n}; the details are given in Appendix~\ref{sec:holder}.

Recall the notation
\begin{align} 
    t_n(x)=t_n(x;E)=\frac{1}{n}\log\|T_n^E(x)\|,
    \qquad
    L_n(E)=\E[t_n]. \notag 
\end{align}

For \(0<E<E_0\) and sufficiently large \(n\), the LDT implies that both \(t_n(x;E)\) and \(t_{2n}(x;E)\) are close to \(c_aE\), apart from exceptional sets whose measures are exponentially small in \(n\). Applying the LDT along the subsequence of the doubling orbit
\[
    x,\ 2^n x,\ 2^{2n}x,\ldots,
\]
one obtains a large-measure set on which the blocked transfer matrices
\[
    B_j=T_n^E(2^{jn}x),\qquad j=0,\ldots,m-1,
\]
satisfy the hypotheses of the Avalanche Principle. Here \(m\) is chosen exponentially large in the LDT scale.
Roughly speaking, the Avalanche Principle allows one to control the logarithmic norm of a long product of \(SL(2,\R)\) matrices in terms of the norms of the individual matrices and their adjacent products, provided that the individual norms are large and that consecutive factors do not exhibit significant cancellation. More precisely, the LDT gives exponential lower bounds on \(\|B_j\|\) and shows that the adjacent products \(B_{j+1}B_j\) are nearly multiplicative in logarithmic norm. The Avalanche Principle then implies that the logarithmic norm of the long product \(B_{m-1}\cdots B_0\) is determined, up to a controlled error, by these local quantities.

Integrating the resulting pointwise estimate for the long product \(B_{m-1}(x)\cdots B_0(x)\) over the large-measure set and controlling its complement by the LDT gives the finite-scale relation
\begin{align}
\bigl|L_{n_1}(E)+L_n(E)-2L_{2n}(E)\bigr|
\le
C\exp\Big(-c\frac{E}{|\log E|^3}n\Big), \notag 
\end{align}
where \(n_1\) is a larger scale chosen exponentially in the LDT scale. Iterating this scale construction yields
\begin{align}
\bigl|L(E)+L_n(E)-2L_{2n}(E)\bigr|
\le
C\exp\Big(-c\frac{E}{|\log E|^3}n\Big).
\label{eqn:L-Ln-L2n}
\end{align}
The exceptional-set estimates and the verification of the Avalanche Principle hypotheses are carried out in \eqref{eqn:m612}--\eqref{eqn:Fn1-Fn-629}, while the induction on scales is given in \eqref{eqn:ns-choice}--\eqref{eqn:L2ns-Lns}.

It remains to convert \eqref{eqn:L-Ln-L2n} into continuity in \(E\). By \eqref{eqn:rho-0}--\eqref{eqn:rho-4} and \eqref{eqn:Tn-rhon}, there exists \(C>0\) such that, after decreasing \(E_0\) if necessary, for all \(0<E<E_0\) and all \(n\ge n_0(E)\),
\begin{align}
    \frac{1}{n}\log\|T_n^E(x)\|
    =
    \frac{1}{n}\log\frac{\rho_n}{\rho_0}
    +{\mathcal O}\Big(\frac{|\log E|}{n}\Big)
    \le C\sqrt E,
    \notag
\end{align}
uniformly in \(x\in\T\). Consequently,
\begin{align}
    \partial_E \log \|T_n^E(x)\|
    \le n e^{C n\sqrt E}. \label{eqn:tn-prime}
\end{align}
Therefore, for \(0<E<E'<E_0\),
\begin{align}
    |L_n(E)-L_n(E')|
    \le
    n e^{C n\sqrt E}|E-E'|.
    \notag
\end{align}
The same estimate holds for \(L_{2n}\), with \(n\) replaced by \(2n\); after enlarging \(C>0\), we keep the same form of the bound. Combining this finite-scale Lipschitz bound with \eqref{eqn:L-Ln-L2n}, we obtain, for some constants \(C,c>0\),
\begin{align}
      |L(E)-L(E')|
      \le
      2n e^{C n\sqrt E}|E-E'|
      +
      C\exp\Big(-c\frac{E}{|\log E|^3}n\Big).
      \notag
\end{align}
We now choose \(n\) so as to balance the two terms. This choice has order
\[
    n\sim \frac{\big|\log |E-E'|\big|}{\sqrt E},
\]
up to harmless logarithmic corrections. In particular, as \(|E-E'|\to0\), this choice satisfies \(n\to\infty\); hence, for \(|E-E'|\) sufficiently small, we have \(n\ge n_0(E)\). With this choice of \(n\), after decreasing \(c>0\) if necessary, we obtain
\begin{align}
      |L(E)-L(E')|
      \le
      C|E-E'|^{\tau(E)},
      \qquad
      \tau(E)=c\frac{\sqrt E}{|\log E|^3}.
      \notag
\end{align}
Thus \(L(E)\) is locally H\"older continuous on \((0,E_0)\), with a H\"older exponent depending on the energy scale.

The H\"older regularity of the IDS follows from the Thouless formula and the Hilbert transform. More precisely, by the argument in \cite[Proposition~10.2 and Lemma~10.3]{goldstein2001holder}, the IDS inherits the same order of local H\"older regularity as the Lyapunov exponent. This proves Theorem~\ref{thm:holder}.

\begin{remark}\label{rem:holder-exponent-degenerates}
The argument above yields a H\"older exponent of order
\[
    \tau(E)=c\frac{\sqrt E}{|\log E|^3},
\]
which tends to zero as \(E\to0^+\) and may not be optimal. It remains unclear whether this deterioration is merely a limitation of the present method or whether one can obtain an exponent that is uniform as \(E\to0^+\). If a uniform exponent is not available, it would be natural to ask what the optimal dependence on \(E\) should be. A similar issue already appears for Schr\"odinger cocycles generated by the doubling map at small coupling in the Bourgain--Schlag setting \cite{BourgainSchlag2000CMP}, where the available H\"older exponent also deteriorates as the coupling tends to zero.
\end{remark}

\subsection{Anderson localization}

We now turn to Anderson localization. In this subsection, we use the \(\varepsilon\)-dependent form of the large-deviation estimate described in Remark~\ref{rem:ldt-equivalent-centers}. Together with positivity of the Lyapunov exponent on \((0,E_0)\), this yields finite-volume Green's function decay on suitable long intervals. The resulting Green's function estimate is then used, through the standard generalized-eigenfunction expansion, to prove exponential decay of eigenfunctions.

Let \(H=H(x)\) be as in \eqref{eqn:div-grad}. For \(\Lambda\subseteq\Z_{\ge0}\), denote by \(H_\Lambda\) the coordinate restriction of \(H\) to \(\Lambda\). In particular, for \(\Lambda=[0,n-1]\), we write
\begin{align}
    H_n=H_n(x):=H_{[0,n-1]}.
    \label{eqn:Hn-def}
\end{align}
For \(E\notin\sigma(H_\Lambda)\), the finite-volume Green's function is
\begin{align}
    G_\Lambda(x;E)=\big(H_\Lambda-E\big)^{-1}.
    \label{eqn:G-Lambda-def}
\end{align}
When \(\Lambda=[0,n-1]\), we write \(G_n(x;E)=G_{[0,n-1]}(x;E)\). For \(n_1,n_2\in[0,n-1]\), its kernel is denoted by
\begin{align}
    G_n(x;E)(n_1,n_2)
    =\big\langle {n_1}\big|{\big(H_n-E\big)^{-1}}\big|{n_2}  \big\rangle.
    \label{eqn:Gn-kernel-def}
\end{align}

The passage from LDT estimates to Green's function decay is standard in this setting. For Schr\"odinger operators generated by the doubling map, this was carried out in \cite[\S8,\S9]{BourgainSchlag2000CMP}. The same strategy applies to the present div-grad model, with minor but necessary modifications coming from the form of the transfer matrices and the corresponding Cramer's rule. We record the resulting Green's function estimate here and give the proof in Appendix~\ref{sec:AL}.
 
\begin{theorem}\label{thm:ldt-green1}
Let \((0,E_0)\) be an interval on which the large-deviation estimate \eqref{eqn:ldt-main} holds and \(L(E)>0\) for \(E\in(0,E_0)\). Fix a compact interval \(I\subsetneq(0,E_0)\) and \(0<\varepsilon<\varepsilon_\ast\), where \(\varepsilon_\ast>0\) is an absolute constant. Then there exists a full-measure set \(\Omega\subset\T\) such that, for every \(x\in\Omega\) and for spectral-a.e. \(E\in I\), there exists \(N_0=N_0(x,E,I,\varepsilon)\) with the following property. For every \(N\ge N_0\), set
\[
    \bar N=\lfloor e^{(\log N)^2}\rfloor .
\]
Then, for every \(n_1,n_2\in[\bar N,2\bar N]\) with \(|n_1-n_2|\ge \bar N/2\), one has
\begin{align}
    \big|G_{[\bar N,2\bar N]}(x;E)(n_1,n_2)\big|
    \le
    \exp\Big(-(1-\varepsilon)L(E)|n_1-n_2|\Big).
    \label{eqn:green-decay}
\end{align}
\end{theorem}

Once \eqref{eqn:green-decay} is available, Anderson localization follows by the standard Green's function expansion argument. Indeed, let \(\xi^E=\{\xi_n^E\}\) be a polynomially bounded generalized eigenfunction associated to an energy \(E\in I\subsetneq(0,E_0)\). Applying the Green's function representation on intervals \([\bar N,2\bar N]\), with \(n\) placed well inside the interval, expresses \(\xi_n^E\) in terms of boundary values of \(\xi^E\). The off-diagonal decay in \eqref{eqn:green-decay} then dominates the polynomial growth allowed by Shnol's theorem, and hence \(\xi^E\) decays exponentially. This gives pure-point spectrum with exponentially decaying eigenfunctions on each compact \(I\subsetneq(0,E_0)\); the passage to \([0,E_0]\) follows by the compact-interval reduction in Appendix~\ref{sec:AL}. Thus Theorem~\ref{thm:AL} follows.


\appendix

 \section{Exponential mixing and decay of correlations}

We record two elementary decay-of-correlation estimates for the doubling map. The first is a local estimate on subintervals, and the second is its global consequence on the torus. These estimates provide the exponential mixing input used in the large-deviation analysis, especially for controlling correlation terms such as \(\E[\zeta_{n-T}Q_n]\).

\begin{lemma}\label{lem:interval-mixing-by-parts}
Let \(f\in C^1(\T)\) satisfy \(\int_\T f(x)\,dx=0\). Then, for every interval \(I\subset \T\), every \(n\ge0\), and every \(g\in C^1(\T)\),
\begin{align}\label{eqn:interval-mixing-by-parts}
   \Bigl| \int_I f(2^n x)g(x)\,dx \Bigr|
   \le 2^{1-n}\|f\|_{L^1(\T)}\|g\|_\infty
      +2^{-n}|I|\,\|f\|_{L^1(\T)}\|g'\|_\infty.
\end{align}
\end{lemma}

\begin{proof}
Since \(f\) has zero mean, the function
\[
    P(x):=\int_0^x f(t)\,dt
\]
is a bounded \(1\)-periodic primitive of \(f\). Indeed,
\[
    P'=f,\qquad
    P(x+1)-P(x)=\int_x^{x+1}f(t)\,dt=\int_\T f(t)\,dt=0,
\]
and
\[
    \|P\|_\infty
    \le \int_0^1 |f(t)|\,dt
    =
    \|f\|_{L^1(\T)}.
\]
In particular, since \(P\) is \(1\)-periodic,
\[
    \sup_{n\ge0}\sup_{x\in\T}|P(2^n x)|
    \le
    \|P\|_\infty
    \le
    \|f\|_{L^1(\T)}.
\]
Now let \(I=[a,b]\subset\T\). Since
\[
    \frac{d}{dx}P(2^n x)=2^n f(2^n x),
\]
integration by parts gives
\begin{align}
   \int_I f(2^n x)g(x)\,dx
   &=
   2^{-n}\int_I \frac{d}{dx}P(2^n x)\,g(x)\,dx \notag\\
   &=
   2^{-n}\bigl[P(2^n x)g(x)\bigr]\big|_a^b
   -2^{-n}\int_I P(2^n x)g'(x)\,dx .  \notag
\end{align}
Taking absolute values and using the bound on \(\|P\|_\infty\), we obtain
\begin{align}
   \Bigl| \int_I f(2^n x)g(x)\,dx \Bigr|
   &\le 2^{1-n}\|P\|_\infty\|g\|_\infty
      +2^{-n}|I|\,\|P\|_\infty\|g'\|_\infty \notag\\
   &\le 2^{1-n}\|f\|_{L^1(\T)}\|g\|_\infty
      +2^{-n}|I|\,\|f\|_{L^1(\T)}\|g'\|_\infty.  \notag
\end{align}
This proves \eqref{eqn:interval-mixing-by-parts}.
\end{proof}

The preceding local estimate immediately yields the following global decay-of-correlation bound on the torus.

\begin{lemma}\label{lem:doubling-mixing}
Let \(f,g\in C^1(\T)\). If
\begin{align}\label{eqn:condi-mean-zero}
    \int_{\T} f(x)\,dx=0,
\end{align}
then, for every \(n\in\Z_{\ge0}\),
\begin{align}\label{eqn:doubling-mixing-shifted-zero}
\abs{ \int_{\T} f(2^n x)\,g(x)\,dx }
\le
2^{-n}\norm{f}_{L^1(\T)}\sup_{x\in\T}|g'(x)|.
\end{align}
Consequently, for general \(f,g\in C^1(\T)\), one has
\begin{align}\label{eqn:doubling-mixing}
\abs{
\int_{\T} f(2^n x)\,g(x)\,dx
-
\int_{\T} f(x)\,dx \int_{\T} g(x)\,dx
}
\le
2^{1-n}\norm{f}_{L^1(\T)}\sup_{x\in\T}|g'(x)|.
\end{align}
\end{lemma}

\begin{proof}
Assume first that \(\int_\T f(x)\,dx=0\). Applying Lemma~\ref{lem:interval-mixing-by-parts} with \(I=\T\), the boundary term vanishes by the periodicity of \(P\) and \(g\). Hence
\begin{align}
\abs{ \int_{\T} f(2^n x)\,g(x)\,dx }
&\le
2^{-n}\|f\|_{L^1(\T)}\|g'\|_\infty,  \notag
\end{align}
which proves \eqref{eqn:doubling-mixing-shifted-zero}.

For the general case, set \(F=f-\int_\T f(x)\,dx\). Then \(\int_\T F(x)\,dx=0\) and \(\|F\|_{L^1(\T)}\le 2\|f\|_{L^1(\T)}\). Also,
\begin{align}
\int_\T f(2^n x)g(x)\,dx
-
\int_\T f(x)\,dx\int_\T g(x)\,dx
=
\int_\T F(2^n x)g(x)\,dx.  \notag
\end{align}
Applying \eqref{eqn:doubling-mixing-shifted-zero} to \(F\) gives \eqref{eqn:doubling-mixing}.
\end{proof}

\begin{remark}
We state and prove the preceding estimates under \(C^1\) assumptions, which are sufficient for the applications to \(C^1\) sampling functions in the div-grad model. These regularity assumptions are not optimal. In Lemma~\ref{lem:interval-mixing-by-parts}, the assumption \(f\in C^1(\T)\) can be relaxed to \(f\in L^1(\T)\) with zero mean: the function \(P\) in the proof is then an absolutely continuous \(1\)-periodic primitive of \(f\), and the same integration-by-parts argument remains valid. On the other hand, the regularity of \(g\) in the global correlation estimate can also be weakened to H\"older continuity, with \(\|g'\|_\infty\) replaced by the corresponding \(C^\alpha\) seminorm. In that case the integration-by-parts argument is no longer available; instead, one uses the standard inverse-branch argument for the doubling map, or equivalently the spectral properties of the Ruelle--Perron--Frobenius transfer operator. For the general theory and broader results on decay of correlations for uniformly expanding maps, we refer the reader to \cite{baladi2000positive,viana2016foundations}.
\end{remark}


\section{H\"older regularity of the Lyapunov exponent and the IDS}\label{sec:holder}

In this appendix, we provide the details of the Avalanche Principle argument outlined in Section~\ref{sec:applications-holder}. The purpose is to prove the finite-scale estimate \eqref{eqn:L-Ln-L2n}, which is used there to establish the H\"older regularity of the Lyapunov exponent and the IDS.

We will use the following form of the Avalanche Principle, which can be found in \cite[Proposition~2.2]{goldstein2001holder}; see also \cite[Proposition~6.1]{bourgain2005greens}.
\begin{proposition}[Avalanche Principle]\label{prop:avalanche-principle}
Let \(B_1,\dots,B_n\) be a sequence in \(\mathrm{SL}_2(\mathbb R)\) satisfying
\begin{align}
\min_{1\le j\le n}\|B_j\|
\ge \mu
\ge n,
\label{eqn:avalanche-hyp-1}
\end{align}
and
\begin{align}
\max_{1\le j<n}
\left|
\log \|B_{j+1}\|
+\log \|B_j\|
-\log \|B_{j+1}B_j\|
\right|
<
\frac{1}{2}\log \mu.
\label{eqn:avalanche-hyp-2}
\end{align}
Then there is an absolute constant \(C>0\), such that 
\begin{align}
\left|
\log \|B_n\cdots B_1\|
+
\sum_{j=2}^{n-1}\log \|B_j\|
-
\sum_{j=1}^{n-1}\log \|B_{j+1}B_j\|
\right|
\le
C\frac{n}{\mu}.
\end{align}
\end{proposition}

Recall the notation
\begin{align}\label{eqn:tn-Ln}
    t_n(x)=t_n(x;E)=\frac{1}{n}\log\|T_n^E(x)\|,
    \qquad
    L_n(E)=\E[t_n].
\end{align}
By the asymptotic formula \eqref{eqn:LE-linear}, we have
\begin{align}
    L_n(E)=c_aE+{\mathcal O}(E^{3/2}|\log E|^2).
    \notag
\end{align}
Since \(c_a>0\) by \eqref{eqn:ca-ass}, after decreasing \(E_0\) if necessary this implies
\begin{align}
    \big|L_n(E)-c_aE\big|
    \le \frac{1}{100}c_aE.
    \label{eqn:Ln-crho66}
\end{align}

For \(0<E<E_0\), \(n\ge n_0(E)\), and \(j\ge0\), define
\begin{align}
\Omega^1_n(j)
&=
\left\{
x\in \mathbb T :
\left|t_n\bigl(2^{jn}x\bigr)-c_aE\right|
>
\frac{c_aE}{100}
\right\}, \qquad
\Omega^2_n(j)
=
\left\{
x\in \mathbb T :
\left|t_{2n}\bigl(2^{jn}x\bigr)-c_aE\right|
>
\frac{c_aE}{100}
\right\}. \notag
\end{align}
For \(m\ge1\), set
\begin{align}
\Omega^0_m(n)
=
\bigcup_{j=0}^{m-1}\Omega^1_n(j)
\cup
\bigcup_{j=0}^{m-1}\Omega^2_n(j).
\notag
\end{align}

It follows from Theorem~\ref{thm:ldt-rhon} and Remark~\ref{rem:ldt-equivalent-centers}, with the fixed tolerance \(\varepsilon=1/100\), that, after replacing \(c>0\) by a smaller constant to absorb the factor \(2\) in the \(2n\)-scale estimate,
\begin{align}
\P\big[\Omega^1_n(j)\big]
\le
\exp\Big(-c\frac{E}{|\log E|^3}n\Big),
\qquad
\P\big[\Omega^2_n(j)\big]
\le
\exp\Big(-c\frac{E}{|\log E|^3}n\Big).
\notag
\end{align}

Take
\begin{align}
m=\Big\lfloor \frac{1}{2n}\exp \Big(\frac{1}{2}c\frac{E}{|\log E|^3} n\Big) \Big\rfloor,
\qquad
n_1=mn. \notag 
\end{align}
Then, for \(0<E<E_0\) and \(n\ge n_0(E)\),
\begin{align}
    m<n_1
    < \frac{1}{2}\exp \Big(\frac{1}{2}c\frac{E}{|\log E|^3} n\Big)
    < \exp \Big(\frac{99}{100}c_a E n\Big).
    \label{eqn:m612}
\end{align}
Here and below, we allow \(E_0>0\) to be decreased and \(n_0(E)\) to be increased, without changing the notation, whenever this is needed to absorb lower-order terms or polynomial and logarithmic prefactors into the exponential estimates. Similarly,
\begin{align}
    m
    >
    \frac{1}{4n}\exp \Big(\frac{1}{2}c\frac{E}{|\log E|^3} n\Big)
    >
    \exp \Big(\frac{1}{4}c\frac{E}{|\log E|^3} n\Big).
    \label{eqn:m613}
\end{align}
Therefore, using the estimates for \(\Omega_n^1(j)\) and \(\Omega_n^2(j)\),
\begin{align}
\P\big[\Omega^0_m(n)\big]
\le
2m\exp \Big(-c\frac{E}{|\log E|^3} n\Big)
<
\exp \Big(-\frac{1}{2}c\frac{E}{|\log E|^3} n\Big). \notag
\end{align}

To apply the Avalanche Principle, let
\begin{align}
   B_j=T^E_n\bigl(2^{jn}x\bigr), \qquad j=0,\ldots,m-1. \notag
\end{align}
Then
\begin{align}
t_n
=
\frac{1}{n}\log \|B_j\|, \qquad 
t_{2n}\bigl(2^{jn}x\bigr)
=
\frac{1}{2n}\log \|B_{j+1}B_j\|.
\notag
\end{align}

For \(x\notin \Omega_m^0(n)\) and \(j=0,\ldots,m-1\), we have
\[
\left|t_n\bigl(2^{jn} x \bigr)-c_aE\right|
<
\frac{1}{100}c_aE,
\qquad
\left|t_{2n}\bigl(2^{jn} x \bigr)-c_aE\right|
<
\frac{1}{100}c_aE.
\]
Equivalently,
\begin{align}
\frac{99}{100}c_aE\,n
<
\log\|B_j\|
<
\frac{101}{100}c_aE\,n,
\label{eqn:fn617}
\end{align}
and
\begin{align}
\frac{99}{100}c_aE\,2n
<
\log\|B_{j+1}B_j\|
<
\frac{101}{100}c_aE\,2n.
\label{eqn:f2n619}
\end{align}

It follows from \eqref{eqn:m612} and \eqref{eqn:fn617} that, for \(j=0,\ldots,m-1\),
\begin{align}
    \|B_j\|
    >
    \exp\Big(\frac{99}{100}c_a E n\Big)
    =:\mu
    >
    m. \notag
\end{align}
Thus \eqref{eqn:avalanche-hyp-1} holds. Moreover, for \(j=0,\ldots,m-2\), \eqref{eqn:fn617} and \eqref{eqn:f2n619} give
\begin{align}
&\Bigl|
\log \|B_{j+1}\|
+
\log \|B_j\|
-
\log \|B_{j+1}B_j\|
\Bigr| \notag\\
&\qquad\le
\bigl|\log \|B_{j+1}\|-c_a E n\bigr|
+
\bigl|\log \|B_j\|-c_a E n\bigr|
+
\bigl|c_a E\,2n-\log \|B_{j+1}B_j\|\bigr| \notag\\
&\qquad<
\frac{1}{100}c_aEn
+
\frac{1}{100}c_aEn
+
\frac{1}{100}c_aE\,2n
=
\frac{4}{100}c_aEn
<
\frac{1}{2}\frac{99}{100}c_aEn
=
\frac{1}{2}\log\mu. \notag
\end{align}
Hence \eqref{eqn:avalanche-hyp-2} also holds. Applying the Avalanche Principle, Proposition~\ref{prop:avalanche-principle}, to the sequence \(B_0,\ldots,B_{m-1}\), we obtain
\begin{align}
\Bigl|
\log \bigl\|B_{m-1}\cdots B_0\bigr\|
+
\sum_{j=1}^{m-2}\log \|B_j\|
-
\sum_{j=0}^{m-2}\log \bigl\|B_{j+1}B_j\bigr\|
\Bigr|
\le
C\frac{m}{\mu}. \notag
\end{align}

Dividing by \(n_1=mn\), and using the definitions of \(B_j\), \(t_n\), and \(t_{2n}\), this gives
\begin{align}
\Biggl|
\frac{1}{n_1}\log \bigl\|T^E_{n_1}(x)\bigr\|
+
\frac{1}{m}\sum_{j=1}^{m-2} t_n\bigl(2^{jn}x\bigr)
-
\frac{2}{m}\sum_{j=0}^{m-2} t_{2n}\bigl(2^{jn}x\bigr)
\Biggr|
\le
C\frac{m}{n_1\mu}
\le
\frac{C}{\mu}.
\label{eqn:Fn1-Fn-629}
\end{align}


For \(x\in \Omega_m^0(n)\) and \(0<E<E_0<1\), the left-hand side of \eqref{eqn:Fn1-Fn-629} is bounded trivially by a constant \(C_1>0\), depending only on bound \(a_+,a_-\) in \eqref{eqn:sample-bound}. Integrating \eqref{eqn:Fn1-Fn-629} over \(\T\setminus \Omega_m^0(n)\) and using this trivial bound on \(\Omega_m^0(n)\), we obtain
\begin{align}
\Biggl|
L_{n_1}(E)
+
\frac{m-2}{m}L_n(E)
-
\frac{2(m-1)}{m}L_{2n}(E)
\Biggr|
&\le
\frac{C}{\mu}
+
C_1 \P\big[\Omega_m^0(n)\big] \notag\\
&\le
C\exp \Big(-\frac{99}{100}c_aEn\Big)
+
C_1\exp \Big(-\frac{1}{2}c\frac{E}{|\log E|^3}n\Big) \notag\\
&\le
C\exp \Big(-\frac{1}{2}c\frac{E}{|\log E|^3}n\Big),
\label{eqn:Ln1-Ln2n-pre}
\end{align}
where \(C>0\) has been enlarged so that the first exponential term is absorbed into the second.

Combining \eqref{eqn:Ln1-Ln2n-pre} with \eqref{eqn:Ln-crho66} and \eqref{eqn:m613}, we get
\begin{align}
\bigl|L_{n_1}(E)+L_n(E)-2L_{2n}(E)\bigr|
&\le
\frac{2}{m}\,\bigl|L_n(E)-L_{2n}(E)\bigr|
+
C\exp \Big(-\frac{1}{2}c\frac{E}{|\log E|^3}n\Big) \notag\\
&\le
2\exp \Big(-\frac{1}{4}c\frac{E}{|\log E|^3}n\Big)
\frac{2c_aE}{100}
+
C\exp \Big(-\frac{1}{2}c\frac{E}{|\log E|^3}n\Big) \notag\\
&\le
C\exp \Big(-\frac{1}{4}c\frac{E}{|\log E|^3}n\Big).
\label{eqn:Ln1-Ln}
\end{align}


Applying the same argument to the scale \(2n_1=2mn\), with the same choice of \(m\), gives
\begin{align}
\bigl|L_{2n_1}(E)+L_n(E)-2L_{2n}(E)\bigr|
\le
C\exp \Big(-c\frac{E}{|\log E|^3} n\Big).
\label{eqn:L2n1-Ln}
\end{align}
Combining \eqref{eqn:Ln1-Ln} and \eqref{eqn:L2n1-Ln}, we obtain
\begin{align}
\bigl|L_{2n_1}(E)-L_{n_1}(E)\bigr|
\le
C\exp \Big(-c\frac{E}{|\log E|^3} n\Big).
\label{eqn:L2n1-Ln1}
\end{align}

Starting from \(n_0=n\), we now construct inductively a sequence \((n_s)_{s\ge0}\) by
\begin{align}
    n_{s+1}
    =
    n_s
    \Big\lfloor
    \frac{1}{2n_s}
    \exp \Big(\frac{1}{2}c\frac{E}{|\log E|^3} n_s\Big)
    \Big\rfloor .
    \label{eqn:ns-choice}
\end{align}
By the convention above, \(n_{s+1}\ge n_s+n\), and hence \(n_s\ge sn\). Applying \eqref{eqn:L2n1-Ln1} at scale \(n_s\) gives
\begin{align}
\bigl|L_{2n_{s+1}}(E)-L_{n_{s+1}}(E)\bigr|
\le
C\exp \Big(-c\frac{E}{|\log E|^3} n_s\Big) \le
C\exp \Big(-c\frac{E}{|\log E|^3} sn\Big).
\label{eqn:L2ns-Lns}
\end{align}
Summing \eqref{eqn:L2ns-Lns} over \(s\ge1\), we obtain
\begin{align}
  \bigl|L(E)-L_{n_1}(E)\bigr|
 \le
  \sum_{s=1}^{\infty}
  \bigl|L_{2n_{s+1}}(E)-L_{n_{s+1}}(E)\bigr| 
  &\le
  C\sum_{s=1}^{\infty}
  \exp \Big(-c\frac{E}{|\log E|^3} sn\Big) \notag\\
  &\le
  C\exp \Big(-c\frac{E}{|\log E|^3} n\Big), \notag 
\end{align}
after decreasing \(c>0\) if necessary. Combining this with \eqref{eqn:Ln1-Ln}, we conclude that
\begin{align}
\bigl|L(E)+L_n(E)-2L_{2n}(E)\bigr|
\le
C\exp \Big(-c\frac{E}{|\log E|^3} n\Big). \notag 
\end{align}
This is the finite-scale estimate stated in \eqref{eqn:L-Ln-L2n} in the main text.


\section{Green's Function Estimate and Anderson Localization near zero}\label{sec:AL}

In this appendix, we provide the details of the Green's function estimates used to prove Anderson localization near the bottom of the spectrum. The argument follows the strategy of \cite[\S8,\S9]{BourgainSchlag2000CMP} for Schr\"odinger operators generated by the doubling map. The modifications needed for the present div-grad Jacobi model mainly concern the relation between transfer matrices and finite-volume determinants, and the corresponding form of Cramer's rule.

The goal is to prove the Green's function decay estimate \eqref{eqn:green-decay}. The proof proceeds in three steps. First, we relate finite-volume determinants to transfer matrices. Second, we use the LDT to obtain uniform upper bounds and suitable lower bounds for transfer-matrix growth. Finally, Cramer's rule and a resolvent expansion yield exponential off-diagonal decay of Green's functions, which implies Anderson localization by the standard Shnol argument.

We recall the finite-volume boundary convention in more detail. For an interval \(\Lambda=[c,d]\subseteq\Z_{\ge0}\), \(H_\Lambda\) is the restriction of \(H\) to \(\Lambda\) with Dirichlet boundary conditions on the exterior boundary
\[
    \partial^+\Lambda=\{c-1,d+1\},
    \qquad
    u_{c-1}=u_{d+1}=0.
\]
Thus, for \(n\in\Lambda\),
\begin{align}
    (H_\Lambda u)_n
    =
    -a_{n+1}u_{n+1}
    +(a_n+a_{n+1})u_n
    -a_nu_{n-1}.
   \notag
\end{align}
For \(E\notin\sigma(H_\Lambda)\), the finite-volume Green's function is
\begin{align}
    G_\Lambda(x;E)=\big(H_\Lambda(x)-E\big)^{-1}. \notag
\end{align}
For \(n_1,n_2\in\Lambda\), its kernel is denoted by
\begin{align}
    G_\Lambda(x;E)(n_1,n_2)
    =
    \braket{n_1}{\big(H_\Lambda(x)-E\big)^{-1}}{n_2}. \notag
\end{align}
In particular, when \(\Lambda=[0,n-1]\), we write \(H_n=H_\Lambda\) and \(G_n=G_\Lambda\), as in \eqref{eqn:Hn-def}--\eqref{eqn:Gn-kernel-def}.

We will use the following standard spectral-theoretic reduction.

\begin{lemma}\label{lem:AL-compact-reduction}
Let \(0<E_1<E_2<E_0\), and set \(I=[E_1,E_2]\subsetneq(0,E_0)\). Suppose that, for a.e. \(x\in\T\), the operator \(H(x)\) has pure-point spectrum in \(I\cap\sigma(H(x))\), with exponentially decaying eigenfunctions. If this holds for every compact interval \(I\subsetneq(0,E_0)\), then, for a.e. \(x\in\T\), the operator \(H(x)\) has pure-point spectrum in \([0,E_0]\cap\sigma(H(x))\), with exponentially decaying eigenfunctions.
\end{lemma}

By Lemma~\ref{lem:AL-compact-reduction}, it is enough to prove pure-point spectrum with exponentially decaying eigenfunctions on a fixed compact interval
\[
    I=[E_1,E_2]\subsetneq(0,E_0).
\]
Throughout the rest of this appendix, we fix such an interval \(I\). We work under the positivity assumption \eqref{eqn:ca-ass} and use the \(\varepsilon\)-dependent large-deviation estimates from Theorem~\ref{thm:ldt-rhon} and Remark~\ref{rem:ldt-equivalent-centers}. 

By \eqref{eqn:ca-ass}, \(L(E)\ge \frac12 c_aE\) for \(0<E<E_0\). Since \(I\subsetneq(0,E_0)\) is compact, there exists \(c_I>0\), depending only on \(I\), such that
\begin{align}
    L(E)\ge c_I,\qquad E\in I.
    \label{eqn:cI-ass}
\end{align}


To bound the Green's function from above, we use Cramer's rule. Set
\begin{align}
    D_n^E(x):=\det\bigl(H_n(x)-E\bigr),
    \qquad
    D_0^E:=1.  \label{eqn:Dn}
\end{align}
Recall that we use the \(0\)-based convention \(H_n=H_{[0,n-1]}\), so the kernel indices of \(G_n(x;E)\) range from \(0\) to \(n-1\). Thus, for \(0\le n_1\le n_2\le n-1\), Cramer's rule gives
\begin{align}\label{eqn:cramer}
    G_n(x;E)(n_1,n_2)
    =
    \frac{
    \left(\prod_{j=n_1+1}^{n_2}a_j(x)\right)
    D_{n_1}^E(x)\,
    D_{n-n_2-1}^E(2^{n_2+1}x)
    }
    {
    D_n^E(x)
    }.
\end{align}
Here an empty product is interpreted as \(1\).

The following identity links the transfer matrix to the finite-volume determinants. It will be the main bridge between Cramer's rule and transfer-matrix growth estimates:
\begin{align}\label{eqn:Tn-full-det-formula}
    T_n^E(x)
    =
    \frac{1}{a_0a_1\cdots a_{n-1}}
    \begin{pmatrix}
        D_n^E(x) & -a_0^2 D_{n-1}^E(2x) \\
        D_{n-1}^E(x) & -a_0^2 D_{n-2}^E(2x)
    \end{pmatrix},
    \qquad n\ge 1.
\end{align}
The proof is a straightforward induction; we include the details for completeness.

\begin{proof}[Proof of \eqref{eqn:Tn-full-det-formula}]
    
By definition,
\[
    D_1^E(x)=a_0+a_1-E.
\]
For \(n\ge2\), expanding the determinant along the last row gives the recursion
\begin{align} 
    D_n^E(x)
    =
    \bigl(a_{n-1}+a_n-E\bigr)D_{n-1}^E(x)
    -
    a_{n-1}^2D_{n-2}^E(x).   \notag 
\end{align}
This recursion implies
\begin{align}
    \begin{pmatrix}
        D_n^E(x) \\ D_{n-1}^E(x)
    \end{pmatrix}
    =
    \bigl(a_0a_1\cdots a_{n-1}\bigr)\,
    T_n^E(x)
    \begin{pmatrix}
        1 \\ 0
    \end{pmatrix},
    \qquad n\ge1.  \notag 
\end{align}
Equivalently,
\begin{align}\label{eqn:Tn-det-relation}
    T_n^E(x)
    \begin{pmatrix}
        1 \\ 0
    \end{pmatrix}
    =
    \frac{1}{a_0a_1\cdots a_{n-1}}
    \begin{pmatrix}
        D_n^E(x) \\ D_{n-1}^E(x)
    \end{pmatrix}.
\end{align}

It remains to compute the second column of \(T_n^E(x)\). Since
\begin{align}
    A_0^E(x)\begin{pmatrix}0\\1\end{pmatrix}
    =
    \frac1{a_0}
    \begin{pmatrix}
        a_1+a_0-E & -a_0^2\\
        1 & 0
    \end{pmatrix}
    \begin{pmatrix}
        0\\1
    \end{pmatrix}
    =
    -a_0
    \begin{pmatrix}
        1\\0
    \end{pmatrix}, \notag 
\end{align}
we have
\begin{align}
    T_n^E(x)\begin{pmatrix}0\\1\end{pmatrix}
    =
    -a_0\,
    A_{n-1}^E(x)\cdots A_1^E(x)
    \begin{pmatrix}
        1\\0
    \end{pmatrix}. \notag 
\end{align}
Applying \eqref{eqn:Tn-det-relation} to the shifted cocycle gives
\begin{align}
    A_{n-1}^E(x)\cdots A_1^E(x)
    \begin{pmatrix}
        1\\0
    \end{pmatrix}
    =
    \frac{1}{a_1a_2\cdots a_{n-1}}
    \begin{pmatrix}
         D_{n-1}^E(2x)\\
         D_{n-2}^E(2x)
    \end{pmatrix}. \notag 
\end{align}
Therefore,
\begin{align}
    T_n^E(x)
    \begin{pmatrix}
        0\\1
    \end{pmatrix}
    =
    \frac{1}{a_0a_1\cdots a_{n-1}}
    \begin{pmatrix}
        -a_0^2D_{n-1}^E(2x)\\
        -a_0^2D_{n-2}^E(2x)
    \end{pmatrix}. \notag 
\end{align}
Together with \eqref{eqn:Tn-det-relation}, this proves \eqref{eqn:Tn-full-det-formula}.

\end{proof}

\subsection{A uniform upper bound and a sampled lower bound for transfer matrices}

To bound the Green's function through Cramer's rule \eqref{eqn:cramer} and the determinant-transfer relation \eqref{eqn:Tn-full-det-formula}, we need information on both upper and lower bounds of the transfer-matrix growth. Because \(\log\|T_n^E(x)\|\) is subadditive in scale, upper bounds are more naturally stable under blocking. Pointwise lower bounds are more delicate. In this subsection we prove a uniform upper bound and a weakened, sampled lower bound. The strengthening of the lower bound along a denser part of the orbit is left to the next subsection. The proof uses dyadic large-deviation estimates, starting with the following sparse-orbit version of Lemma~\ref{lem:large-deviation-dyadic-average}.

\begin{lemma}[Sparse dyadic averages]\label{lem:large-deviation-sparse-dyadic-average}
Let \(K>1\), and assume that \(F\in C^1(\T)\) satisfies \(|F|\le 1\) and \(|F'|\le K\). Then there exist absolute constants \(C,c>0\) such that, for every \(\delta>0\) and every \(k\in\Z_+\), whenever
\begin{align}
    m\ge C\delta^{-2}
    \max\left\{
        1,\frac{\log(K/\delta)}{k}
    \right\},
    \label{eqn:sparse-average-parameter-condition}
\end{align}
one has
\begin{align}
\mathbb P\Bigg[
x\in\mathbb T:
\left|
\E[F]
-
\frac{1}{m}\sum_{j=0}^{m-1}F\bigl(2^{jk}x\bigr)
\right|
>
\delta
\Bigg]
\le
\exp\left(
-c\,
\min\left\{
1,\frac{k}{\log(K/\delta)}
\right\}\delta^2 m
\right).
\label{eqn:large-deviation-sparse-dyadic-average}
\end{align}
\end{lemma}

\begin{remark}
The proof uses the same martingale construction as in Lemma~\ref{lem:large-deviation-dyadic-average}, applied now to a sampled dyadic orbit. More precisely, when the sampling gap \(k\) is smaller than the smoothing scale \(\log(K/\delta)\), one groups the sampled orbit into residue classes so that each class has an effective spacing comparable to \(\log(K/\delta)\).

The estimate interpolates between the original dyadic average and the genuinely sparse case. When \(k\lesssim \log(K/\delta)\), the rate function is of order
\[
    \min\left\{
    1,\frac{k}{\log(K/\delta)}
    \right\}
    \sim
    \frac{k}{\log(K/\delta)},
\]
which is consistent with the logarithmic loss in \eqref{eqn:large-deviation-dyadic-average} when \(k=1\), up to constants. When \(k\gtrsim \log(K/\delta)\), the sampled dyadic map \(x\mapsto 2^k x\) has stronger effective hyperbolicity and mixing properties, and the estimate has no logarithmic loss. In this range, the right-hand side of \eqref{eqn:large-deviation-sparse-dyadic-average} reduces to \(\exp(-c\delta^2m)\).
\end{remark}




\begin{proof}
The proof follows the same martingale construction as in the proof of Lemma~\ref{lem:large-deviation-dyadic-average}. Similarly, we may assume \(\E(F)=0\). Choose \(T\sim \log(K/\delta)\) so that \(K2^{-T}<\delta/2\), and set
\[
    q=\left\lceil \frac{T}{k}\right\rceil .
\]
Then \(qk\ge T\). We split the average into residue classes modulo \(q\). Writing
\[
    j=\ell q+r,\qquad 0\le r\le q-1,
\]
we have
\begin{align}
    \frac{1}{m}\sum_{j=0}^{m-1}F(2^{jk}x)
    =
    \sum_{r=0}^{q-1}\frac{m_r}{m}
    \left(
    \frac{1}{m_r}
    \sum_{\ell=0}^{m_r-1}F\bigl(2^{(\ell q+r)k}x\bigr)
    \right),
    \notag
\end{align}
where \(m_r\) is the number of indices \(0\le j\le m-1\) of the form \(j=\ell q+r\). In particular,
\[
    \frac{m}{q}-1\le m_r\le \frac{m}{q}+1
\]
for each residue class that occurs.

For each fixed \(r\), we have
\[
    (\ell q+r)k=rk+\ell(qk),
\]
so the corresponding sampled orbit has effective spacing \(qk\ge T\). Therefore, by the same conditional-expectation approximation used in \eqref{eqn:Er-appro-F}--\eqref{eqn:finite-block-conditional-approx},
\begin{align}
    \left|
    \frac{1}{m_r}\sum_{\ell=0}^{m_r-1}F\bigl(2^{(\ell q+r)k}x\bigr)
    -
    \frac{1}{m_r}\sum_{\ell=0}^{m_r-1}
    \E_{rk+(\ell+1)qk}
    \bigl[
    F\bigl(2^{rk+\ell qk}x\bigr)
    \bigr]
    \right|
    &\le
    K2^{-qk}
    \le
    K2^{-T}
    <
    \frac{\delta}{2}.
    \notag
\end{align}

For this fixed residue class \(r\), set
\[
    d_\ell^{\,r}
    :=
    \E_{rk+(\ell+1)qk}
    \bigl[
    F\bigl(2^{rk+\ell qk}x\bigr)
    \bigr],
    \qquad 0\le \ell\le m_r-1.
\]
By the tower property and \eqref{eqn:doubling-conditional-expectation},
\[
    \E_{rk+\ell qk}\big[d_\ell^{\,r}\big]=0.
\]
Thus the partial sums of \(d_\ell^{\,r}\) form a martingale with respect to the sampled filtration, with increments bounded by \(1\).

Applying Azuma's inequality \eqref{eqn:azuma} to the rescaled martingale sum over this residue class, with the choice
\[
    \delta'=\frac{\delta}{2}\sqrt{m_r},
\]
gives
\begin{align}
    \P\Bigg(
    \left|
    \frac{1}{m_r}\sum_{\ell=0}^{m_r-1}
    \E_{rk+(\ell+1)qk}
    \bigl[
    F\bigl(2^{rk+\ell qk}x\bigr)
    \bigr]
    \right|
    >
    \frac{\delta}{2}
    \Bigg)
    \le
    2\exp(-c\delta^2 m_r).
    \notag
\end{align}
Combining this with the conditional-expectation approximation gives
\begin{align}
    \P\Bigg(
    \left|
    \frac{1}{m_r}
    \sum_{\ell=0}^{m_r-1}F\bigl(2^{(\ell q+r)k}x\bigr)
    \right|
    >
    \delta
    \Bigg)
    \le
    2\exp(-c\delta^2 m_r).
    \notag
\end{align}
Taking a union bound over \(0\le r\le q-1\), we obtain
\begin{align}
    \P\Bigg[
    x\in\mathbb T:
    \left|
    \frac{1}{m}\sum_{j=0}^{m-1}F(2^{jk}x)
    \right|
    >
    \delta
    \Bigg]
    \le
    \sum_{r=0}^{q-1}2\exp(-c\delta^2 m_r).
    \notag
\end{align}
Since \(m_r\ge m/q-1\), if
\[
    m>\max\left\{2q,\frac{4q\log(2q)}{c\delta^2}\right\},
\]
then \(m_r\ge m/(2q)\) for all \(r\), and the prefactors are absorbed into the exponential with a relaxed exponent. Therefore,
\begin{align}
    \P\Bigg[
    x\in\mathbb T:
    \left|
    \frac{1}{m}\sum_{j=0}^{m-1}F(2^{jk}x)
    \right|
    >
    \delta
    \Bigg]
    \le
    \exp\left(-c\delta^2\frac{m}{4q}\right).
    \label{eqn:sparse-average-general-q}
\end{align}
If \(T\ge k\), then \(q\le 2T/k\le C\log(K/\delta)/k\), and hence \eqref{eqn:sparse-average-general-q} implies
\begin{align}
    \P\Bigg[
    x\in\mathbb T:
    \left|
    \frac{1}{m}\sum_{j=0}^{m-1}F(2^{jk}x)
    \right|
    >
    \delta
    \Bigg]
    \le
    \exp\left(
    -c'\delta^2\frac{km}{\log(K/\delta)}
    \right).
    \notag
\end{align}

In the range \(k\ge T\), we have \(q=1\). Thus no residue-class decomposition is needed, and \eqref{eqn:sparse-average-general-q} reduces to
\begin{align}
    \P\Bigg[
    x\in\mathbb T:
    \left|
    \frac{1}{m}\sum_{j=0}^{m-1}F(2^{jk}x)
    \right|
    >
    \delta
    \Bigg]
    \le
    \exp\left(-c\delta^2m\right),
    \notag
\end{align}
after increasing the lower bound on \(m\) and decreasing \(c>0\) if necessary. This proves \eqref{eqn:large-deviation-sparse-dyadic-average}.
\end{proof}

\begin{lemma}\label{lem:Tn-upp-ldt714}
Let \(I\) and \(c_I\) be as in \eqref{eqn:cI-ass}. There is a full-measure set
\(\Omega_1=\Omega_1(I)\subset\T\), independent of \(\eps\), such that, for every \(x\in\Omega_1\) and every \(\eps>0\), there exists \(n_0=n_0(x,I,\eps)\) such that, for all \(E\in I\) and all \(n\ge n_0\), the following hold:
\begin{enumerate}[(a)]
    \item For every \(0\le j\le n^2\),
    \begin{align}
        \frac{1}{n}\log\|T_n^E(2^jx)\|
        \le
        L(E)+4\eps c_I.
        \label{eqn:Tn-upper717}
    \end{align}

    \item There exists \(p\in[n^4,2n^4-1]\) such that
    \begin{align}
        \frac{1}{n}\log\|T_n^E(2^p x)\|
        \ge
        L_n(E)-3\eps c_I.
        \label{eqn:Tn-lower-nonunif}
    \end{align}
\end{enumerate}
\end{lemma}

\begin{proof}
Following the notation in \eqref{eqn:tn-Ln}, denote
\begin{align} \notag 
    t_k(x;E)=\frac{1}{k}\log\|T_k^E(x)\|,
    \qquad
    L_k(E)=\E[t_k].
\end{align}
As in \eqref{eqn:tn-prime}, there exists \(C>0\) such that, for all \(x\) and \(E\),
\begin{align}
    |\partial_E t_k(x;E)|\le e^{Ck},
    \qquad
    |\partial_x t_k(x;E)|\le 2^k e^{Ck}\le e^{Ck},
    \label{eqn:tk-partial-E-x}
\end{align}
after increasing \(C\) if necessary.

\noindent\(\bullet\) {\bf The uniform upper bound \eqref{eqn:Tn-upper717}.}

Applying \eqref{eqn:large-deviation-sparse-dyadic-average} to a rescaled version of \(F(x)=t_k(x;E)\), we take \(K=e^{Ck}\). The deviation threshold is \(\delta=\eps c_I\), with \(\eps>0\). Hence, for \(k>\log((\eps c_I)^{-1})\),
\[
    \log(K/\delta)\le Ck+k=(C+1)k,
\]
and therefore
\[
    \min\left\{
    1,\frac{k}{\log(K/\delta)}
    \right\}
    \ge c>0.
\]
Thus, for \(m\ge C_{\eps,I}\), \eqref{eqn:large-deviation-sparse-dyadic-average} gives, for each fixed \(j_0\),
\begin{align}
&\mathbb P\Biggl[
x\in\mathbb T:
\left|
L_k(E)
-
\frac{1}{m}\sum_{s=0}^{m-1}
t_k\bigl(2^{sk+j_0}x;E\bigr)
\right|
>
\eps c_I
\Biggr] \notag\\
&\qquad =
\mathbb P\Biggl[
x\in\mathbb T:
\left|
L_k(E)
-
\frac{1}{m}\sum_{s=0}^{m-1}
t_k\bigl(2^{sk}x;E\bigr)
\right|
>
\eps c_I
\Biggr] \notag\\
&\qquad \le
\exp\bigl(-c(\eps c_I)^2m\bigr)
=
\exp(-c_{\eps,I}m).
\notag
\end{align}

In particular, taking \(m=k^4\), we obtain
\begin{align}
\mathbb P\Biggl[
x\in\mathbb T:
\left|
L_k(E)
-
\frac{1}{k^4}\sum_{s=0}^{k^4-1}
t_k\bigl(2^{sk+j_0}x;E\bigr)
\right|
>
\eps c_I
\Biggr]
\le
\exp(-c_{\eps,I}k^4),
\label{eqn:sparse-tk-average-ldt}
\end{align}
where \(c_{\eps,I}>0\) is independent of \(k\), \(E\), and \(j_0\).

We now make the estimate uniform in \(E\in I\). Split \(I\) into intervals of length \(\eps c_I e^{-Ck}\), where \(e^{Ck}\) is the bound for \(\partial_E t_k\) in \eqref{eqn:tk-partial-E-x}, and denote their left endpoints by \(e_1,\ldots,e_M\), with \(M\le Ce^{Ck}/(\eps c_I)\).

Then, for any \(E\in I\), let \(e_\alpha\) be the closest endpoint, so that \(|E-e_\alpha|\le \eps c_I e^{-Ck}\). If
\[
\left|
L_k(e_\alpha)
-
\frac{1}{k^4}\sum_{s=0}^{k^4-1}
t_k\bigl(2^{sk+j_0}x;e_\alpha\bigr)
\right|
<
\eps c_I,
\]
then, by the derivative bound for \(\partial_E t_k\),
\begin{align}
\Bigg|
L_k(E)
-
\frac{1}{k^4}\sum_{s=0}^{k^4-1}
t_k\bigl(2^{sk+j_0}x;E\bigr)
\Bigg|
&<
\eps c_I
+
\big|L_k(E)-L_k(e_\alpha)\big| \notag\\
&\quad
+
\frac{1}{k^4}\sum_{s=0}^{k^4-1}
\left|
t_k\bigl(2^{sk+j_0}x;E\bigr)
-
t_k\bigl(2^{sk+j_0}x;e_\alpha\bigr)
\right| \notag\\
&\le
\eps c_I+2e^{Ck}|E-e_\alpha| \notag\\
&\le
3\eps c_I.
\notag
\end{align}

Hence, for each fixed \(j_0\),
\begin{align}
&\Biggl\{
x\in\T:
\sup_{E\in I}
\left|
L_k(E)
-
\frac{1}{k^4}\sum_{s=0}^{k^4-1}
t_k\bigl(2^{sk+j_0}x;E\bigr)
\right|
\ge
3\eps c_I
\Biggr\} \notag\\
&\qquad\subset
\bigcup_{\alpha=1}^{M}
\Biggl\{
x\in\T:
\left|
L_k(e_\alpha)
-
\frac{1}{k^4}\sum_{s=0}^{k^4-1}
t_k\bigl(2^{sk+j_0}x;e_\alpha\bigr)
\right|
>
\eps c_I
\Biggr\}.
\notag
\end{align}

Now define the exceptional set
\begin{align}
\mathcal E_k
=
\Biggl\{
x\in\T:
\sup_{E\in I}
\max_{0\le j_0\le k^{10}}
\left|
L_k(E)
-
\frac{1}{k^4}\sum_{s=0}^{k^4-1}
t_k\bigl(2^{sk+j_0}x;E\bigr)
\right|
\ge
3\eps c_I
\Biggr\}. \notag 
\end{align}
Using \eqref{eqn:sparse-tk-average-ldt}, the union bound over \(0\le j_0\le k^{10}\), and the energy net above, we obtain
\begin{align}
    \P(\mathcal E_k)
    &\le
    (k^{10}+1)
    \frac{C}{\eps c_I}e^{Ck}
    \exp(-c_{\eps,I}k^4) \le
    e^{-k} \notag 
\end{align}
for all \(k\ge k_0(\eps,I)\), uniformly in \(E\in I\). In particular, \(\sum_k \P(\mathcal E_k)<\infty\).

Then
\begin{align}
    \Omega_1=\T\setminus \limsup_{k\ge k_0(\eps,I)} \mathcal E_k  \notag 
\end{align}
has full measure. Hence, for any \(x\in \Omega_1\), there exists \(k_0=k_0(x,\eps,I)\) such that, for all \(k\ge k_0\), all \(E\in I\), and all \(0\le j_0\le k^{10}\),
\begin{align}
 \Bigg|
 L_k(E)
 -
 \frac{1}{k^4}\sum_{s=0}^{k^4-1}
 t_k\bigl(2^{sk+j_0}x;E\bigr)
 \Bigg|
 \le 3\eps c_I.
 \label{eqn:729}
\end{align}

Using \eqref{eqn:L-Ln-asym}, after increasing \(k_0\) if necessary, we may also assume that
\[
    L_k(E)\le L(E)+\eps c_I,
    \qquad E\in I.
\]
Therefore, for any \(0\le j_0\le k^{10}\),
\begin{align}
 \frac{1}{k^4}\sum_{s=0}^{k^4-1}
 t_k\bigl(2^{sk+j_0}x;E\bigr)
 \le
 L_k(E)+3\eps c_I
 \le
 L(E)+4\eps c_I.
 \label{eqn:sparse-average-upper}
\end{align}

For any \(n=mk\) and any \(0\le j_0\le n^2\), submultiplicativity gives
\begin{align}
    t_n(2^{j_0}x;E)
    &=
    \frac{1}{n}\log \|T_n^E(2^{j_0}x)\| \notag\\
    &=
    \frac{1}{n}
    \log \left\|
    \prod_{s=0}^{m-1}
    T_k^E(2^{sk+j_0}x)
    \right\| \notag\\
    &\le
    \frac{1}{m}\sum_{s=0}^{m-1}
    \frac{1}{k}\log \|T_k^E(2^{sk+j_0}x)\| \notag\\
    &=
    \frac{1}{m}\sum_{s=0}^{m-1}
    t_k(2^{sk+j_0}x;E).
    \label{eqn:tN-tn}
\end{align}

Now take \(m=k^4\) and \(n=mk=k^5\). Then \(n^2=k^{10}\), and combining \eqref{eqn:tN-tn} with \eqref{eqn:sparse-average-upper} gives, for every \(0\le j_0\le n^2\), every \(x\in\Omega_1\), and every \(k\ge k_0(x,\eps,I)\),
\begin{align}
    \frac{1}{n}\log\|T_n^E(2^{j_0}x)\|
    =
    t_n(2^{j_0}x;E)
    \le
    L(E)+4\eps c_I. \notag
\end{align}

\noindent\(\bullet\) {\bf The sampled lower bound \eqref{eqn:Tn-lower-nonunif}.}

For the lower bound, we use the same argument with sampling gap \(1\). Applying \eqref{eqn:large-deviation-sparse-dyadic-average} to a rescaled version of \(F(x)=t_k(x;E)\), with sampling gap \(1\) and sample length \(m\), we estimate the deviation set
\begin{align}
    \mathbb P\Biggl[
x\in\mathbb T:
\left|
L_k(E)
-
\frac{1}{m}\sum_{s=0}^{m-1}
t_k\bigl(2^{s+j_0}x;E\bigr)
\right|
>
\eps c_I
\Biggr].
\notag
\end{align}
Taking \(m=k^4\), and then repeating the same energy-net and Borel--Cantelli argument as above, we may take the same full-measure set \(\Omega_1\), after intersecting with the previous full-measure set if necessary, so that the following also holds. For any \(x\in\Omega_1\), there exists \(k_0=k_0(x,\eps,I)\) such that, for all \(k\ge k_0\), all \(E\in I\), and all \(0\le j_0\le k^{10}\),
\begin{align}
 \Bigg|
 L_k(E)
 -
 \frac{1}{k^4}\sum_{s=0}^{k^4-1}
 t_k\bigl(2^{s+j_0}x;E\bigr)
 \Bigg|
 \le 3\eps c_I.
 \label{eqn:consecutive-average-tk}
\end{align}
In particular, using only the lower side of \eqref{eqn:consecutive-average-tk}, we have
\begin{align}
    L_k(E)
    -
    \frac{1}{k^4}\sum_{s=0}^{k^4-1}
    t_k\bigl(2^{s+j_0}x;E\bigr)
    \le
    3\eps c_I.
    \notag
\end{align}
Taking \(j_0=k^4\), this gives
\begin{align}
    L_k(E)
    -
    \frac{1}{k^4}\sum_{s=k^4}^{2k^4-1}
    t_k\bigl(2^{s}x;E\bigr)
    \le
    3\eps c_I.
    \notag
\end{align}
Therefore, there exists \(p\in[k^4,2k^4-1]\) such that
\begin{align}
    t_k\bigl(2^p x;E\bigr)
    \ge
    L_k(E)-3\eps c_I.
    \notag
\end{align}
This proves \eqref{eqn:Tn-lower-nonunif}, with \(k\) in place of \(n\), after increasing \(n_0=k_0\) if necessary.
\end{proof}


A consequence of Cramer's rule and the uniform upper bound in
\eqref{eqn:Tn-upper717} is the following Green's function estimate.

\begin{lemma}\label{lem:green-bound-transfer}
Let \(I\) and \(c_I\) be as in \eqref{eqn:cI-ass}, and let \(\Omega_1\) be the full-measure set in Lemma~\ref{lem:Tn-upp-ldt714}. For every \(x\in\Omega_1\) and every \(\eps>0\), there exists \(n_0=n_0(x,I,\eps)\) such that, for all \(n\ge n_0\), all \(E\in I\), and all \(0\le n_1\le n_2\le n-1\),
\begin{align}
    \bigl|G_n(x;E)(n_1,n_2)\bigr|
    \le
    \frac{|P_n(x)|}{|D_n^E(x)|}
    \exp\!\Bigl(
        nL(E)-|n_2-n_1|L(E)+10\eps c_I n
    \Bigr),
    \label{eqn:green-numer}
\end{align}
where \(D_n^E\) is as in \eqref{eqn:Dn} and
\[
    P_n(x):=\prod_{j=0}^{n-1}a_j(x).
\]
As a consequence, for one of the pairs
\begin{align}
(\widetilde n,\widetilde x)\in \Lambda \overset{\text{def}}{=}
\Big\{
(n,x),\ (n-1,x),\ (n-1,2x),\ (n-2,2x)
\Big\},
\label{eqn:ntildextilde-choice}
\end{align}
we have
\begin{align}
\bigl|G_{\widetilde n}(\widetilde x;E)(n_1,n_2)\bigr|
\le
\exp\!\Bigl(
-|n_2-n_1|L(E)
+
\bigl[nL(E)-\log\|T_n^E(x)\|\bigr]
+
11\eps c_I n
\Bigr),
\label{eqn:green-bound-tn}
\end{align}
for all admissible \(n_1,n_2\) for the Green's function on the corresponding interval.
\end{lemma}

\begin{proof}
Throughout the proof, we consider \(x\in\Omega_1\) and \(n\ge n_0\), where \(n_0\) is large enough so that \eqref{eqn:Tn-upper717} holds at all scales used below. We use the maximum matrix norm; the equivalence with other matrix norms only changes constants, which are absorbed into the error terms by increasing \(n_0\) if necessary.

It suffices to consider \(0\le n_1\le n_2\le n-1\). Recall that we use the \(0\)-based convention \(H_n=H_{[0,n-1]}\). We first estimate the two determinants in the numerator of Cramer's rule \eqref{eqn:cramer}. By \eqref{eqn:Tn-full-det-formula}, we have, for any \(k>0\),
\[
    |D_k^E(y)|
    \le
    |P_k(y)|\,\|T_k^E(y)\|.
\]

If \(n_1\ge n_0\), then by \eqref{eqn:Tn-upper717},
\begin{align}
    |D_{n_1}^E(x)|
    \le
    |P_{n_1}(x)|\,\|T_{n_1}^E(x)\|
    \le
    |P_{n_1}(x)|
    \exp\!\Bigl(n_1L(E)+4\eps c_I n\Bigr).
    \notag
\end{align}
If \(n_1<n_0\), then
\[
    \|T_{n_1}^E(x)\|\le e^{Cn_0}.
\]
By enlarging the final threshold for \(n\) so that \(n\ge Cn_0/(\eps c_I)\), we have \(e^{Cn_0}\le e^{\eps c_I n}\), and hence the desired bound for \(D_{n_1}^E(x)\) follows.

Hence, for all \(0\le n_1\le n-1\),
\begin{align}
    |D_{n_1}^E(x)|
    \le
    |P_{n_1}(x)|
    \exp\!\Bigl(n_1L(E)+4\eps c_I n\Bigr).
    \label{eqn:left-det-green-bound}
\end{align}

Next set
\[
    \ell=n-n_2-1.
\]
If \(\ell\ge n_0\), then, assuming for simplicity that \(\sqrt n\) is an integer,
\[
    (\ell+\sqrt n)^2\ge n\ge n_2+1.
\]
If \(\sqrt n\) is not an integer, we replace it by \(\lceil\sqrt n\rceil\), which does not affect the estimates below. Thus \eqref{eqn:Tn-upper717} may be applied to \(T_{\ell+\sqrt n}^E(2^{n_2+1}x)\). By the cocycle identity,
\[
    T_{\ell+\sqrt n}^E(2^{n_2+1}x)
    =
    T_{\sqrt n}^E(2^n x)\,
    T_\ell^E(2^{n_2+1}x),
\]
and hence
\[
    \|T_\ell^E(2^{n_2+1}x)\|
    \le
    \bigl\|[T_{\sqrt n}^E(2^n x)]^{-1}\bigr\|
    \,
    \|T_{\ell+\sqrt n}^E(2^{n_2+1}x)\|.
\]
Since \(T_{\sqrt n}^E(\cdot)\in{\rm SL}_2(\mathbb R)\) and the one-step matrices are uniformly bounded for \(E\in I\), we have uniformly in the phase
\[
    \bigl\|[T_{\sqrt n}^E(2^n x)]^{-1}\bigr\|
    \le
    e^{C\sqrt n}.
\]
Therefore,
\begin{align}
    \|T_\ell^E(2^{n_2+1}x)\|
    &\le
    e^{C\sqrt n}
    \exp\!\Bigl((\ell+\sqrt n)\bigl(L(E)+4\eps c_I\bigr)\Bigr) \notag\\
    &\le
    \exp\!\Bigl(\ell L(E)+5\eps c_I n\Bigr),
    \notag
\end{align}
after increasing \(n_0\) if necessary. If \(\ell<n_0\), the same bound follows from the crude estimate
\[
    \|T_\ell^E(2^{n_2+1}x)\|\le e^{Cn_0}\le e^{\eps c_I n}.
\]
Hence
\begin{align}
    |D_{n-n_2-1}^E(2^{n_2+1}x)|
    \le
    |P_{n-n_2-1}(2^{n_2+1}x)|
    \exp\!\Bigl((n-n_2-1)L(E)+5\eps c_I n\Bigr).
    \label{eqn:right-det-green-bound}
\end{align}

Combining \eqref{eqn:cramer}, \eqref{eqn:left-det-green-bound}, and \eqref{eqn:right-det-green-bound}, and using \(a_-\le a_j(x)\le a_+\), give
\begin{align}
    \bigl|G_n(x;E)(n_1,n_2)\bigr|
    &\le
    \frac{|P_n(x)|}{|D_n^E(x)|}
    \exp\!\Bigl(
        nL(E)-|n_2-n_1|L(E)+10\eps c_I n
    \Bigr),
    \notag
\end{align}
where the harmless constants coming from the missing boundary factors in the products have been absorbed into \(e^{\eps c_I n}\) for all sufficiently large \(n\). This proves \eqref{eqn:green-numer}.

It remains to derive \eqref{eqn:green-bound-tn}. From \eqref{eqn:Tn-full-det-formula}, at least one of the four determinants
\[
    D_n^E(x),\qquad
    D_{n-1}^E(x),\qquad
    D_{n-1}^E(2x),\qquad
    D_{n-2}^E(2x)
\]
has size comparable to \(|P_n(x)|\,\|T_n^E(x)\|\). More precisely,
\begin{align}
\max_{(\widetilde n,\widetilde x)\in
 \Lambda}
\bigl|D_{\widetilde n}^E(\widetilde x)\bigr|
\ge \min\{1,
 a_+^{-2}\}\,|P_n(x)|\,\|T_n^E(x)\|.
\label{eqn:det-max-lower}
\end{align}
Moreover,
\begin{align}
P_{n-1}(x)&=\frac{P_n(x)}{a_{n-1}(x)}, \notag\\
P_{n-1}(2x)&=\prod_{j=0}^{n-2}a_j(2x)=\prod_{j=1}^{n-1}a_j(x)=\frac{P_n(x)}{a_0(x)}, \notag\\
P_{n-2}(2x)&=\prod_{j=0}^{n-3}a_j(2x)=\prod_{j=1}^{n-2}a_j(x)=\frac{P_n(x)}{a_0(x)a_{n-1}(x)}.
\label{eqn:P-shift-relations}
\end{align}
Therefore, for any \((\widetilde n,\widetilde x)\in\Lambda\),
\begin{align}
\frac{|P_{\widetilde n}(\widetilde x)|}{|P_n(x)|}
\le
a_-^{-2}.
\label{eqn:P-ratio-wt}
\end{align}

Choose \((\widetilde n,\widetilde x)\in\Lambda\) for which \eqref{eqn:det-max-lower} holds. Applying \eqref{eqn:green-numer} at the scale \(\widetilde n\), and using \eqref{eqn:det-max-lower} and \eqref{eqn:P-ratio-wt}, we obtain
\begin{align}
\bigl|G_{\widetilde n}(\widetilde x;E)(n_1,n_2)\bigr|
&\le
\max\{1,a_+^2\}a_-^{-2}
\|T_n^E(x)\|^{-1}
\exp\!\Bigl(
\widetilde nL(E)-|n_2-n_1|L(E)+10\eps c_I n
\Bigr) \notag\\
&\le
\exp\!\Bigl(
-|n_2-n_1|L(E)
+
\bigl[nL(E)-\log\|T_n^E(x)\|\bigr]
+
11\eps c_I n
\Bigr),
\notag
\end{align}
after absorbing the fixed constant \(\max\{1,a_+^2\}a_-^{-2}\) into \(e^{\eps c_I n}\). This proves \eqref{eqn:green-bound-tn}.
\end{proof}

\subsection{A strengthened Lower bound of the transfer matrices}

The previous subsection gives two complementary estimates: the uniform upper bound \eqref{eqn:Tn-upper717} and the sampled lower bound \eqref{eqn:Tn-lower-nonunif}. The sampled lower bound gives good transfer-matrix growth at one suitably chosen point along the orbit. For the Green's function estimates, we need a stronger form of this information along a consecutive block of shifts, so that good finite-volume boxes can be placed throughout the interval under consideration. In this subsection we obtain this blockwise lower bound by ruling out the exceptional situation in which a finite-volume Green's function is very large while the corresponding transfer matrix remains too small at a later shift.

\begin{lemma}\label{lem:bad-set-BN}
Let \(I\) and \(c_I\) be as in \eqref{eqn:cI-ass}. Fix $\epsilon>0$ and set \(\bar N=\lfloor e^{(\log N)^2}\rfloor\). Define \(\mathcal B_N\) to be the collection of \(x\in\T\) such that there exist
\[
    N_1\in [N^{12},3N^{12}],\qquad
    k\in[\bar N,2\bar N],
    \qquad
    E\in I,
\]
for which
\begin{align}
    \|G_{N_1}(x;E)\|>e^{N^2}
    \label{eqn:Gn1744}
\end{align}
and
\begin{align}
    \frac{1}{N}\log \|T_N^E(2^k x)\|
    <
    L(E)-4\eps c_I.
    \label{eqn:TnE745}
\end{align}
Then there exist \(N_0=N_0(I,\eps)\) and \(c=c(I,\eps)>0\) such that, for all \(N\ge N_0\),
\begin{align}
    \P(\mathcal B_N)\le e^{-cN}.
    \label{eqn:mes-Bn}
\end{align}
\end{lemma}

The argument follows the corresponding elimination step in \cite[\S9.2]{BourgainSchlag2000CMP} for Schr\"odinger operators generated by the doubling map. We rewrite the proof that fits in our context, with the scale choices and large-deviation estimates used here, and include  more details for the reader's convenience.

In particular, we use the same order of scales \(N_1\sim N^{12}\) and \(\bar N\sim e^{(\log N)^2}\) as in \cite[Proposition~9.2]{BourgainSchlag2000CMP}. The particular power \(N^{12}\) is not essential. In the applications below, any sufficiently large fixed power greater than \(3\) would be enough, while still having \(N_1^2\ll \bar N\), so that the \(x\)-mesh of size \(e^{-N_1^2}\) is much coarser than the dyadic scale \(2^{-k}\) for \(k\in[\bar N,2\bar N]\).

\begin{proof}
Suppose \(x\in\mathcal B_N\). Then there exist \(N_1\in[N^{12},2N^{12}]\), \(k\in[\bar N,2\bar N]\), and \(E\in I\) such that \eqref{eqn:Gn1744} and \eqref{eqn:TnE745} hold. Write
\[
    \sigma(H_{N_1}(x))
    =
    \big\{\ \mu_1(x)\le\mu_2(x)\le\cdots\le\mu_{N_1}(x)\ \big\}.
\]
Since
\[
    \|G_{N_1}(x;E)\|^{-1}
    =
    \dist(E,\sigma(H_{N_1}(x))),
\]
the bound \eqref{eqn:Gn1744} implies that, for some \(1\le j\le N_1\),
\begin{align}
    |E-\mu_j(x)|\le e^{-N^2}.
    \label{eqn:mu-j-E}
\end{align}

By the finite-scale energy derivative bound \eqref{eqn:tn-prime}, for \(N\) sufficiently large,
\begin{align}
    \left|
    \frac{1}{N}\log\|T_N^E(2^kx)\|
    -
    \frac{1}{N}\log\|T_N^{\mu_j(x)}(2^kx)\|
    \right|
    \le e^{CN}|E-\mu_j(x)|
    \le
    \frac{1}{4}\eps c_I.
    \notag
\end{align}
Similarly, \( \left|
    L_N(E)-L_N(\mu_j(x))
    \right|
    \le
     \eps c_I/4.\) 
Moreover, using \eqref{eqn:L-Ln-asym}, after increasing \(N_0(I,\eps)\) if necessary, we may assume uniformly for all energies  \(E'\in I\) that \(|L_N(E')-L(E')|\le  \eps c_I/4\). 
Consequently, 
\[
    \left|
    L(E)-L(\mu_j(x))
    \right|
    \le
    \frac{3}{4}\eps c_I.
\]
Equivalently, this estimate may also be obtained from the H\"older continuity of \(L(\cdot)\), after choosing uniform H\"older parameters on the compact interval \(I\).

Thus, from \eqref{eqn:TnE745}, we obtain
\begin{align}
    \frac{1}{N}\log\|T_N^{\mu_j(x)}(2^kx)\|
    &<
    L(E)-4\eps c_I+\frac{1}{4}\eps c_I \notag\\
    &\le
    L(\mu_j(x))-3\eps c_I, \label{eqn:Tn-muj-2k-x}
\end{align}
after increasing \(N_0(I,\eps)\) if necessary.

Next we freeze the \(x\)-dependent eigenvalue on a sufficiently fine \(x\)-mesh. For each fixed \(N_1\in[N^{12},2N^{12}]\), divide \(\T\) into intervals
\[
    J_\ell=[x_{\ell-1},x_\ell],
    \qquad
    |J_\ell|\sim e^{-N_1^2},
    \qquad
    \ell=1,\ldots,M,
\]
where \(M\sim e^{N_1^2}\). A finer mesh would also work, as long as it beats the exponential growth in \(N_1\) and remains coarser than the dyadic scale \(2^{-k}\) used below. Note that the mesh size \(|J_{\ell}|\) and the number of intervals \(M\) cancel in the final estimate, since
\[
    \sum_{\ell=1}^{M}|J_\ell|\le 1.
\]
Thus this fine mesh does not introduce any additional loss in the final bound for \(\P(\mathcal B_N)\).

Let \(x_\ell\) denote the left endpoint of \(J_\ell\). If \(x\in J_\ell\), then, since \(a(\cdot)\in C^1\) and \(N_1\sim N^{12}\), a direct estimate of the matrix norm gives, for all sufficiently large \(N\),
\[
    \|H_{N_1}(x)-H_{N_1}(x_\ell)\|
    \le
    C2^{N_1}|x-x_\ell|
    \le
    C2^{N_1}e^{-N_1^2}
    \le
    e^{-N^2}.
\]
Therefore, by spectral continuity, there exists an eigenvalue
\(\mu_{j'}(x_\ell)\in\sigma(H_{N_1}(x_\ell))\) such that
\[
    |\mu_j(x)-\mu_{j'}(x_\ell)|
    \le
    \|H_{N_1}(x)-H_{N_1}(x_\ell)\|
    \le
    e^{-N^2}.
\]

Using again the finite-scale energy derivative bound and the finite-scale comparison \eqref{eqn:L-Ln-asym}, in the same way as in \eqref{eqn:mu-j-E}--\eqref{eqn:Tn-muj-2k-x}, we may replace \(\mu_j(x)\) by \(\mu_{j'}(x_\ell)\), losing at most another harmless fraction of \(\eps c_I\). Hence, for \(N\ge N_0(I,\eps)\),
\[
    2^kx\in S_N(\mu_{j'}(x_\ell)),
\]
where
\begin{align}
    S_N(E')
    :=
    \left\{
    \widetilde x\in\T:
    \frac{1}{N}\log\|T_N^{E'}(\widetilde x)\|
    <
    L(E')-2\eps c_I
    \right\}.
    \label{eqn:SN-def}
\end{align}
By the large deviation estimate in Theorem~\ref{thm:ldt-rhon} and Remark~\ref{rem:ldt-equivalent-centers}, there exists \(c=c(\eps,I)>0\) such that, uniformly for \(E'\in I\) and \(N\ge N_0(I,\eps)\),
\begin{align}
    \P(S_N(E'))\le e^{-cN}.
    \label{eqn:754}
\end{align}

Consequently,
\begin{align}
\P(\mathcal B_N)
&\le
\sum_{k=\bar N}^{2\bar N}
\sum_{N_1=N^{12}}^{3N^{12}}
\sum_{\ell=1}^{M}
\sum_{\mu_j(x_\ell)\in\sigma(H_{N_1}(x_\ell))\cap I}
\P\bigl(
x\in J_\ell:\ 2^kx\in S_N(\mu_j(x_\ell))
\bigr).
\label{eqn:BN-union-frozen}
\end{align}

We now estimate each term in the last sum. Since
\[
    k\sim \bar N=\lfloor e^{(\log N)^2}\rfloor>N_1^3
\]
for all sufficiently large \(N\), we have
\[
    2^{-k}\le e^{-N_1^2}\sim |J_\ell|.
\]
Thus we can cover \(J_\ell\) by \(\widetilde M\) dyadic intervals \(J_{\ell,j}\) of length \(2^{-k}\), where
\[
    \widetilde M
    \le
    |J_\ell|2^k+2
    \le
    |J_\ell|2^{k+1}.
\]
For a fixed energy \(E'\), by a change of variable \(t=2^kx\) on each \(J_{\ell,j}\), we have
\begin{align}
    \P\bigl(x\in J_{\ell,j}: 2^kx\in S_N(E')\bigr)
    =
    \int_{J_{\ell,j}}\one_{S_N(E')}(2^kx)\,dx
    &=
    2^{-k}\int_{\T}\one_{S_N(E')}(t)\,dt \notag\\
    &=
    2^{-k}\P(S_N(E')).
    \notag
\end{align}
Hence,
\begin{align}
 \P\bigl(x\in J_\ell: 2^kx\in S_N(E')\bigr)
 &\le
 \sum_{j=1}^{\widetilde M}
 \P\bigl(x\in J_{\ell,j}: 2^kx\in S_N(E')\bigr) \notag\\
 &\le
 (|J_\ell|2^k+2)\,2^{-k}\P(S_N(E')) \notag\\
 &\le
 2|J_\ell|\P(S_N(E')).
 \label{eqn:local-change-variable}
\end{align}

For each fixed \(N_1\) and \(x_\ell\), the number of eigenvalues
\(\mu_j(x_\ell)\in\sigma(H_{N_1}(x_\ell))\cap I\) is trivially bounded by \(N_1\). Combining \eqref{eqn:BN-union-frozen}, \eqref{eqn:754}, and \eqref{eqn:local-change-variable}, we obtain
\begin{align}
\P(\mathcal B_N)
&\le
2
\sum_{k=\bar N}^{2\bar N}
\sum_{N_1=N^{12}}^{3N^{12}}
\sum_{\ell=1}^{M}
\sum_{\mu_j(x_\ell)\in\sigma(H_{N_1}(x_\ell))\cap I}
|J_\ell|\,e^{-cN} \notag\\
&\le
2
\sum_{k=\bar N}^{2\bar N}
\sum_{N_1=N^{12}}^{3N^{12}}
N_1
\left(\sum_{\ell=1}^{M}|J_\ell|\right)
e^{-cN} \notag\\
&\le
C\bar N\,N^{24}e^{-cN}.
\notag
\end{align}
Since \(\bar N=e^{(\log N)^2}=e^{o(N)}\), the last expression is bounded by \(e^{-c'N}\) for some \(c'=c'(\eps,I)>0\) and all sufficiently large \(N\). This proves \eqref{eqn:mes-Bn}.
\end{proof}

 We next show that the large-Green-function condition \eqref{eqn:Gn1744} occurs at suitable scales for spectral-a.e. energy. Once this is established, the exceptional-set estimate in Lemma~\ref{lem:bad-set-BN} will force the small-transfer condition \eqref{eqn:TnE745} to fail eventually on a full-measure set.

By a standard Shnol-type theorem, for spectral-measure a.e. \(E\in\sigma(H(x))\), there exists a generalized eigenfunction
\(\xi=\xi^E=(\xi_n)_{n\ge -1}\) satisfying
\begin{align}
    \xi_{-1}=0,\qquad \xi_0=1,
    \label{eqn:shnol-init}
\end{align}
and
\begin{align}
    (H_x-E)\xi=0,
    \label{eqn:shnol-eqn}
\end{align}
and there exist constants \(A_E>0\) such that
\begin{align}
    |\xi_n|\le A_E(1+n),
    \qquad n\ge0.
    \label{eqn:shnol-poly}
\end{align}
See \cite{shnol1954behavior,simon1981spectrum} for the original arguments in the differential setting, and \cite[Theorem~2.4.2]{DamFil2022ESO1} for a modern proof in the one-dimensional discrete setting.  In the present half-line setting, \(\delta_0\) is cyclic, so the relevant spectral measure can be taken to be the one associated with \(\delta_0\), and the generalized eigenfunction can be chosen as the canonical polynomial solution \(\xi_n^E=p_n(E)\), where \(p_n\) is defined by the cyclicity relation \(\delta_n=p_n(H(x))\delta_0\). The polynomial bound in \eqref{eqn:shnol-poly} follows from the standard argument. The polynomial power may be chosen to be any fixed number strictly larger than \(1/2\); we take it to be \(1\) for simplicity. With the normalization \(\xi_0=1\) in \eqref{eqn:shnol-init}, the prefactor \(A_E\) is in general \(E\)-dependent.

We now combine this normalized polynomially bounded solution with the finite-scale Green function estimates to obtain the large Green function lower bound needed below.
\begin{lemma}\label{lem:large-green-some-scale}
Let \(I\) and \(\Omega_1\) be as in Lemma~\ref{lem:Tn-upp-ldt714},, and let
\begin{align}
    \Omega_2=\bigcap_{m\ge0}2^{-m}\Omega_1 .
\end{align}
Then \(\Omega_2\) also has full measure. For every \(x\in\Omega_2\) and for spectral-a.e. \(E\in I\), there exists \(N_0=N_0(x,E,I)\) such that, for all \(N\ge N_0\), there exists \(N_1\in[N^{12},3N^{12}]\) satisfying
\begin{align}
    \|G_{N_1}(x;E)\|>e^{N^2}.
    \label{eqn:Gn-norm-lower}
\end{align}
\end{lemma}

\begin{proof}
Fix \(x\in\Omega_2\). Then \(2^m x\in\Omega_1\) for every \(m\ge0\). We also fix \(\eps>0\) sufficiently small, to be chosen below.

We first apply the sampled lower bound \eqref{eqn:Tn-lower-nonunif} of Lemma~\ref{lem:Tn-upp-ldt714} with \(n=N^3\). Thus, for all sufficiently large \(N\), there exists
\[
    p\in[N^{12},2N^{12}-1]
\]
such that
\begin{align}
    \frac{1}{N^3}\log\|T_{N^3}^E(2^p x)\|
    \ge
    L_{N^3}(E)-3\eps c_I.
    \label{eqn:sampled-lower-p}
\end{align}
Using \eqref{eqn:L-Ln-asym}, after increasing \(N_0(x,I,\eps)\) if necessary, we also have
\[
    |L_{N^3}(E)-L(E)|\le \eps c_I.
\]
Therefore,
\begin{align}
    N^3L(E)-\log\|T_{N^3}^E(2^p x)\|
    \le
    4\eps c_I N^3.
    \label{eqn:TN3-large-at-p}
\end{align}

Since \(x\in\Omega_2\), we have \(2^p x\in\Omega_1\). Applying \eqref{eqn:green-bound-tn} of Lemma~\ref{lem:green-bound-transfer} with scale \(N^3\) and base point \(2^p x\), we obtain one pair
\begin{align}
(N_2,p')
\in
\Big\{
(N^3,p),\ (N^3-1,p),\ (N^3-1,p+1),\ (N^3-2,p+1)
\Big\}
\notag
\end{align}
such that, for local indices \(0\le r,s\le N_2-1\),
\begin{align}
\bigl|G_{N_2}(2^{p'}x;E)(r,s)\bigr|
\le
\exp\!\Bigl(
-|r-s|L(E)
+
\bigl[N^3L(E)-\log\|T_{N^3}^E(2^p x)\|\bigr]
+
11\eps c_I N^3
\Bigr).
\notag
\end{align}
In global coordinates, \(G_{N_2}(2^{p'}x;E)\) is the Green's function on the interval
\[
    \Lambda=[p',p'+N_2-1]=[p',p''].
\]
Hence, for all \(r,s\in\Lambda\),
\begin{align}
\bigl|G_\Lambda(x;E)(r,s)\bigr|
\le
\exp\!\Bigl(
-|r-s|L(E)
+
\bigl[N^3L(E)-\log\|T_{N^3}^E(2^p x)\|\bigr]
+
11\eps c_I N^3
\Bigr).
\label{eqn:green-global-before-lower}
\end{align}
Combining \eqref{eqn:green-global-before-lower} with \eqref{eqn:TN3-large-at-p}, and using \(L(E)\ge c_I\), give
\begin{align}
\bigl|G_\Lambda(x;E)(r,s)\bigr|
\le
\exp\!\Bigl(
-c_I|r-s|+15\eps c_I N^3
\Bigr).
\label{eqn:green-global-decay}
\end{align}

Let
\[
    N_1=p'+\frac{N^3}{2}.
\]
For simplicity, we assume that \(N^3/2\) is an integer; otherwise we replace it by its integer part, which only changes the estimates below by harmless constants. Since \(p'\in\{p,p+1\}\), we have
\[
    N_1=p+\frac{N^3}{2}+O(1).
\]
In particular, for all sufficiently large \(N\),
\[
    N_1\in[N^{12},3N^{12}].
\]
Moreover, \(N_1\in\Lambda\), and
\[
    |N_1-p'|\ge \frac{N^3}{2}-1,
    \qquad
    |N_1-p''|\ge \frac{N^3}{2}-3.
\]
Choosing \(\eps>0\) sufficiently small and then taking \(N\) large, \eqref{eqn:green-global-decay} implies
\begin{align}
    \bigl|G_\Lambda(x;E)(N_1,p')\bigr|
    \le e^{-3N^2},
    \qquad
    \bigl|G_\Lambda(x;E)(N_1,p'')\bigr|
    \le e^{-3N^2}.
    \label{eqn:boundary-green-small}
\end{align}

Fix a spectral-a.e. \(E\in I\) for which the generalized eigenfunction \(\xi=\xi^E\) in \eqref{eqn:shnol-init}--\eqref{eqn:shnol-poly} exists. Restricting \eqref{eqn:shnol-eqn} to \(\Lambda=[p',p'']\), we obtain
\begin{align}
     (H_\Lambda(x)-E)
     \begin{pmatrix}
         \xi_{p'}\\
         \xi_{p'+1}\\
         \vdots\\
         \xi_{p''}
     \end{pmatrix}
     =
     \begin{pmatrix}
         a_{p'}\xi_{p'-1}\\
         0\\
         \vdots\\
         a_{p''+1}\xi_{p''+1}
     \end{pmatrix},
     \label{eqn:restricted-eigen-eqn-Lambda}
\end{align}
up to signs that are immaterial for the estimates. Solving for the \(N_1\)-th component gives
\begin{align}
    \xi_{N_1}
    =
    G_\Lambda(x;E)(N_1,p')a_{p'}\xi_{p'-1}
    +
    G_\Lambda(x;E)(N_1,p'')a_{p''+1}\xi_{p''+1}.
    \label{eqn:xi-N1-boundary}
\end{align}
Combining \(|a|\le a_+\), \eqref{eqn:boundary-green-small}, and \eqref{eqn:shnol-poly}, we get
\begin{align}
    |\xi_{N_1}|
    \le
    A_E(1+3N^{12})e^{-3N^2}
    \le
    e^{-2N^2}
    \label{eqn:xi-N1-small}
\end{align}
for all sufficiently large \(N\), where the lower threshold may depend on \(E\).

Finally, restrict \eqref{eqn:shnol-eqn} to \([0,N_1-1]\). Since \(\xi_{-1}=0\), the only boundary forcing comes from the right endpoint, and we get
\begin{align}
    \begin{pmatrix}
         \xi_0\\
         \xi_1\\
         \vdots\\
         \xi_{N_1-1}
    \end{pmatrix}
    =
    G_{N_1}(x;E)
    \begin{pmatrix}
         0\\
         0\\
         \vdots\\
         a_{N_1}\xi_{N_1}
    \end{pmatrix},
    \notag
\end{align}
again up to an irrelevant sign. Since \(\xi_0=1\), the norm of the vector on the left is at least \(1\). Hence
\begin{align}
    1
    \le
    a_+\|G_{N_1}(x;E)\|\,|\xi_{N_1}|. \notag
\end{align}
Using \eqref{eqn:xi-N1-small}, and increasing \(N_0(x,E,I)\) if necessary, this gives
\begin{align}
    \|G_{N_1}(x;E)\|\ge e^{N^2}. \notag
\end{align}
This proves \eqref{eqn:Gn-norm-lower}.
\end{proof}

By Lemma~\ref{lem:bad-set-BN}, the sets \(\mathcal B_N\) are summable. Hence, by the Borel--Cantelli lemma,
\begin{align}
    \Omega_3
    :=
    \Omega_2\setminus \limsup_{N\to\infty}\mathcal B_N
\end{align}
has full measure. Combining this with Lemma~\ref{lem:large-green-some-scale}, we obtain the desired strengthened lower bound for transfer matrices along the whole block of shifts.

\begin{corollary}\label{cor:Tn-lower}
Let \(I\) and \(c_I\) be as in \eqref{eqn:cI-ass}, and let \(\Omega_3\) be defined as above. For every \(x\in\Omega_3\) and for spectral-a.e. \(E\in I\), there exists \(N_0=N_0(x,E,I,\eps)\) such that, for all \(N\ge N_0\), with
\[
    \bar N=\lfloor e^{(\log N)^2}\rfloor
\]
as in Lemma~\ref{lem:bad-set-BN}, one has, for every \(k\in[\bar N,2\bar N]\),
\begin{align}
    \frac{1}{N}\log \|T_N^E(2^k x)\|
    \ge
    L(E)-4\eps c_I
    \ge
    (1-4\eps)L(E).
    \label{eqn:7.103}
\end{align}
\end{corollary}

\begin{proof}
Since \(x\in\Omega_3\), there exists \(N_0'=N_0'(x,I,\eps)\) such that
\[
    x\notin\mathcal B_N
\]
for all \(N\ge N_0'\). On the other hand, by Lemma~\ref{lem:large-green-some-scale}, for spectral-a.e. \(E\in I\), there exists \(N_0''=N_0''(x,E,I)\) such that, for every \(N\ge N_0''\), there is \(N_1\in[N^{12},3N^{12}]\) satisfying
\begin{align}
    \|G_{N_1}(x;E)\|>e^{N^2}. \notag
\end{align}

Set \(N_0=\max\{N_0',N_0''\}\). Suppose that, for some \(N\ge N_0\) and some \(k\in[\bar N,2\bar N]\),
\begin{align}
    \frac{1}{N}\log\|T_N^E(2^k x)\|
    &<
    L(E)-4\eps c_I.
    \label{eqn:Tn-lower-contradiction}
\end{align}
Then \eqref{eqn:Tn-lower-contradiction}, together with the large-Green-function condition above, implies that \(x\in\mathcal B_N\). This contradicts the choice of \(N_0'\), since \(N\ge N_0\ge N_0'\). Therefore, the first inequality in \eqref{eqn:7.103} holds for all \(N\ge N_0\) and all \(k\in[\bar N,2\bar N]\). The final inequality in \eqref{eqn:7.103} follows from the definition of \(c_I\) in \eqref{eqn:cI-ass}, which gives \(L(E)\ge c_I\) for all \(E\in I\).
\end{proof}


\subsection{Decay of the Green's function and the generalized eigenfunction}

We now complete the proof of Anderson localization on a fixed compact interval \(I\subsetneq(0,E_0)\). The first step is to prove the finite-volume Green's function decay estimate stated in Theorem~\ref{thm:ldt-green1}. This estimate is obtained by combining the strengthened transfer-matrix lower bound from Corollary~\ref{cor:Tn-lower} with the resolvent expansion on suitably placed good boxes.

\begin{proof}[Proof of Theorem~\ref{thm:ldt-green1}]
Fix \(0<\eps<1/100\), and let \(\Omega_3=\Omega_3(I,\eps)\) be the full-measure set from Corollary~\ref{cor:Tn-lower}. Fix \(x\in\Omega_3\). For spectral-a.e. \(E\in I\), Corollary~\ref{cor:Tn-lower} gives \(N_0=N_0(x,E,I,\eps)\) such that, for all \(N\ge N_0\), with
\[
    \bar N=\lfloor e^{(\log N)^2}\rfloor,
\]
one has
\begin{align}
    \frac{1}{N}\log \|T_N^E(2^k x)\|
    \ge
    L(E)-4\eps c_I
    \label{eqn:green-proof-transfer-lower}
\end{align}
for every \(k\in[\bar N,2\bar N]\). We will increase \(N_0\) several times below, without changing the notation.

We first record the local consequence of \eqref{eqn:green-proof-transfer-lower}. Let \(k\in[\bar N,2\bar N]\). Applying Lemma~\ref{lem:green-bound-transfer} at scale \(N\) and base point \(2^k x\), together with \eqref{eqn:green-proof-transfer-lower}, gives one pair
\begin{align}
(\widetilde N,\widetilde k)
\in
\Big\{
(N,k),\ (N-1,k),\ (N-1,k+1),\ (N-2,k+1)
\Big\}
\label{eqn:good-box-choice}
\end{align}
such that, in global coordinates, for
\[
    \widetilde \Lambda
    =
    [\widetilde k,\widetilde k+\widetilde N-1],
\]
one has
\begin{align}
\bigl|G_{\widetilde \Lambda}(x;E)(r,s)\bigr|
&\le
\exp\!\Bigl(
-|r-s|L(E)
+
\bigl[N L(E)-\log \|T_N^E(2^k x)\|\bigr]
+
11\eps c_I N
\Bigr) \notag\\
&\le
\exp\!\Bigl(
-|r-s|L(E)+15\eps L(E)N
\Bigr)
\label{eqn:local-green-good-box}
\end{align}
for all \(r,s\in\widetilde\Lambda\), where we also used \(L(E)\ge c_I\) on \(I\) in the last line.

We will repeatedly use the following elementary placement observation. For every \(n\in[\bar N,2\bar N]\), one can choose \(k\in[\bar N,2\bar N]\), and hence a corresponding interval \(\widetilde\Lambda\) as above, such that
\[
    n\in\widetilde\Lambda
    \qquad\text{and}\qquad
    \widetilde\Lambda\subset[\bar N,2\bar N].
\]
Moreover, every boundary point of \(\widetilde\Lambda\) that lies in \([\bar N,2\bar N]\) is separated from \(n\) by at least \(N/2-3\), i.e., 
\begin{align}
    \operatorname{dist}\!\left(
        n,\partial\widetilde\Lambda\cap[\bar N,2\bar N]
    \right)
    \ge \frac{N}{2}-3.
    \label{eqn:good-box-boundary-distance}
\end{align}
Indeed, in the bulk one chooses \(k\) with \(k= n-N/2+\mathcal O(1)\), while near the two endpoints one uses the one-sided choices \(k=\bar N\) or \(k=2\bar N-N+1\). The possible shifts in \eqref{eqn:good-box-choice} change these distances only by an absolute constant.

We next prove a preliminary bound on the norm of the Green's function on the whole interval \([\bar N,2\bar N]\): for \(x\in\Omega_3\), for spectral-a.e. \(E\in I\), and for all sufficiently large \(N\), depending on \(x,E,I,\eps\), one has
\begin{align}
    \bigl\|G_{[\bar N,2\bar N]}(x;E)\bigr\|
    \le
    \exp(21\eps c_I N).
    \label{eqn:global-green-norm-bound}
\end{align}
To prove this bound, we first apply the resolvent identity at
\[
    z=E+i\eta,\qquad \eta>0,
\]
so that all finite-volume resolvents are well defined. Fix \(x\in\Omega_3\),  spectral-a.e. \(E\in I\), and a sufficiently large \(N\). Since the local Green's function estimate \eqref{eqn:local-green-good-box} holds at the real energy \(E\), the corresponding local resolvents are finite at \(E\). Hence, by finite-dimensional continuity, and since only finitely many good intervals and matrix entries are used at this fixed scale, there exists
\[
    \eta_0=\eta_0(x,E,N,I,\eps)>0
\]
such that, for every \(0<\eta<\eta_0\), every good interval \(\widetilde\Lambda\) used below, and all \(r,s\in\widetilde\Lambda\),
\begin{align}
\bigl|G_{\widetilde \Lambda}(x;z)(r,s)\bigr|
&\le
\exp\!\Bigl(
-|r-s|L(E)+16\eps c_I N
\Bigr),
\label{eqn:local-green-good-box-z}
\end{align}
where the harmless loss has been absorbed into the coefficient of \(\eps c_I N\). In particular,
\begin{align}
    \bigl|G_{\widetilde\Lambda}(x;z)(r,s)\bigr|
    \le
    e^{16\eps c_I N}.
    \label{eqn:local-crude-bound-z}
\end{align}

Let \(n,m\in[\bar N,2\bar N]\). Choose a good interval
\(\widetilde\Lambda=[\ell,r]\subset[\bar N,2\bar N]\) associated with \(n\) as above and satisfying \eqref{eqn:good-box-boundary-distance}. 
Applying the resolvent identity to \(\widetilde\Lambda\subset[\bar N,2\bar N]\) at the energy \(z=E+i\eta\), we have
\begin{align}
G_{[\bar N,2\bar N]}(x;z)(n,m)
&=
\one_{\{m\in\widetilde\Lambda\}}
G_{\widetilde\Lambda}(x;z)(n,m) \notag\\
&\quad
-
a_{\ell}(x)\,
G_{\widetilde\Lambda}(x;z)(n,\ell)\,
G_{[\bar N,2\bar N]}(x;z)(\ell-1,m) \notag\\
&\quad
-
a_{r+1}(x)\,
G_{\widetilde\Lambda}(x;z)(n,r)\,
G_{[\bar N,2\bar N]}(x;z)(r+1,m).
\label{eqn:resolvent-identity-z}
\end{align}
If \(\widetilde\Lambda\) touches one endpoint of \([\bar N,2\bar N]\), then the corresponding exterior boundary term is omitted.

Applying \eqref{eqn:local-green-good-box-z} to the boundary entries of the good box and using the distance estimate \eqref{eqn:good-box-boundary-distance}, we obtain
\begin{align}
    \bigl|G_{\widetilde\Lambda}(x;z)(n,\ell)\bigr|
    +
    \bigl|G_{\widetilde\Lambda}(x;z)(n,r)\bigr|
     \le
    2\exp\!\Bigl(
    -\Big(\frac{N}{2}-3\Big)L(E)+16\eps c_I N
    \Bigr)  
    \le
    2e^{-\frac14 c_I N}.
    \label{eqn:local-boundary-small-z}
\end{align}
Here, as before, the missing boundary term is omitted in the one-sided case. In the second inequality we used \(L(E)\ge c_I\), and we increased \(N_0\) to absorb the fixed endpoint error.

Taking absolute values in \eqref{eqn:resolvent-identity-z}, using \(|a_j(x)|\le a_+\), and applying \eqref{eqn:local-crude-bound-z} together with \eqref{eqn:local-boundary-small-z}, we obtain, uniformly in \(n,m\in[\bar N,2\bar N]\),
\begin{align}
    \bigl|G_{[\bar N,2\bar N]}(x;z)(n,m)\bigr|
    \le
    e^{16\eps c_I N}
    +
    2a_+e^{-\frac14 c_I N}
    \bigl\|G_{[\bar N,2\bar N]}(x;z)\bigr\|.
    \label{eqn:global-entry-bound-prelim-z}
\end{align}
Taking the maximum over \(n,m\) and using
\[
    \bigl\|G_{[\bar N,2\bar N]}(x;z)\bigr\|
    \le
    (2\bar N+1)
    \max_{n,m\in[\bar N,2\bar N]}
    \bigl|G_{[\bar N,2\bar N]}(x;z)(n,m)\bigr|,
\]
we get
\begin{align}
    \bigl\|G_{[\bar N,2\bar N]}(x;z)\bigr\|
    \le
    (2\bar N+1)e^{16\eps c_I N}
    +
    2a_+(2\bar N+1)e^{-\frac14 c_I N}
    \bigl\|G_{[\bar N,2\bar N]}(x;z)\bigr\|. \notag
\end{align}
Since \(\bar N=e^{(\log N)^2}=e^{o(N)}\), after increasing \(N_0\) we have
\[
    2a_+(2\bar N+1)e^{-\frac14 c_I N}\le \frac12, \qquad {\rm and} \qquad 
    2(2\bar N+1)e^{16\eps c_I N}\le e^{21\eps c_I N}.
\]
Therefore, for every \(0<\eta<\eta_0\),
\begin{align}
    \bigl\|G_{[\bar N,2\bar N]}(x;E+i\eta)\bigr\|
    \le
    e^{21\eps c_I N}.
    \label{eqn:global-green-norm-bound-z}
\end{align}
This estimate is uniform for \(0<\eta<\eta_0\). Letting \(\eta\downarrow0\), we obtain
\begin{align}
    \bigl\|G_{[\bar N,2\bar N]}(x;E)\bigr\|
    \le
    e^{21\eps c_I N}. \notag 
\end{align}
This proves \eqref{eqn:global-green-norm-bound}. In the remaining applications of the resolvent identity, we use the same \(E+i\eta\) limiting argument implicitly and write the Green's functions directly at the real energy \(E\).


We now prove the off-diagonal decay. Let \(n,m\in[\bar N,2\bar N]\) satisfy
\[
    d:=|n-m|\ge \frac{\bar N}{2}.
\]
We assume \(n<m\); the case \(m<n\) is identical.

At each step below, whenever the current point \(y\) satisfies \(|y-m|\ge N+10\), we choose a good interval
\(\widetilde\Lambda_y\subset[\bar N,2\bar N]\) containing \(y\), not containing \(m\), and such that each relevant boundary point of \(\widetilde\Lambda_y\) is at distance at least \(N/2-3\) from \(y\). Using the same \(E+i\eta\) limiting argument as above, we apply the resolvent identity at the real energy \(E\). Since \(m\notin\widetilde\Lambda_y\), the interior term vanishes. Applying \eqref{eqn:local-boundary-small-z} at the real energy, and using \(|a_j(x)|\le a_+\), we obtain
\begin{align}
    \bigl|G_{[\bar N,2\bar N]}(x;E)(y,m)\bigr|
    \le
    2a_+
    \exp\!\Bigl(
    -\Big(\frac12-17\eps\Big)L(E)N
    \Bigr)
    \max_{y'\in\partial\widetilde\Lambda_y}
    \bigl|G_{[\bar N,2\bar N]}(x;E)(y',m)\bigr|.
    \label{eqn:one-step-patching}
\end{align}
Here \(\partial\widetilde\Lambda_y\) denotes the relevant boundary points of \(\widetilde\Lambda_y\) inside \([\bar N,2\bar N]\); in the one-sided case, there is only one such term.

We iterate \eqref{eqn:one-step-patching}
\[
    t:=\left\lfloor\frac{2(1-2\eps)d}{N}\right\rfloor
\]
times. Every terminal point \(y_t\) arising after these iterations satisfies
\[
    |y_t-m|
    \ge
    d-t\left(\frac{N}{2}+3\right).
\]
Since \(d\ge\bar N/2\) and \(\bar N\gg N\), after increasing \(N_0\) if necessary, we have
\[
    |y_t-m|\ge \eps d\ge N+10.
\]
Thus, the iteration is valid for at least \(t\) steps. Using \eqref{eqn:global-green-norm-bound} to control the Green's-function factors remaining after the final iteration, we obtain
\begin{align}
    \bigl|G_{[\bar N,2\bar N]}(x;E)(n,m)\bigr|
    &\le
    \left[
        2a_+
        \exp\!\Bigl(
            -\Bigl(\frac12-17\eps\Bigr)L(E)N
        \Bigr)
    \right]^t
    e^{21\eps c_I N}.
    \label{eqn:after-t-patching}
\end{align}

The intuition behind the choice of \(t\) is as follows. Each application of the resolvent identity moves the first argument by roughly \(N/2\), while the factor appearing at each step in \eqref{eqn:after-t-patching} is of order
\[
    \exp\!\Bigl(-\bigl(\frac12+o(1)\bigr)L(E)N\Bigr).
\]
Thus, after about \(t\sim 2|n-m|/N\) iterations, one expects a total decay of order
\[
    \left[
    \exp\!\Bigl(-\bigl(\frac12+o(1)\bigr)L(E)N\Bigr)
    \right]^t
    \sim
    \exp\bigl(-(1+o(1))L(E)|n-m|\bigr).
\]
We now make this computation precise. For \(N\) sufficiently large,
\[
    \frac{2(1-3\eps)d}{N}\le t\le \frac{2d}{N}.
\]
Therefore, \eqref{eqn:after-t-patching} gives
\begin{align}
\log \bigl|G_{[\bar N,2\bar N]}(x;E)(n,m)\bigr|
&\le
\frac{2d}{N}\log(2a_+)
-
2\Big(\frac12-17\eps\Big)(1-3\eps)L(E)d
+
21\eps c_I N.
\notag
\end{align}
After increasing \(N_0\), the first and last terms are both bounded by \(\eps L(E)d\), since \(d\gg N\) and \(L(E)\ge c_I\). Finally, for \(0<\eps<1/100\),
\[
    2\Big(\frac12-17\eps\Big)(1-3\eps)
    \ge
    1-42\eps.
\]
Hence
\begin{align}
    \bigl|G_{[\bar N,2\bar N]}(x;E)(n,m)\bigr|
    \le
    \exp\Bigl(-(1-44\eps)L(E)|n-m|\Bigr).\notag 
\end{align}
This proves the Green's function decay estimate in the form needed below. Equivalently, after replacing \(\eps\) by a sufficiently small multiple of the tolerance in the statement, this gives \eqref{eqn:green-decay}.
\end{proof}

We now use \eqref{eqn:green-decay} to complete the proof of Anderson localization. The argument is the standard final step: the Green's function expansion controls a generalized eigenfunction inside a large interval by its boundary values, and the decay in \eqref{eqn:green-decay} dominates the polynomial growth allowed by Shnol's theorem.

\begin{proof}[Proof of Theorem~\ref{thm:AL}]
By Lemma~\ref{lem:AL-compact-reduction}, it suffices to prove the statement on an arbitrary compact interval \(I\subsetneq(0,E_0)\). Fix such an interval \(I\), and let \(\Omega\) be the full-measure set given by Theorem~\ref{thm:ldt-green1}.

Fix \(x\in\Omega\). For spectral-a.e. \(E\in I\), let \(\xi=\{\xi_n\}_{n\ge -1}\) be a polynomially bounded generalized eigenfunction satisfying \eqref{eqn:shnol-init}--\eqref{eqn:shnol-poly}. We show that \(\xi\) decays exponentially.

Let \(n\) be sufficiently large. Choose \(N\) so that, with
\[
    \bar N=\lfloor e^{(\log N)^2}\rfloor,
\]
one has
\[
    n\in[\bar N,2\bar N],
    \qquad
    \dist\bigl(n,\partial[\bar N,2\bar N]\bigr)\ge \frac{\bar N}{3}.
\]
For instance, one may choose \(\bar N\sim \frac23 n\). Restricting the generalized eigenvalue equation \((H(x)-E)\xi=0\) to the interval \([\bar N,2\bar N]\), we obtain, up to signs that are immaterial for the estimates,
\begin{align}
    \xi_n
    =
    G_{[\bar N,2\bar N]}(x;E)(n,\bar N)a_{\bar N}\xi_{\bar N-1}
    +
    G_{[\bar N,2\bar N]}(x;E)(n,2\bar N)a_{2\bar N+1}\xi_{2\bar N+1}.
    \label{eqn:xi-green-expansion}
\end{align}
Using \(|a_j(x)|\le a_+\), the Green's function estimate \eqref{eqn:green-decay}, and the polynomial bound \eqref{eqn:shnol-poly}, we get
\begin{align}
    |\xi_n|
    &\le
    C_E(1+2\bar N)
    \exp\!\Bigl(
    -(1-\varepsilon)L(E)\dist\bigl(n,\partial[\bar N,2\bar N]\bigr)
    \Bigr) \notag\\
    &\le
    C_E(1+2\bar N)
    \exp\!\Bigl(
    -(1-\varepsilon)L(E)\frac{\bar N}{3}
    \Bigr).
    \notag
\end{align}
Since \(\bar N\sim \frac23 n\), the polynomial prefactor is absorbed into the exponential for all sufficiently large \(n\). Taking \(\varepsilon>0\) sufficiently small in Theorem~\ref{thm:ldt-green1}, we obtain
\begin{align}
    |\xi_n|
    \le
    \exp\!\Bigl(-\frac14 L(E)n\Bigr)
    \notag 
\end{align}
for all sufficiently large \(n\). Thus the polynomially bounded generalized eigenfunction \(\xi\) is exponentially decaying.

It follows that, for spectral-a.e. \(E\in I\), the corresponding generalized eigenfunction is an \(\ell^2(\Z_{\ge0})\) eigenfunction with exponential decay. Hence the spectral measure on \(I\) is supported on exponentially decaying eigenfunctions, and \(H(x)\) has pure-point spectrum in \(I\). Since \(I\subsetneq(0,E_0)\) was arbitrary, Lemma~\ref{lem:AL-compact-reduction} gives pure-point spectrum with exponentially decaying eigenfunctions on \([0,E_0]\). This proves Theorem~\ref{thm:AL}.
\end{proof}


\vspace{2cm}

\noindent\textbf{Acknowledgments.}
\phantomsection
\addcontentsline{toc}{section}{Acknowledgments}
W. Wang was supported by the National Key R\&D Program of China under grants 2025YFA1016600 and 2025YFA1016601.  L. Li was supported by AMS-Simons Travel Grant 2024-2026. S. Zhang was supported by the NSF grant DMS-2418611.  We thank Jake Fillman for many useful discussions and suggestions, especially for kindly sharing the proof of Theorem~\ref{thm:ess-spectrum} with us.


{
  \bigskip
  \vskip 0.08in \noindent --------------------------------------

\footnotesize
\medskip

L.~Li, {Department of Mathematics, 
Texas A\&M University, 155 Ireland Street,
College Station, TX 77843}\par\nopagebreak
    \textit{E-mail address}:  \href{mailto:longli@tamu.edu }{longli@tamu.edu }

\vskip 0.4cm

 W. ~Wang, {SKLMS, Academy of Mathematics and Systems Science, Chinese Academy of Sciences, Beijing 100190, China}\par\nopagebreak
  \textit{E-mail address}: \href{mailto:ww@lsec.cc.ac.cn}{ww@lsec.cc.ac.cn}
  
\vskip 0.4cm

S.~Zhang, {Department of Mathematics and Statistics, University of Massachusetts Lowell, 
Southwick Hall, 
11 University Ave.
Lowell, MA 01854
 }\par\nopagebreak
  \textit{E-mail address}: \href{mailto:shiwen\_zhang@uml.edu}{shiwen\_zhang@uml.edu}
}

\end{document}